\documentclass[11 pt]{amsart}

\usepackage{latexsym}
\usepackage{amssymb, amsmath}
\usepackage{amsthm}

\usepackage{geometry}
\usepackage{hyperref}

\numberwithin{equation}{section}

\newtheorem{theorem}{Theorem}[section]
\newtheorem{lemma}[theorem]{Lemma}
\newtheorem{cor}[theorem]{Corollary}
\newtheorem{sublem}[theorem]{Sublemma}
\newtheorem{proposition}[theorem]{Proposition}

\newtheorem{remark}[theorem]{Remark}

\def\beq{\begin{equation}}
\def\eeq{\end{equation}}

\def\B{\mathcal{B}}
\def\cC{\mathcal{C}}
\newcommand{\C}{\mathcal{C}}

\def\cF{\mathcal{F}}
\newcommand{\F}{\mathcal{F}}
\def\G{\mathcal{G}}

\def\I{\mathcal{I}}
\def\eps{\varepsilon}
\def\cK{\mathcal{K}}
\def\K{\mathcal{K}}
\newcommand{\Lp}{\mathcal{L}}
\def\cM{\mathcal{M}}

\def\cO{\mathcal{O}}

\def\cR{\mathcal{R}}
\def\cS{\mathcal{S}}
\newcommand{\Si}{\mathcal{S}}
\def\cT{\mathcal{T}}

\def\cV{\mathcal{V}}
\def\V{\mathcal{V}}
\def\cW{\mathcal{W}}
\newcommand{\W}{\mathcal{W}}

\def\Z{\mathcal{Z}}

\newcommand{\Q}{\mathcal{Q}}

\newcommand{\E}{\mathcal{E}}
\newcommand{\N}{\mathbb{N}}

\newcommand{\Ho}{\mathbb{H}}\newcommand{\bH}{\mathbb{H}}

\newcommand{\tQ}{\tilde{Q}}
\newcommand{\tGamma}{\tilde{\Gamma}}
\newcommand{\tu}{\tilde{u}}
\newcommand{\p}{\mathbf{p}}
\newcommand{\bq}{\mathbf{q}}
\newcommand{\bv}{\mathbf{v}}
\newcommand{\bx}{\mathbf{x}}
\newcommand{\bF}{\mathbf{F}}
\newcommand{\bG}{\mathbf{G}}

\newcommand{\vf}{\varphi}
\newcommand{\po}{\psi_1}
\newcommand{\pt}{\psi_2}
\newcommand{\bp}{\overline{\psi}}
\newcommand{\R}{\mathbb{R}}
\newcommand{\ve}{\varepsilon}

\newcommand{\ob}{\overline{\omega}}
\newcommand{\ds}{\displaystyle}
\newcommand{\vu}{\vec{u}}
\newcommand{\bmu}{\overline{\mu}}

\begin{document}

\title{A functional analytic approach to perturbations of \\the Lorentz gas}
\author{Mark F. Demers \and Hong-Kun Zhang}
\address{Mark F. Demers, Department of Mathematics and Computer Science, Fairfield University, Fairfield CT 06824, USA.}
\email{mdemers@fairfield.edu}
\address{Hong-Kun Zhang, Department of Mathematics and Statistics, UMass Amherst, MA 01003, USA.}
\email{hongkun@math.umass.edu}

\thanks{M.\ D.\ is partially supported by NSF Grant DMS-1101572. H.-K.\ Z.\ is partially supported by NSF CAREER Grant DMS-1151762.}

\date{\today}

\begin{abstract}
We present a functional analytic framework based on the spectrum of the
transfer operator to study billiard maps associated
with perturbations of the periodic Lorentz gas.  We show that recently constructed Banach spaces
for the billiard map of the classical Lorentz gas are flexible enough to admit a wide variety of perturbations,
including: movements and deformations of scatterers; billiards subject to external forces;
nonelastic reflections with kicks and slips at the boundaries of the scatterers; 
and random perturbations comprised
of these and possibly other classes of maps.  The spectra and spectral projections
of the transfer operators are shown
to vary continuously with such perturbations so that the spectral gap enjoyed by the
classical billiard persists and important limit theorems follow.
\end{abstract}

\maketitle

\section{Introduction}
\label{intro}

The Lorentz gas is known to enjoy strong ergodic
properties:  both the continuous time dynamics and the billiard maps
are completely hyperbolic, ergodic, K-mixing and Bernoulli (see
\cite{Sin70, GO, SC87, CH96} and the references therein).
Young \cite{Y98} proved exponential
decay of correlations for billiard maps corresponding to the finite horizon periodic Lorentz gas
using Markov extensions; this technique was subsequently extended to other dispersing
billiards \cite{C99} and used to obtain important limit theorems such as local large deviation
estimates and almost-sure invariance principles \cite{melbourne nicol, melbourne nicol 2, rey}.

In this setting, it is natural to ask how the statistical properties of dispersing billiard maps vary
with the shape and position of the scatterers.  Alternatively, one may change the billiard
dynamics by introducing an external force between collisions or by considering nonelastic
reflections at the boundaries.  Such perturbed dynamics lead to nonequilibrium billiards
whose invariant measures are singular with respect to Lebesgue measure.

One of the first nonequilibrium physical models that was studied rigorously is the
periodic Lorentz gas with a small constant electrical field \cite{CELSa, CELSb} and the
well-known Ohm's law was proved for that case.   More general external forces
were handled in \cite{Ch01, Ch08, CD09} and billiards with kicks at reflections have been studied
in \cite{mark, Z09}.
Recently,  Chernov and Dolgopyat \cite{CD} used coupling methods to study the motion
of a point particle colliding with a moving scatterer.  Locally perturbed periodic
rearrangements of scatterers have also been the subject of recent studies \cite{DSV}.
Despite such successes, the study of perturbations of billiards has thus far been handled
on a case by case basis, with methods adapted and developed for each specific type of
perturbation considered.

In this paper, we propose a unified framework in which
to study a large class of perturbations of dispersing billiards.
This framework is based on the spectral
analysis of the transfer operator associated with the billiard map
and uses the recent work \cite{demers zhang}
which successfully constructed Banach spaces on which the transfer operator for the
classical periodic Lorentz gas has a spectral gap.

We first present abstract conditions under which we have uniform control of
spectral data for a given class of perturbed maps.  We then prove that four
broad classes of perturbations of billiards fit within this framework, namely:
\begin{itemize}
  \item[(i)]  Tables with shifted, rotated or deformed scatterers;
  \item[(ii)]  Billiards under small external forces which bend trajectories during flight;
  \item[(iii)]  Billiards with kicks or twists at reflections, including slips along the disk;
  \item[(iv)]  Random perturbations comprised of maps with uniform properties (including
  any of the above classes, or a combination of them).
\end{itemize}
In particular, the results on random perturbations are a version of time-dependent billiards,
in which scatterers are allowed to change positions between collisions.
The fact that our main theorems, \ref{thm:uniform} and \ref{thm:close}, are proved in
an abstract setting will facilitate the application of this framework to other classes
of perturbations as they arise in future works.

The present functional analytic approach uses the Banach spaces 
constructed in \cite{demers zhang}
as well as the perturbative framework of Keller and Liverani \cite{keller liverani}
to prove that the spectral data and spectral projectors, including invariant measures, rates
of decay of correlations, variance in the central limit theorem, etc, vary H\"older
continuously for the classes of perturbations mentioned above
(see \cite{baladi book, liv} for expositions of this approach).
In addition, this
approach yields new results for the perturbed billiard maps in terms of local
limit theorems, in particular giving new information about the evolution of noninvariant
measures in the context of these limit theorems.  For example, applying
Corollary~\ref{cor:limit theorems} to billiards under external forces and kicks, we obtain
a local large deviation estimate with a rate function that is the same for all probability measures
in our Banach space.  This implies in particular that Lebesgue measure and the
singular SRB measure for the perturbed billiard have the same large deviation
rate function.

The paper is organized as follows.  In Section 2, we describe our abstract framework,
state precisely the applications which serve as our model perturbations and
formulate our main results.
In Section 3, we lay out our common approach under the general conditions
{\bf (H1)}-{\bf (H5)} which guarantee the required uniform Lasota-Yorke inequalities
for Theorem~\ref{thm:uniform}, proved
in Section~\ref{uniform};
we also formulate conditions  {\bf (C1)}-{\bf (C4)} to verify that a perturbation is small in the sense of our Banach spaces for Theorem~\ref{thm:close}, proved in Section~\ref{close}.
The investigations of the concrete models are provided in Sections~\ref{perts} and \ref{kick}.


\section{Setting and Results}
\label{results}

In this section, we describe the abstract framework into which we will place our
perturbations and formulate precisely the classes of concrete deterministic perturbations to which
our results apply.  We also formulate a class of random perturbations with maps drawn from
any mixture of the deterministic perturbations described below.
We postpone until Section~\ref{common} a precise description of
the Banach spaces and the formal requirements on the abstract class of maps $\mathcal{F}$.


\subsection{Perturbative framework}
\label{framework}

We recall here the perturbative framework of Keller and Liverani \cite{keller liverani}.
Suppose there exist two Banach spaces $(\B, \| \cdot \|_\B)$ and $(\B_w, | \cdot |_w )$
with the unit ball of $\B$ compactly embedded in $\B_w$, $| \cdot |_w \leq \| \cdot \|_\B$,
and a family of bounded linear operators $\{ \Lp_\ve \}_{\ve \ge 0}$ defined on both
$\B_w$ and $\B$ such that
the following holds.\footnote{The results of \cite{keller liverani} hold in a more
general setting, but we only state the version we need for our purposes.}  There exist constants
$C, \eta >0$ and $\sigma < 1$ such that for
all $\ve \ge 0$ and $n \geq 0$,
\begin{equation}
\label{eq:pert LY}
\begin{split}
| \Lp_\ve^n h|_w & \leq C \eta^n |h|_w \qquad \mbox{for all $h \in \B_w$}, \\
\| \Lp_\ve^n h \|_\B & \le C \sigma^n \| h \|_\B + C \eta^n |h|_w \qquad \mbox{for all $h \in \B$}.
\end{split}
\end{equation}
If $\sigma < \eta$, the operators $\Lp_\ve$ are quasi-compact with essential
spectral radius bounded by $\sigma$ and spectral radius at most $\eta$ (see for example
\cite{baladi book}).
Suppose further that
\begin{equation}
\label{eq:pert small}
||| \Lp_\ve - \Lp_0 ||| := \sup \{ |\Lp_\ve h - \Lp_0 h |_w : h \in \B, \| h \|_\B \le 1 \} \le \rho(\ve),
\end{equation}
where $\rho(\ve)$ is a non-increasing upper semicontinuous function satisfying
$\lim_{\ve \to 0} \rho(\ve) = 0$.

The main result of \cite{keller liverani} is the following.
Let sp$(\Lp_0)$ denote the spectrum of $\Lp_0$.  For any $\sigma_1 > \sigma$, by quasi-compactness,
sp$(\Lp_0) \cap \{ z \in \mathbb{C} : |z| \geq \sigma_1 \}$ consists of finitely
many eigenvalues $\varrho_1, \ldots, \varrho_k$ of finite multiplicity.
Thus there exists $t_* >0$ and we may choose $\sigma_1$ such that
$| \varrho_i - \varrho_j | > t_*$ for $i \ne j$ and dist$(\mbox{sp}(\Lp_0), \{ |z| = \sigma_1 \}) > t_*$.
For $t < t_*$ and $\ve \geq 0$, define the spectral projections,
\[
\begin{split}
\Pi_\ve^{(j)} & := \frac{1}{2\pi i} \int_{|z - \varrho_j|=t} (z- \Lp_\ve)^{-1} \, dz \qquad \mbox{and} \\
\Pi_\ve^{(\sigma_1)} & := \frac{1}{2\pi i} \int_{|z|=\sigma_1} (z- \Lp_\ve)^{-1} \, dz .
\end{split}
\]

\begin{theorem} (\cite{keller liverani})
\label{thm:kl}
Assume that \eqref{eq:pert LY} and \eqref{eq:pert small} hold.
Then for each $t \le t_*$ and $s < 1 - \frac{\log \sigma_1}{\log \sigma}$, there
exist $\ve_1, C >0$ such that for any $0 \le \ve < \ve_1$, the spectral projections
$\Pi_\ve^{(j)}$ and $\Pi_\ve^{(\sigma_1)}$ are well defined
and satisfy, for each $j = 1, \ldots k$,
\begin{itemize}
 \item[(1)] $||| \Pi_\ve^{(j)} - \Pi_0^{(j)} ||| \le C \rho(\ve)^s$ and
 $||| \Pi_\ve^{(\sigma_1)} - \Pi_0^{(\sigma_1)} ||| \le C \rho(\ve)^s$ ;
 \item[(2)] rank$(\Pi_\ve^{(j)}) = \mbox{rank}(\Pi_0^{(j)})$;
 \item[(3)] $\| \Lp_\ve^n \Pi_\ve^{(\sigma_1)} \| \le C \sigma_1^n$, for all $n\ge 0$.
\end{itemize}
\end{theorem}

We say an operator $\Lp$ has a spectral gap if $\Lp$ has a simple eigenvalue of maximum
modulus and
all other eigenvalues have strictly smaller modulus.
The above theorem implies in particular that if
$\Lp_0$ has a spectral gap, then so does $\Lp_\ve$ for $\ve$ sufficiently small.
In addition, the related statistical properties (for instance, invariant measures, rates of decay of correlations,
variance of the Central Limit Theorem) are stable and vary H\"older continuously
as a function of $\rho(\ve)$.
This is the framework into which we will place our perturbations of the Lorentz gas.


\subsection{An abstract result for a class of maps with uniform properties}
\label{abstract}

We begin by fixing the phase space $M$ of a billiard map
associated with a periodic Lorentz gas.  That is, we place finitely many (disjoint)
scatterers $\Gamma_i$, $i=1, \ldots d$,
on $\mathbb{T}^2$ which have $\C^3$ boundaries with strictly positive curvature.
The classical billiard flow on the table
$\mathbb{T}^2 \setminus \cup_i \{ \mbox{interior } \Gamma_i \}$
is induced by a particle traveling at unit speed and undergoing elastic collisions at the
boundaries.  In what follows, we also consider particles whose motion between collisions
follows slightly curved trajectories (due to external forces)
as well as certain types of collisions which do not obey the
usual law of reflection.

In all cases, the billiard map associated with the flow is the Poincar\'e map
corresponding to collisions with the scatterers.
Its phase space is $M = \cup_{i=1}^d I_i \times [-\pi/2, \pi/2]$,
where $\ell(I_i) = |\partial \Gamma_i|$, i.e.\ the length of $I_i$ equals
the arclength of $\partial \Gamma_i$, $i=1, \ldots d$.
$M$ is parametrized by the canonical coordinates $(r, \vf)$ where $r$ represents the
arclength parameter on the boundaries of the scatterers (oriented clockwise) and
$\vf$ represents the angle an outgoing (postcollisional) trajectory makes with the
unit normal to the boundary at the point of collision.

The phase space $M$ and
coordinates so defined are fixed for
all classes of perturbations we consider; however, the configuration space (the billiard
table on which the particles flow) and the laws which govern the motion of the particles
may vary as long as all variations give rise to the
same phase space $M$, i.e.\ the number of $\Gamma_i$ and the arclengths
of their boundaries do
not change.
See Remark~\ref{rem:arclength} for a way to relax this requirement on the arclength.
For any $x = (r,\vf) \in M$, we define $\tau(x)$ to be the first collision of the
trajectory starting at $x$ under the billiard flow.  The billiard map is defined wherever
$\tau(x) < \infty$.  We say that the billiard has finite horizon if there is an upper bound
on the function $\tau$.  Otherwise, we say the billiard has infinite horizon.
Notice that the function $\tau$ depends on the placement of the scatterers in $\mathbb{T}^2$,
while $M$ is independent of their placement.

\bigskip
We assume there exists a class of maps $\F$ on $M$ satisfying properties
{\bf (H1)}-{\bf (H5)} of Section~\ref{class of maps} with uniform constants.
For each $T \in \F$, in Section~\ref{transfer} we
define the transfer operator $\Lp_T$ associated with
$T$ on an appropriate class of distributions $h$ by
\[
\Lp_Th(\psi) = h(\psi \circ T), \;\; \; \mbox{for suitable test functions } \psi.
\]
In Section~\ref{norms}, we define Banach spaces of distributions
$( \B, \| \cdot \|_\B )$ and $(\B_w, | \cdot |_w )$, preserved under the action
of $\Lp_T$, $T \in \F$, such that the unit ball of $\B$ is compactly embedded
in $\B_w$.

\begin{theorem}
\label{thm:uniform}
Fix $M$ as above and suppose there exists a class of maps $\F$ satisfying
{\bf (H1)}-{\bf (H5)} of Section~\ref{class of maps}.  Then $\Lp_T$ is well defined
as a bounded linear operator on $\B$ for each $T \in \F$.  In addition,
there exist $C>0$, $\sigma <1$
such that for any $T \in \F$ and $n \geq 0$,
\begin{equation}
\begin{split}
\label{eq:uniform LY}
| \Lp_T^n h|_w & \leq C \eta^n |h|_w \qquad \mbox{for all $h \in \B_w$},  \\
\| \Lp_T^n h \|_\B & \le C \sigma^n \| h \|_\B + C \eta^n |h|_w \qquad \mbox{for all $h \in \B$},
\end{split}
\end{equation}
where $\eta \ge 1$ is from {\bf (H5)}.
This, plus the compactness of $\B$ in $\B_w$, implies that all the operators
$\Lp_T$, $T \in \F$, are quasi-compact with essential
spectral radius bounded by $\sigma$:  i.e., outside of any disk of radius greater than $\sigma$,
their spectra contain finitely many eigenvalues of finite multiplicity.  Moreover, for each $T \in \F$,
\begin{itemize}
  \item[(i)] the spectral radius of $\Lp_T$ is  1 and the elements of the peripheral spectrum
are measures absolutely continuous with respect to $\bmu := \lim_{n \to \infty} \frac 1n \sum_{i=0}^{n-1} \Lp_T^i 1$;
  \item[(ii)] an ergodic, invariant probability measure $\nu$ for $T$ is in $\B$ if and only if
  $\nu$ is a physical measure\footnote{Recall that a physical measure for $T$ is
  an ergodic, invariant probability measure $\nu$ such that
  $\lim_{n \to \infty} \frac 1n \sum_{i=1}^{n-1} f(T^i x) = \int f \, d\nu$ for a positive Lebesgue
  measure set of $x \in M$.}
  for $T$;
  \item[(iii)]   there exist a finite number of $q_\ell \in \mathbb{N}$ such that the spectrum of
  $\Lp$ on the unit circle is
  $\cup_\ell \{ e^{2\pi i \frac{k}{q_\ell}} : 0 \le k < q_\ell, \, k \in \mathbb{N} \}$.
  The peripheral spectrum contains no Jordan blocks.
  \item[(iv)]   Let $\Si_{\pm n, \ve}^{\bH}$ denote the $\ve$-neighborhood of $\Si_{\pm n}^\bH$,
  the singularity set for $T^{\pm n}$ (with homogeneity strips).
    Then
  for each $\nu$ in the peripheral spectrum and $n \in \mathbb{N}$,
  we have $\nu(\Si_{\pm n, \ve}^\bH) \le C_n \ve^\alpha$,
  for some constants $C_n >0$.
  \item[(v)]  If $(T \bmu)$ is ergodic, then 1 is a simple eigenvalue. If $(T^n, \bmu)$ is ergodic
  for all $n \in \N$, then 1 is the only eigenvalue of modulus 1, $(T, \bmu)$ is mixing and
  enjoys exponential decay of correlations for H\"older observables.
\end{itemize}
\end{theorem}

Theorem~\ref{thm:uniform} is proved in Section~\ref{uniform}.
In Section~\ref{distance}, we define a distance $d_\F(\cdot , \cdot)$ between maps in $\F$.
Our next result shows that this distance controls the size of perturbations in the
spectra of the associated transfer operators.

\begin{theorem}
\label{thm:close}
Let $\beta >0$ be from the definition of $(\B, \| \cdot \|_\B)$ in Section~\ref{norms}.
There exists $C >0$ such that if $T_1, T_2 \in \F$ with $d_\F(T_1, T_2) \le \ve$, then
\[
||| \Lp_{T_1} - \Lp_{T_2} ||| \le C \ve^{\beta/2},
\; \; \; \mbox{where $||| \cdot |||$ is from \eqref{eq:pert small}.}
\]
\end{theorem}

We prove Theorem~\ref{thm:close} in Section~\ref{close}.
According to Theorem~\ref{thm:kl}, an immediate consequence of Theorems~\ref{thm:uniform}
and \ref{thm:close} is the following.

\begin{cor}
\label{cor:limit theorems}
If $T_0 \in \F$ has a spectral gap, then all $T \in X_\ve(T_0) = \{ T \in \F : d_\F(T,T_0) < \ve$
have a spectral gap for $\ve$ sufficiently small.  In particular, the maps in $X_\ve$ enjoy
the following limit theorems (among others), which follow from the existence of a spectral
gap.

Fix $T \in F$ with a spectral gap. Let
$\gamma = \max \{p , 2\beta + \delta \}$ for some $\delta>0$, where $p$ and $\beta$ are from Sect.~\ref{norms},
and let $g \in \C^\gamma(M)$.  Define $S_ng = \sum_{k=0}^{n-1} g \circ T^k$.
\begin{itemize}
  \item[(a)] (Local large deviation estimate)  For any (not necessarily invariant)
  probability measure $\nu \in \B$,
  \[
  \lim_{\ve \to 0} \lim_{n \to \infty} \frac 1n \log \nu \Big( x \in M : \frac 1n S_ng(x) \in [t-\ve, t+\ve] \Big)
  = - I(t)
  \]
  where the rate function $I(t)$ is independent of $\nu \in \B$ (but may depend on $T$), and
  $t$ is in a neighborhood of the mean $\bmu(g)$.
 \item[(b)]  (Almost-sure invariance principle).  Suppose $\bmu(g)=0$ and distribute
 $(g \circ T^j)_{j \in \N}$ according to a probability measure $\nu \in \B$.  Then there
 exist $\lambda >0$, a probability space $\Omega$ with random variables $\{ X_n \}$
 satisfying $S_ng \stackrel{\mbox{dist.}}{=} X_n$, and a Brownian motion $W$ with
 variance $\varsigma^2\geq 0$ such that
 \[
 X_n = W(n) + \mathcal{O}(n^{1/2-\lambda}) \qquad \mbox{as $n\to \infty$ almost-surely
 in $\Omega$}.
 \]
\end{itemize}
\end{cor}

\noindent
The proof of the corollary is the same as that of \cite[Theorem 2.6]{demers zhang} and will
not be repeated here.


\subsection{Applications to concrete classes of deterministic perturbations}
\label{concrete}

In this section we describe precisely several types of perturbations of the Lorentz gas
which fall under the abstract framework we have outlined above.
In light of Theorems~\ref{thm:uniform} and \ref{thm:close}, it suffices to check two things
for each class of perturbations we will introduce:
(1) {\bf (H1)}-{\bf (H5)} hold uniformly in each class; (2) the perturbations are small
in the sense of the distance $d_\F( \cdot , \cdot)$.

\bigskip
\noindent
{\bf A.  Movements and Deformations of Scatterers.} \\
We fix the phase space $M = \cup_{i=1}^d I_i \times [-\frac{\pi}{2}, \frac{\pi}{2}]$
associated with a billiard map corresponding to a periodic Lorentz gas with $d$ scatterers
as described above.  We assume that the billiard particle
moves along straight lines and undergoes elastic reflections at the boundaries.

For given $I_1, \ldots, I_d$, we use the notation $Q = Q(\{ \Gamma_i \}_{i=1}^d ; \{ I_i \}_{i=1}^d)$ to denote the configuration of
scatterers $\Gamma_1, \ldots, \Gamma_d$ placed on the billiard table
such that $|\partial \Gamma_i| = \ell(I_i)$, $i = 1, \ldots, d$.
Since we have fixed $I_1, \ldots, I_d$, $M$ remains the same
for all configurations $Q$ that we consider.
For each such configuration, we define
\[
\tau_{\min}(Q) =
\inf \{ \tau(x) : \tau(x) \mbox{ is defined for the configuration } Q \} .
\]
Similarly, $\K_{\min}(Q)$ and $\K_{\max}(Q)$ denote the minimum and maximum curvatures
respectively of the $\Gamma_i$ in the configuration $Q$.  The constant
$E_{\max}(Q)$ denotes the maximum $C^3$ norm of the $\partial \Gamma_i$ in $Q$.

For each fixed $\tau_*, \K_*, E_* >0$, define $\Q_1(\tau_*, \K_*, E_*)$ to be the collection
of all configurations $Q$ such that $\tau_{\min}(Q) \ge \tau_*$,
$\K_* \le \K_{\min}(Q) \le K_{\max}(Q) \le \K_*^{-1}$, and $E_{\max}(Q) \le E_*$.
The horizon for $Q \in \Q_1(\tau_*, \K_*, E_*)$ is allowed to be finite or infinite.
Let $\F_1(\tau_*, \K_*, E_*)$ be the corresponding set of billiard maps induced by the
configurations in $\Q_1$.  It follows from \cite{demers zhang} that
for any $T \in \F_1(\tau_*, \K_*, E_*)$, $\Lp_T$ has a spectral gap in $\B$. 
We prove the following theorems in Section~\ref{perts}.

\begin{theorem}
\label{thm:F1}
Fix $I_1, \ldots, I_d$ and let $\tau_*, \K_*, E_* >0$.  The family
$\F_1(\tau_*, \K_*, E_*)$ satisfies {\bf (H1)}-{\bf (H5)} with uniform constants
depending only on $\tau_*$, $\K_*$ and $E_*$.  As a consequence
of Theorem~\ref{thm:uniform}, $\Lp_T$ is quasi-compact as an operator on $\B$
for each $T \in \F_1(\tau_*, \K_*, E_*)$ with uniform bounds on its essential
spectral radius.
\end{theorem}

We fix an initial configuration of scatterers $Q_0 \in \Q_1(\tau_*, \K_*, E_*)$ and consider
configurations $Q$ which alter each $\partial \Gamma_i$ in $Q_0$ to a curve
$\partial \tilde{\Gamma}_i$ having the same arclength as $\partial \Gamma_i$.
We consider each $\partial \Gamma_i$ as a parametrized curve
$u_i:I_i \to M$ and each $\partial \tilde{\Gamma}_i$ as parametrized by $\tilde{u}_i$.
Define $\Delta(Q, Q_0) = \sum_{i=1}^d |u_i - \tilde{u}_i|_{C^2(I_i,M)}$.

\begin{theorem}
\label{thm:deform}
Choose $\gamma \le  \min \{ \tau_*/2, \K_*/2 \}$ and let $\F_A(Q_0, E_*; \gamma)$
be the set of all billiard maps corresponding to
configurations $Q$ such that $\Delta(Q,Q_0) \le \gamma$ and $E_{\max}(Q) \le E_*$.

Then $\F_A(Q_0, E_*; \gamma) \subset \F_1(\tau_*/2, \K_*/2, E_*)$
and $d_\F(T_1, T_2) \le C|\gamma|^{2/15}$ for any $T_1, T_2 \in \F_A(Q_0, E_*; \gamma)$.
If all $T_i \in \F_A(Q_0, E_*; \gamma)$ have uniformly bounded finite horizon,
then $d_\F(T_1, T_2) \le C|\gamma|^{1/3}$.

As a consequence, the eigenvalues outside a disk of radius $\sigma <1$ and the corresponding
spectral projectors of $\Lp_T$
vary H\"older continuously for all $T \in \F_A(Q_0, E_*; \gamma)$
and all $\gamma$ sufficiently small.
\end{theorem}

\begin{remark}
\label{rem:arclength}
 (a) A remarkable aspect of this result is that it allows us to move configurations from
finite to infinite horizon without interrupting H\"older continuity of the statistical properties such as
the rate of decay of correlations and the variance in the CLT, among others.

\smallskip
\noindent
(b)  The requirement that all deformations of the initial configuration $Q_0$
maintain the same arclength
can be relaxed.  The purpose of this requirement is to define
the corresponding transfer operators on fixed spaces $\B$ and $\B_w$.  If a scatterer $\Gamma_i$
is deformed into $\Gamma_i'$ with a slight change in arclength, we can reparametrize
$\Gamma_i'$ (no longer according to arclength) using the same interval $I_i$ as for
$\Gamma_i$.  This will change the derivatives of maps in the class $\F_B(Q_0, E_*; \gamma)$
slightly, but since properties {\bf (H1)}-{\bf (H5)} have some leeway built into the uniform constants,
for small enough reparametrizations the same properties will hold with slightly weakened constants.
\end{remark}

\smallskip
\noindent
{\bf B.  Billiards Under Small External Forces with Kicks and Slips.} \\
As in part A, we fix $\tau_*, \K_*$ and $E_*$ and choose a fixed
$Q_0 \in \Q_1(\tau_*, \K_*, E_*)$.  In this section, we consider the dynamics of the
billiard map on the table $Q_0$, but subject to external forces both during flight
and at collisions.

Let $\bq=(x,y)$ be the position of a particle in a billiard table $Q_0$ and
$\p$ be the velocity vector. For a $C^2$ stationary external force,
$\mathbf{F}: \mathbb{T}^2 \times \mathbb{R}^2 \to \mathbb R^2$, the perturbed billiard flow	
$\Phi^t$ satisfies the following differential equation between collisions:
\begin{equation}\label{flowf}
    \frac{d \bq}{dt} =\p(t) , \qquad
    \frac{d \p}{dt} = \mathbf{F}(\bq, \p) .
\end{equation}
At collision, the trajectory experiences possibly nonelastic reflections with slipping along
the boundary:
\begin{equation}
\label{reflectiong}
(\bq^+(t_i), \p^+(t_i)) = (\bq^-(t_i), \cR  \p^-(t_i)) +\mathbf G(\bq^-(t_i), \p^-(t_i))
\end{equation}
 where $\cR \p^-(t_i)= \p^-(t_i)+2(n(\bq^-)\cdot \p^{-})n(\bq^-))$ is the usual reflection operator,
 $n(\bq)$ is the unit normal vector to the billiard wall $\partial Q_0$ at
 $\bq$ pointing inside the table $Q_0$, and $\bq^-(t_i), \p^-(t_i)$, $\bq^+(t_i)$ and
 $\p^+(t_i)$ refer to the incoming and
 outgoing position and velocity vectors, respectively.
 $\mathbf G$ is an external force acting on the
 incoming trajectories.
 Note that we allow $\bf G$ to change both the position and velocity of the particle at the moment
 of collision.  The change in velocity can be thought of as a kick or twist while a change in
 position can model a slip along the boundary at collision.

In \cite{Ch01, Ch08}, Chernov considered billiards under small external forces $\mathbf{F}$
with $\mathbf G=0$, and $\mathbf{F}$ to be stationary. In \cite{Z09} a twist force was considered assuming $\mathbf{F}=0$ and $\bG$
depending on and affecting only the velocity, not the position.
Here we consider a combination of these two cases for systems under more general forces $\mathbf{F}$ and $\mathbf{G}$.
We make four assumptions, combining those in \cite{Ch01, Z09}. \\

 \noindent\textbf{(A1)} (\textbf{Invariant space}) \emph{Assume the dynamics preserve a smooth function $\E(\bq,\p)$. Its level surface $\Omega_c:=\E^{-1}(c)$, for any $c>0$, is a compact 3-d manifold such that $\|\p\|>0$ on $\Omega_c$ and for each $\bq\in Q$ and $\p\in S^1$ the ray $\{(\bq, t\p), t>0\}$ intersects the manifold $\Omega_c$ in exactly one point.} \\

Assumption ({\bf{A1}}) specifies an additional integral of motion, so that we only consider restricted systems on a compact phase space.  In particular, ({\bf{A1}}) implies that the speed $p=\|\p\|$ of the billiard along any typical trajectory at time $t$ satisfies
$$0<p_{\min}\leq p(t) \leq p_{\max}<\infty$$
for some constants $p_{\min}\leq p_{\max}$. Under this assumption the particle will not become
overheated, and its speed will remain bounded. For any phase point $\bx=(\bq, \p) \in \Omega$
for the flow, let $\tau(\bx)$ be the length of the  trajectory between
$\bx$ and its next non-tangential collision. \\

\noindent(\textbf{A2}) (\textbf{Finite horizon}) \emph{There exist $\tau_{\max}>\tau_{\min}>0$ such that free paths between successive reflections are uniformly bounded,
$\tau_*/2 \le \tau_{\min} \le \tau( \bx ) \le \tau_{\max} \le \tau_*^{-1}$, $\forall  \bx \in \Omega$.
Since $Q_0 \in \Q_1(\tau_*, \K_*, E_*)$, the curvature $\cK(r)$ of the boundary
is also uniformly bounded for all $r \in \partial Q_0$.} \\

\noindent(\textbf{A3}) (\textbf{Smallness of the perturbation}). \emph{We assume there exists
$\eps_1>0$ small enough, such that
$$\|\mathbf{F}\|_{C^1}<\eps_1, \|\mathbf G\|_{C^1}<\eps_1 .$$ }

Let $\bv=(\cos\theta, \sin\theta)$ denote the unit velocity vector with $\theta\in [0,2\pi]$,
and  $\cM$ be a level surface $\Omega_c$ with coordinates
$(\bq,\theta)$, for some fixed $c>0$.
Denote $T_{\bF,\bG}: M\to M$ as the  billiard map associated to the flow on
$\cM$, where $M$ is the collision space containing all post-collision vectors
based at the boundary of the billiard table $Q_0$.\\

\noindent(\textbf{A4})  \emph{We assume both forces $\bF$ and $\bG$ are stationary
and that $\bG$ preserves tangential collisions.
In addition, we assume that the singularity set of $T^{-1}_{\bF, \bG}$ is the same as that
of $T^{-1}_{\bF, \mathbf{0}}$.\footnote{The assumption on the singularity set
of $T^{-1}_{\bF, \bG}$ is not essential to our approach, but is made to simplify the proofs
in Section~\ref{kick},  since the paper is already quite long and we include a number of
distinct applications.}} \\

The case $\mathbf{F} = \mathbf{G}=0$ corresponds to the classical billiard dynamics. It
preserves the kinetic energy $\E =\frac{1}{2} \|\p\|^2$.
We denote by $\F_B(Q_0, \tau_*, \eps_1)$ the class of all perturbed billiard maps
defined by the dynamics \eqref{flowf} and \eqref{reflectiong} under
forces $\mathbf{F}$ and $\mathbf{G}$, satisfying assumptions {\bf (A1)}-{\bf (A4)}.

\begin{theorem}
\label{thm:C1}
For any $T\in \cF(Q_0, \tau_*, \eps_1)$, the perturbed system
$T$ satisfies {\bf (H1)}-{\bf (H5)} with uniform constants
depending only on $\eps_1$, $\tau_*$, $\K_*$ and $E_*$.
\end{theorem}

\begin{theorem}
\label{thm:C2}
Within the class $\cF_B(Q_0, \tau_*, \eps_1)$,
the change of either the force $\bF$ or $\bG$ by  a small amount
$\delta$ yields a perturbation of size $\mathcal{O}(|\delta|^{1/2})$
in the distance $d_\F(\cdot, \cdot)$. 

As a consequence, the spectral gap enjoyed by the classical billiard $T_{\mathbf{0}, \mathbf{0}}$
persists for all $T_{\bF, \bG} \in \cF(Q_0, \tau_*, \eps_1)$ 
for $\ve_1$ sufficiently small so that we may apply the limit theorems of
Corollary~\ref{cor:limit theorems} to any such $T_{\bF, \bG}$.  
\end{theorem}

The limit theorems implied by Theorem~\ref{thm:C2} are new even for the
simplified maps $T_{\bF,\mathbf{0}}$ and $T_{\mathbf{0}, \bG}$.
We provide the proofs of Theorems~\ref{thm:C1} and \ref{thm:C2} in Section~\ref{kick}.


\subsection{Smooth random perturbations}
\label{random}

We follow the expositions in \cite{demers liverani, demers zhang}.
Suppose $\F$ is a class of maps satisfying {\bf (H1)}-{\bf (H5)} and let
$d_\F(\cdot, \cdot)$ be the distance in $\F$ defined in Section~\ref{distance}.  For
$T_0 \in \F$, $\ve > 0$, define
\[
X_\ve(T_0) = \{ T \in \F : d_\F(T, T_0) < \ve \},
\]
to be the $\ve$-neighborhood of $T_0$ in $\F$.

Let $(\Omega, \nu)$ be a probability space and let $g: \Omega \times M \to \mathbb{R}^+$
be a measurable function satisfying:  There exist constants $a, A >0$ such that
\begin{enumerate}
  \item[(i)] $g(\omega, \cdot) \in \C^1(M, \mathbb{R}^+)$
  and $|g(\omega, \cdot)|_{\C^1(M)} \le A$ for each $\omega \in \Omega$;
  \item[(ii)] $\int_\Omega g(\omega, x) d\nu(\omega) = 1$ for each $x \in M$;
  \item[(iii)] $g(\omega, x) \geq a $  for all
  $\omega \in \Omega$, $x \in M$.
\end{enumerate}

We define a random walk on $M$ by assigning to each $\omega \in \Omega$, a map
$T \in X_\ve(T_0)$.  Starting at $x \in M$, we choose $T_\omega \in X_\ve(T_0)$
according to the
distribution $g(\omega, x)d\nu$.  We apply $T_\omega$ to $x$ and repeat this
process starting at $T_\omega x$.  We say the process defined in this way has size
$\Delta(\nu,g) \le \ve$.

Notice that if $\nu$ is the Dirac measure centered at $\omega_0$, then this
process corresponds to the deterministic perturbation $T_{\omega_0}$ of $T_0$.
If $g \equiv 1$, then the choice of $T_\omega$ is independent of the position $x$, while
in general this formulation allows the choice of the next map to depend on the previous
step taken.

The transfer operator $\Lp_{(\nu,g)}$ associated with the random process is defined by
\[
\Lp_{(\nu,g)}h(x) = \int_\Omega \Lp_{T_\omega} h(x) \, g(\omega, T_\omega^{-1}x) \, d\nu(\omega)
\]
for all $h \in L^1(M,m)$, where $m$ is Lebesgue measure on $M$.

\begin{theorem}
\label{thm:random}
The transfer operator $\Lp_{(\nu,g)}$ satisfies
the uniform Lasota-Yorke inequalities given by Theorem~\ref{thm:uniform}.
Let $\ve_0$ be given by \eqref{eq:s-unstable} and let $\ve \le \ve_0$.  If
$\Delta(\nu,g) \le \ve$, then there exists a constant $C>0$ depending only on {\bf (H1)}-{\bf (H5)},
such that $||| \Lp_{(\nu,g)} - \Lp_{T_0} ||| \le C A \ve^{\beta/2}$.

It follows that all the operators $\Lp_{(\nu,g)}$ enjoy a spectral gap for
$\ve$ sufficiently small if $\Lp_{T_0}$ has a spectral gap and
the limit theorems of Corollary~\ref{cor:limit theorems} apply to
$\Lp_{(\nu,g)}$.
\end{theorem}


\subsection{Large perturbations:  Large translations, rotations and deformations of scatterers}
\label{large pert}

If we fix $\tau_*, \K_* >0$ and $E_* < \infty$, then Theorems~\ref{thm:uniform} and
\ref{thm:F1} imply that the transfer operator $\Lp_T$ corresponding to any
$T \in \F_1(\tau_*, \K_*, E_*)$ is quasi-compact with essential spectral radius bounded by
$\sigma<1$.  In fact, \cite[Theorem 2.5]{demers zhang} implies that $\Lp_T$
has a spectral gap.

Now choose a compact interval $J \subset \R$ and parametrize a continuous path
in $\F_1(\tau_*, \K_*, E_*)$ according to the distance $d_\F(\cdot, \cdot)$.  To
each point $s \in J$ is assigned a map $T_s \in \F_1(\tau_*, \K_*, E_*)$
and a corresponding transfer operator $\Lp_s$.  Fix $\sigma_1 > \sigma$.
Due to Theorem~\ref{thm:close}, there exists $\ve_s >0$ such that
the spectra and spectral projectors of $\Lp_{s'}$ outside the disk of radius
$\sigma_1$ vary H\"older continuously for $s' \in (s-\ve_s, s+\ve_s) =: B(s, \ve_s)$.

The balls $B(s, \ve_s)$, $s \in J$ form an open cover of $J$ and since $J$ is compact,
there is a finite subcover $\{ B(s_i, \ve_{s_i}) \}_{i=1}^n$.  Because these intervals overlap,
as we move along the entire path from one end of $J$ to the other, the spectra and
spectral projectors of $\Lp_s$ vary H\"older continuously in $s$.  We have
proved the following.

\begin{theorem}
\label{thm:large pert}
Let $J \subset \R$ be a compact interval and let $\{ T_s \}_{s \in J} \subset \F_1(\tau_*, \K_*, E_*)$
be a continuously parametrized path according to the distance $d_\F(\cdot , \cdot)$.
Then the spectra and spectral projectors of the associated transfer operators
$\Lp_s$ vary H\"older continuously in the distance $d_\F(\cdot, \cdot)$.

As a consequence,  the related dynamical properties of $T_s$, such as the rate of decay of
correlations and variance in the Central Limit Theorem, vary H\"older continuously even
across large movements and deformations of scatterers as long as the resulting
maps remain in $\F_1(\tau_*, \K_*, E_*)$.   Indeed, since $J$ is compact, the continuity of
the spectral data implies that the spectral gap is uniform along such paths even when the
resulting configurations are no longer close to the original.
This regularity holds as we move scatterers
in such a way that the table changes from finite to infinite horizon.
\end{theorem}

\begin{remark}
One could just as well apply the above large movements of scatterers to billiards
under external forces in the uniform families $\F_B(Q_s, \tau_*, \ve_1)$ and allow
the configurations $Q_s$, $s \in J$, to change over a continuously parametrized path in
$\F_1(\tau_*, \K_*, E_*)$ as long as the horizon along the path
remains bounded uniformly above by $\tau_*^{-1}$.  Theorem~\ref{thm:large pert}
applies to such families of maps as well since they all possess spectral gaps by
Theorem~\ref{thm:C2}.
\end{remark}


\section{Common Approach}
\label{common}

In this section, we describe the common approach we will take for each class of perturbations that we consider.  We begin by formulating general conditions {\bf (H1)}-{\bf (H5)} under which the perturbations of
a billiard map will satisfy the Lasota-Yorke inequalities \eqref{eq:pert LY} with uniform constants.
We also introduce general conditions {\bf (C1)}-{\bf (C4)} to verify that a perturbation is small
in the sense of \eqref{eq:pert small}.  Theorems~\ref{thm:uniform} and \ref{thm:close} show that these
conditions are sufficient to establish the framework of \cite{keller liverani}.  Once this is accomplished, we only need to check that these conditions are satisfied
for each class of perturbations described above.


\subsection{A class of maps with uniform properties}
\label{class of maps}

We fix the phase space $M = \cup_i I_i \times [ - \frac{\pi}{2}, \frac{\pi}{2} ]$
of a billiard map
associated with a periodic Lorentz gas as in Section~\ref{abstract}.

We define the set $\Si_0 = \{ \vf = \pm \frac{\pi}{2} \}$ and for a fixed $k_0 \in \mathbb{N}$,
we define for $k \geq k_0$, the homogeneity strips,
\beq\label{homogeneity}
\Ho_k = \{ (r,\vf) : \pi/2 - k^{-2} < \vf < \pi/2 - (k+1)^2 \}.
\eeq
The strips $\Ho_{-k}$ are defined similarly near $\vf = -\pi/2$.  We also define
$\Ho_0 = \{ (r, \vf) : -\pi/2 + k_0^{-2} < \vf < \pi/2 - k_0^{-2} \}$.
The set $\Si_{0,H} = \Si_0 \cup (\cup_{|k| \ge k_0} \partial \Ho_{\pm k} )$ is therefore fixed
and will give rise to the singularity sets for the maps that we define below,
i.e. for any map $T$ that we consider, we define
$\Si_{\pm n}^T = \cup_{i = 0}^n T^{\mp i} \Si_{0,H}$ to be the singularity sets for
$T^{\pm n}$, $n \ge 0$.

Suppose there exists a class of invertible maps $\mathcal{F}$ such that
for each $T \in \mathcal{F}$, $T : M \setminus \Si_1^T \to M \setminus \Si_{-1}^T$
is a $C^2$ diffeomorphism on each connected component of $M \setminus \Si_1^T$.
We assume that elements of $\mathcal{F}$ enjoy the following uniform properties.

\bigskip
\noindent
{\bf(H1)}
{\em Hyperbolicity and singularities.}  There exist continuous families
of stable and unstable cones $C^s(x)$ and $C^u(x)$, defined on all of $M$, which are strictly invariant for the
class $\F$, i.e.,  $DT(x) C^u(x) \subset C^u(Tx)$ and $DT^{-1}(x) C^s(x) \subset C^s(T^{-1}x)$
for all $T \in \F$ wherever $DT$ and $DT^{-1}$ are defined.

\noindent
\smallskip
We require that the cones $C^s(x)$ and $C^u(x)$ are uniformly transverse on $M$
and that $\Si_{-n}^T$ is uniformly transverse to $C^s(x)$ for each $n \in \N$ and
all $T \in \F$.   We assume in addition that
$C^s(x)$ is uniformly transverse to
the horizontal and vertical directions on all of $M$.\footnote{This is not a restrictive
assumption for perturbations of the Lorentz
gas since the standard cones $\hat C^s$ and $\hat C^u$ for the billiard map satisfy this
property (see for example \cite[Section 4.5]{chernov book});  the common cones
$C^s(x)$ and $C^u(x)$ shared by all maps in the class $\F$ must therefore lie
inside $\hat C^s(x)$ and $\hat C^u(x)$ and therefore satisfy this property.
In any case, a weaker formulation of this assumption  is necessary:  we use in the compactness
argument that the lengths of stable curves in the homogeneity strips $\Ho_k, k \ge k_0$, are
proportional to the width of the strips.  This is only true if stable curves are transverse
to the horizontal direction in such strips.}

\noindent
\smallskip
Moreover,  there exist constants $C_e>0$ and $\Lambda >1$ such that for all
$T \in \F$,
\begin{equation}
\label{eq:uniform hyp}
\| DT^n(x) v \| \ge C_e^{-1} \Lambda^n \| v\|, \forall v \in C^u(x), \; \; \;
\mbox{and} \; \; \;
\| DT^{-n}(x) v \| \ge C_e^{-1} \Lambda^n \| v\|, \forall v \in C^s(x),
\end{equation}
for all $n \ge 0$, where $\| \cdot \|$ is the Euclidean norm on the tangent space $\mathcal{T}_xM$.

\smallskip
\noindent
For any stable curve  $W \in \widehat \W^s$ (see {\bf (H2)} below), the set $W \cap T\Si_0$
(not counting homogeneity strips) is finite or countable
and has at most $K$ accumulation points on $W$, where $K\geq 0$ is a constant
uniform for $T \in \F$.
Let $x_{\infty}$ be one of them
and let $\{x_n\}$ denote the monotonic sequence of points $W \cap T\Si_0$ converging to
$x_{\infty}$.  We denote the part of $W$ between $x_n$ and $x_{n+1}$ by $W_n$.
We assume there exists $C_a >0$ such that
the expansion factor on $W_n$ satisfies
\begin{equation}
\label{eq:expansion}
C_a n [\cos  \vf(T^{-1}x)]^{-1} \| v \| \leq \|DT^{-1}(x) v\| \leq C_a^{-1} n [\cos \vf(T^{-1}x)]^{-1} \| v \| ,\,\,\,\,\,\,\forall x\in W_n, \forall v\in C^s(x),
\end{equation}
where $\vf(y)$ denotes the angle at the point $y = (r, \vf) \in M$.
Let exp$_x$ denote the exponential map from $\mathcal{T}_xM$ to $M$.
We require the following bound on the second derivative,
\begin{equation}
\label{eq:2 deriv}
C_a n^2 [\cos  \vf(T^{-1}x)]^{-3} \leq \|D^2T^{-1}(x) v\| \leq C_a^{-1} n^2 [\cos \vf(T^{-1}x)]^{-3},\,\,\,\,\,\,\forall x\in W_n,
\end{equation}
for all $v \in \mathcal{T}_xM$ such that $T^{-1}(\mbox{exp}_x(v))$ and $T^{-1}x$ lie in the
same homogeneity strip.

\smallskip
\noindent
We assume there exist constants $c_s, \upsilon_0 >0$ such that if
$x \in W_n$ and $T^{-1}x \in \Ho_k$, then
\begin{equation}
\label{eq:3 deriv}  k \ge c_s n^{\upsilon_0} .\end{equation}

\smallskip
\noindent
If $K=0$ (i.e. the billiard has finite horizon) then the indexing scheme above based on
$n$ is finite and \eqref{eq:expansion}, \eqref{eq:2 deriv} and \eqref{eq:3 deriv} hold with $n=1$.

If $W \cap T\Si_0$ has no accumulating points, we assume (\ref{eq:expansion}), \eqref{eq:2 deriv} hold with $n=1$.

\bigskip
\noindent
{\bf(H2)}
{\em Families of stable and unstable curves.}   We call $W$ a {\em stable curve}
for a map $T \in \F$ if the tangent line to $W$, $\mathcal{T}_xW$ lies in $C^s(x)$
for all $x \in W$.  We call $W$ {\em homogeneous} if $W$ is contained in one homogeneity
strip $\Ho_k$.   Unstable curves are defined similarly.

\smallskip
\noindent
We assume that there exists a family of smooth stable curves, $\widehat\W^s$,
such that each $W \in \widehat\W^s$ is a $\C^2$ stable curve with curvature bounded above
by a uniform constant $B >0$.  The family $\widehat\W^s$ is required to be invariant under
$\F$ in the following sense:  For any $W \in \widehat\W^s$ and $T \in \F$,
the connected components of $T^{-1}W$ are again elements of $\widehat \W^s$.

\smallskip
\noindent
A family of unstable curves $\widehat\W^u$ is defined analogously, with obvious modifications:
For example, we require the connected components of $TW$ to be elements of
$\widehat\W^u$ for all $W \in \widehat\W^u$ and $T \in \F$.

\bigskip
\noindent
{\bf(H3)}
{\em One-step expansion.}
We formulate the one-step expansion in terms of an adapted norm $\| \cdot \|_*$,
uniformly equivalent to $\| \cdot \|$, in
which the constant $C_e$ in \eqref{eq:uniform hyp} can be taken to be $1$, i.e. we have expansion
and contraction in one step in the adapted norm. We assume such a norm exists for
maps in the class $\F$.

\smallskip
\noindent
Let $W \in \widehat W^s$.
For any $T \in \F$, we partition the connected components of $T^{-1}W$ into
maximal pieces $V_i = V_i(T)$ such that each $V_i$ is a homogeneous stable curve  in some
$\Ho_k$, $k\geq k_0$, or $\Ho_0$.
Let $|J_{V_i}T|_*$
denote the minimum contraction on $V_i$ under $T$ in the metric
induced by the adapted norm $\| \cdot \|_*$.
We assume that for some choice of $k_0$,
\begin{equation}
\label{eq:step1}
\limsup_{\delta \to 0} \sup_{T \in \F} \sup_{|W|<\delta} \sum_i |J_{V_i}T|_* < 1,
\end{equation}
where $|W|$ denotes the arclength of $W$.

\smallskip
\noindent
In addition, we require that the above sum converges even when the expansion on each piece
is weakened slightly in the following sense:  There exists $\varsigma_0 < 1$ such that for
all $\varsigma > \varsigma_0$, there exists $C_\varsigma = C_\varsigma (\varsigma, \delta)$
such that for all $T \in \F$ and any $W \in \widehat\W^s$ with $|W| < \delta$,
\begin{equation}
\label{eq:weakened step1}
\sum_i |J_{V_i}T|^\varsigma_{\C^0(V_i)} < C_\varsigma ,
\end{equation}
where $J_{V_i}T$ denotes the stable Jacobian of $T$ along the curve $V_i$ with respect
to arc length.
We formulate \eqref{eq:weakened step1} in terms of
the usual Euclidean norm since we do not need $C_\varsigma <1$, i.e. we only need the
above sum to be finite in some uniform sense.

\bigskip
\noindent
{\bf(H4)}
{\em Bounded distortion.}
There exists a constant $C_d>0$ with the following properties.
Let $W' \in \widehat\W^s$ and for any $T \in \F$, $n \in \N$, let $x, y \in W$ for some
connected component $W \subset T^{-n}W'$ such that $T^iW$ is a
homogeneous stable curve for each $0 \le i \le n$. Then,
\begin{equation}
\label{eq:distortion stable}
\left| \frac{J_\mu T^n(x)}{J_\mu T^n(y)} -1 \right|  \; \leq \;  C_d d_W(x,y)^{1/3}   \; \;
\mbox{and} \; \; \left| \frac{J_WT^n(x)}{J_WT^n(y)} -1 \right| \; \leq \; C_d d_W(x,y)^{1/3},
\end{equation}
where $J_\mu T^n$ is the Jacobian of $T^n$ with respect to the smooth measure
$d\mu = \cos \vf dr d\vf$.

\smallskip
\noindent
We assume the analogous bound along unstable leaves:
If $W \in \widehat\W^u$ is an unstable curve such that $T^iW$ is a homogeneous unstable curve
for $0\le i \le n$, then for any $x, y \in W$,
\begin{equation}
\label{eq:D u dist}
\left| \frac{J_\mu T^n(x)}{J_\mu T^n(y)} -1 \right|  \; \leq \;  C_d d(T^nx,T^ny)^{1/3}     .
\end{equation}

\bigskip
\noindent
{\bf(H5)}
{\em Control of Jacobian.}
Let $\beta, q < 1$ be from the definition of the norms in Section~\ref{norms} and let
$\theta_*<1$ be from \eqref{eq:one step contract}.
Assume there exists a constant
$\eta < \min \{ \Lambda^\beta, \Lambda^q, \theta_*^{\alpha-1} \}$
such that for any $T \in \F$,
\[
(J_\mu T(x))^{-1} \le \eta, \; \; \;
\mbox{wherever $J_\mu T(x)$ is defined.}
\]


\subsection{Transfer operator}

\label{transfer}

Recall the family of stable curves $\widehat \W^s$ defined by {\bf (H2)}.  We define
a subset $\W^s \subset \widehat \W^s$ as follows.  By {\bf (H3)} we may choose
$\delta_0 > 0$ for which there exists $\theta_*<1$ such that
\begin{equation}
\label{eq:one step contract}
\sup_{T \in \F} \sup_{|W| \le \delta_0} \sum_i |J_{V_i}T|_* \le \theta_* .
\end{equation}
We shrink $\delta_0$ further if necessary so that the graph transform argument in
Lemma~\ref{lem:angles}(a) holds.
The set $\W^s$ comprises all those stable curves  $W \in \widehat\W^s$ such that
$|W| \le \delta_0$.

For any $T \in \F$,
we define scales of spaces
using the set of stable curves $\W^s$ on which the
{\em transfer operator} $\Lp_T$ associated with $T$ will act.
Define $T^{-n}\W^s$
to be the set of homogeneous stable curves $W$ such that $T^n$ is smooth on $W$ and
$T^iW \in \W^s$ for $0 \leq i \le n$.   It follows from {\bf (H2)}  that
$T^{-n}\W^s \subset \W^s$.
We denote (normalized) Lebesgue measure on $M$ by $m$.

For $W \in T^{-n}\W^s$, a complex-valued test function $\psi: M \to \mathbb{C}$, and $0<p\le 1$ define $H^p_W(\psi)$ to be
the H\"older constant of $\psi$ on $W$ with exponent $p$ measured in the
Euclidean metric.
Define $H^p_n(\psi) = \sup_{W \in T^{-n}\W^s} H^p_W(\psi)$
and let $\tilde{\C}^p(T^{-n}\W^s) = \{ \psi : M \to \mathbb{C} \mid H^p_n(\psi) < \infty \}$,
denote the set of complex-valued functions which are H\"older continuous on elements of
$T^{-n}\W^s$.
The set $\tilde{\C}^p(T^{-n}\W^s)$ equipped with the norm
$| \psi |_{\C^p(T^{-n}\W^s)} = |\psi|_\infty + H^p_n(\psi)$ is a Banach space.
Similarly, we define $\tilde \C^p(\widehat \W^u)$, the set of functions which are H\"older continuous
with exponent $p$ on unstable curves $\widehat \W^u$.

It follows from \eqref{eq:C1 C0} that if $\psi \in \tilde{\C^p}(T^{-(n-1)}\W^s)$, then
$\psi \circ T \in \tilde{\C}^p(T^{-n}\W^s)$.  Thus
if $h\in(\tilde \C^p(T^{-n}\W^s))'$, is an element of the dual of $\tilde \C^p(T^{-n}\W^s)$,
then
$\Lp_T :(\tilde \C^p(T^{-n}\W^s))'\to (\tilde \C^p(T^{-(n-1)}\W^s))'$ acts on $h$ by
\[
\Lp_T h(\psi) = h(\psi \circ T) \quad \forall \psi \in \tilde \C^p(T^{-(n-1)}\W^s).
\]
Recall that $d\mu = c \cos \vf dr d\vf$ denotes the smooth invariant measure for the unperturbed
Lorentz gas.
If $h \in L^1(M,\mu)$, then $h$ is canonically identified with a signed measure
absolutely continuous with respect to $\mu$, which we shall also call $h$, i.e.,
$
h(\psi) = \int_M \psi h \, d\mu.
$
With the above
identification, we write $L^1(M,\mu) \subset (\tilde \C^p(T^{-n}\W^s))'$ for each $n \in \N$.
Then restricted to $L^1(M,\mu)$, $\Lp_T$ acts according to the familiar
expression
\[
\Lp_T^n h = h \circ T^{-n} \; (J_\mu T^n(T^{-n}))^{-1} \; \; \;
\mbox{for any $n \geq 0$ and $h \in L^1(M,\mu)$.}
\]

\begin{remark}
In \cite{demers zhang}, we used Lebesgue measure as a reference measure to
show that the functional analytic framework developed there did not need to assume
the existence of a smooth invariant measure.  Now that $\mu$ has been established
in our function space $\B$, however, we find it more convenient to use it as a starting
point in our study of the classes of perturbations considered here.
It also simplifies our norms and estimates slightly
since for example, it eliminates the need for the $\cos W$ weight in our test functions
that was used in \cite{demers zhang}.
We do not assume that $\mu$ is an invariant measure for $T \in \F$; indeed, the SRB
measures for such $T$ are in general singular with respect to Lebesgue measure.
\end{remark}


\subsection{Definition of the Norms}
\label{norms}

The norms are defined via integration on the set of stable curves
$\W^s$.  Before defining the norms, we define
the notion of a distance $d_{\W^s}(\cdot, \cdot)$ between such curves as well as
a distance $d_q(\cdot, \cdot)$ defined among functions supported on these curves.

Due to the transversality condition on the stable cones
$C^s(x)$ given by {\bf (H1)}, each stable curve $W$ can be viewed
as the graph of a function $\vf_W(r)$ of the arc length parameter $r$.
For each $W \in \W^s$,
let $I_W$ denote the interval on which
$\vf_W$ is defined and set $G_W(r) = (r, \vf_W(r))$ to be its graph so that
$W = \{ G_W(r) : r \in I_W \}$.
We let $m_W$ denote the unnormalized arclength measure on $W$.

Let $W_1, W_2 \in \W^s$ and identify them with the graphs $G_{W_i}$ of their
functions $\vf_{W_i}$, $i = 1,2$.  Suppose $W_1, W_2$ lie in the same component of $M$
and let $I_{W_i}$ be the $r$-interval on which each curve is defined.
Denote by $\ell(I_{W_1} \triangle I_{W_2})$ the length of the symmetric difference
between $I_{W_1}$ and $I_{W_2}$.
Let
$\Ho_{k_i}$ be the homogeneity strip containing $W_i$.
We define the distance between $W_1$ and $W_2$ to be,
\[
d_{\W^s} (W_1,W_2) = \eta(k_1, k_2) +
\ell( I_{W_1} \triangle I_{W_2}) + |\vf_{W_1} -\vf_{W_2}|_{\C^1(I_{W_1} \cap I_{W_2})}
\]
where $\eta(k_1,k_2) = 0$ if $k_1=k_2$ and $\eta(k_1,k_2) = \infty$ otherwise,
i.e., we only compare curves which lie in the same homogeneity strip.

For $0 \leq p \leq 1$, denote by
$\tilde{\C}^p(W)$ the set of continuous complex-valued functions on
$W$ with H\"{o}lder exponent $p$, measured in the Euclidean
metric, which we denote by $d_W(\cdot, \cdot)$.
We then denote by $\C^p(W)$ the closure of $\C^\infty(W)$
in the $\tilde{\C}^p$-norm\footnote{While $\C^p(W)$ may not contain
all of $\tilde{\C}^p(W)$, it does contain $\C^{p'}\!(W)$ for all $p'>p$.}:
$| \psi |_{\C^p(W)} = |\psi|_{\C^0(W)} + H^p_W(\psi)$, where
$H^p_W(\psi)$ is the H\"older constant of $\psi$ along $W$.
Notice that with this definition,
$|\psi_1 \psi_2 |_{\C^p(W)} \le |\psi_1|_{\C^p(W)} |\psi_2|_{\C^p(W)}$.
We define $\tilde{\C}^p(M)$ and $\C^p(M)$ similarly.

Given two functions
$\psi_i\in\C^q(W_i,\mathbb{C})$, $q >0$, we define the distance between
$\psi_1$, $\psi_2$ as
\[
d_q(\po,\pt) =|\po\circ G_{W_1}-\pt\circ G_{W_2}|_{\C^q(I_{W_1} \cap I_{W_2})}.
\]
We will define the required Banach spaces by closing $\C^1(M)$ with respect to
the following set of norms.
For $s,p \geq 0$, define the following norms for test functions,
\[
|\psi|_{W,s,p}:=|W|^s \cdot|\psi|_{\C^p(W)} .
\]

Now fix $0 < p \le \frac 13 $.
Given a function $h \in \C^1(M)$, define the \emph{weak norm}
of $h$ by
\begin{equation}
\label{eq:weak}
|h|_w:=\sup_{W\in\W^s}\sup_{\substack{\psi \in\C^p(W)\\
|\psi|_{W,0,p} \leq 1}}\int_W h \psi \; dm_W .
\end{equation}
Choose\footnote{The restrictions on the constants are placed according to the
dynamical properties of $T$.  For example, $p \le 1/3$ due to the distortion bounds
in {\bf (H4)}, while $\alpha < 1-\varsigma_0$ so that Lemma~\ref{lem:growth}(d) can be
applied with
$\varsigma = 1 - \alpha > \varsigma_0$.}
$\alpha$, $\beta$, $q >0$ such that $\alpha < 1-\varsigma_0$, $q < p$ and
$\beta \leq \min \{ \alpha, p-q \}$.
We define the \emph{strong stable norm} of $h$ as
\begin{equation}
\label{eq:s-stable}
\|h\|_s:=\sup_{W\in\W^s}\sup_{\substack{\psi\in\C^q(W)\\
|\psi|_{W,\alpha,q}\leq 1}}\int_W h \psi \; dm_W
\end{equation}
and the \emph{strong unstable norm} as
\begin{equation}
\label{eq:s-unstable}
\|h\|_u:=\sup_{\varepsilon \leq \varepsilon_0} \; \sup_{\substack{W_1,
W_2 \in \W^s \\
d_{\W^s} (W_1,W_2)\leq \varepsilon}}\;
\sup_{\substack{\psi_i \in \C^p(W_i) \\ |\psi_i|_{W_i,0,p}\leq 1\\ d_q(\psi_1,\psi_2)
\leq \ve}} \;
\frac{1}{\varepsilon^\beta} \left| \int_{W_1} h
\psi_1 \; dm_W - \int_{W_2} h \psi_2 \; dm_W \right|
\end{equation}
where $\ve_0 > 0$ is chosen less than $\delta_0$, the maximum length of $W \in \W^s$ which
is determined by \eqref{eq:one step contract}.
We then define the \emph{strong norm} of $h$ by
\[
\|h\|_\B = \|h\|_s + b \|h\|_u
\]
where $b$ is a small constant chosen in Section~\ref{uniform}.

We define $\B$ to be the completion of $\C^1(M)$ in the strong norm\footnote{As a measure,
$h \in \C^1(M)$ is identified with $hd\mu$ according to our earlier convention.
As a consequence, Lebesgue measure $dm = (\cos \vf)^{-1} d\mu$ is not automatically
included in $\B$ since $(\cos \vf)^{-1} \notin \C^1(M)$.  We will prove in
Lemma~\ref{lem:lebesgue} that in fact, $m \in \B$ (and $\B_w$).}
and $\B_w$ to be the completion of $\C^1(M)$ in the weak norm.



\subsection{Distance in $\F$}
\label{distance}

We define a distance in $\mathcal{F}$ as follows.  Let $\ve_0$ be from \eqref{eq:s-unstable}.
For $T_1, T_2 \in \F$ and $\ve \le \ve_0$, let
$N_\ve(\Si^i_{-1})$ denote the $\ve$-neighborhood in $M$ of the singularity set $\Si^i_{-1}$ of
$T_i^{-1}$, $i = 1,2$.   We say $d_{\F} (T_1, T_2) \le \ve$
if
the maps are close away from their singularity sets in the following sense:
For $x \notin N_\ve(\Si^1_{-1} \cup \Si^2_{-1})$,

\medskip
\noindent
\parbox{.07 \textwidth}{\bf(C1)}
\parbox[t]{.91 \textwidth}{
$ \displaystyle
d(T_1^{-1}(x) , T_2^{-1}(x))  \le \ve$;
}

\medskip
\noindent
\parbox{.07 \textwidth}{\bf(C2)}
\parbox[t]{.91 \textwidth}{
$ \displaystyle
\left| \frac{J_\mu T_i(x)}{ J_\mu T_j(x)}  - 1 \right| \le \ve$, $i,j = 1,2$;
}

\medskip
\noindent
\parbox{.07 \textwidth}{\bf(C3)}
\parbox[t]{.91 \textwidth}{
$ \displaystyle
\left| \frac{J_WT_i(x)}{ J_WT_j(x)}  - 1 \right| \le \ve$,
for any $W \in \W^s$, $i,j = 1,2$, and $x \in W$;
}

\medskip
\noindent
\parbox{.07 \textwidth}{\bf(C4)}
\parbox[t]{.91 \textwidth}{
$ \displaystyle
\| DT_1^{-1}(x) v - DT_2^{-1}(x) v \|  \le \sqrt{\ve}$, for any unit vector $v \in \mathcal{T}_xW$,
$W \in \W^s$.
}

\subsection{Preliminary estimates}
\label{preliminary}

Before proving the Lasota-Yorke inequalities, we show how {\bf (H1)}-{\bf (H5)} imply
several other uniform properties for our class of maps $\F$.  In particular, we will be
interested in iterating the one-step expansion relations given by {\bf (H3)}.
We recall the estimates we need from \cite[Section 3.2]{demers zhang}.

Let $T \in \F$ and $W \in \W^s$.  Let $V_i$ denote the maximal connected components
of $T^{-1}W$ after cutting due to singularities and the boundaries of the homogeneity
strips.  To ensure that each component of $T^{-1}W$ is in $\W^s$, we subdivide any of the
long pieces $V_i$ whose length is $>\delta_0$, where $\delta_0$ is chosen
in \eqref{eq:one step contract}.        This process is then iterated
so that
given $W \in \W^s$,      we construct the components of $T^{-n}W$, which  we call
the $n^{\mbox{\scriptsize th}}$ generation $\G_n(W)$, inductively as follows.
Let $\G_0(W) = \{ W \}$ and suppose we have
defined $\G_{n-1}(W) \subset \W^s$.       First, for any $W' \in \G_{n-1}(W)$,
we partition $T^{-1}W'$ into at most countably many pieces $W'_i$ so that $T$ is
smooth on each $W'_i$ and
each $W'_i$ is a homogeneous stable curve.
If any $W'_i$ have length greater than $\delta_0$, we subdivide those pieces into pieces
of length between $\delta_0/2$ and $\delta_0$.
We define $\G_n(W)$ to be the collection of all pieces $W^n_i \subset T^{-n}W$ obtained
in this way.  Note that each $W^n_i$ is in $\W^s$ by {\bf (H2)}.

At each iterate of $T^{-1}$, typical curves in $\G_n(W)$ grow in size, but there exist a portion of curves which are trapped in tiny homogeneity strips and in the infinite horizon case, stay too close to the infinite
horizon points.  In Lemma~\ref{lem:growth}, we make precise the sense in which the proportion
of curves that never grow to a fixed
length decays exponentially fast.

For $W \in \W^s$, $n \geq 0$, and $0 \le k \le n$, let $\G_k(W) = \{ W^k_i \}$ denote
the $k^{\mbox{\scriptsize th}}$ generation pieces in $T^{-k}W$.  Let
$B_k(W) = \{ i : |W^k_i|< \delta_0/3 \}$  and $L_k(W) = \{ i : |W^k_i| \ge \delta_0/3 \}$
denote the index of the short and long elements of $\G_k(W)$, respectively.
We consider
$\{\G_k \}_{k=0}^n$ as a tree with $W$ as its root and $\G_k$ as the
$k^{\mbox{\scriptsize th}}$ level.

At level $n$, we group the pieces as follows.  Let $W^n_{i_0} \in \G_n(W)$ and
let $W^k_j \in L_k(W)$ denote the most recent long ``ancestor" of $W^n_{i_0}$, i.e.\
$k = \max \{ 0 \leq \ell \le n : T^{n-\ell}(W^n_{i_0}) \subset W^\ell_j \; \mbox{and} \;
j \in L_\ell \}$.  If no such ancestor exists, set $k=0$ and $W^k_j = W$.  Note that
if $W^n_{i_0}$ is long, then $W^k_j = W^n_{i_0}$.  Let
\[
\I_n(W^k_j) = \{ i : W^k_j \in L_k(W) \; \mbox{is the most recent long ancestor of} \; W^n_i \in \G_n(W) \}.
\]
The set $\I_n(W)$ represents those curves $W^n_i$ that belong to short pieces in $\G_k(W)$
at each time step $1 \leq k \le n$, i.e.\ such $W^n_i$ are never part of a piece
that has grown to length $\geq \delta_0/3$.

We collect the results of  \cite[Section~3.2]{demers zhang} in the following lemma.

\begin{lemma}
\label{lem:growth}  (\cite{demers zhang})
Let $W \in \W^s$, $T \in \F$ and for $n \geq 0$, let $\I_n(W)$ and $\G_n(W)$
be defined as above.
There exist constants $C_1, C_2, C_3 >0$, independent of $W$ and $T$, such that for
any $n\geq 0$,
\begin{itemize}
  \item[(a)] $\ds
\sum_{i \in \I_n(W)} |J_{W^n_i}T^n|_{\C^0(W^n_i)} \leq C_1 \theta_*^n $;
 \item[(b)] $\ds
\sum_{W^n_i \in \G_n(W)} |J_{W^n_i}T^n|_{\C^0(W^n_i)} \le C_2     $;
 \item[(c)]  for any $0 \leq \varsigma \leq 1$,
$\ds
\sum_{W^n_i \in \G_n(W)} \frac{|W^n_i|^\varsigma}{|W|^\varsigma} \;  |J_{W^n_i}T^n|_{\C^0(W^n_i)} \le C_2^{1-\varsigma} $;
 \item[(d)] for $\varsigma > \varsigma_0$,
  $\ds
\sum_{W^n_i \in \G_n(W)} |J_{W^n_i}T^n|_{\C^0(W^n_i)}^\varsigma \le C_3^n$,
where $C_3$ depends on $\varsigma$.
\end{itemize}
\end{lemma}

\begin{proof}
The proofs of these items are combinatorial and require no more specific information
about the maps than the uniform properties given by {\bf (H2)}, {\bf (H3)} and {\bf (H4)}.

\smallskip
\noindent
(a) This is Lemma 3.1 of \cite{demers zhang}.  The constant $C_1$
depends only on the constant relating the Euclidean norm $\| \cdot \|$ to the
adapted norm $\| \cdot \|_*$.  As such, $C_1$ is independent of $T \in \F$,
$W \in \W^s$ and $n \in \mathbb{N}$.

\smallskip
\noindent
(b) This statement is \cite[Lemma 3.2]{demers zhang}.  The constant $C_2=C_2(\delta_0,\theta_*,C_1, C_d)$.

\smallskip
\noindent
(c) This is \cite[Lemma 3.3]{demers zhang}.  It follows from (b) by an application of
Jensen's inequality.

\smallskip
\noindent
(d) This follows from \eqref{eq:weakened step1} and is proved in \cite[Lemma 3.4]{demers zhang}.
The constant $C_3 = \delta_0^{-1} C_\varsigma (1+C_d)^{2\varsigma}$ is uniform
for $T \in \F$, but depends on $\varsigma$.
\end{proof}

Next we prove a distortion bound for the stable Jacobian of $T$ along
different stable curves in the following context.
Let $W^1, W^2 \in \W^s$ and suppose there exist $U^k \subset T^{-n}W^k$, $k=1,2$,
such that for $0 \le i \le n$,
\begin{enumerate}
  \item[(i)]  $T^iU^k \in \W^s$ and the curves $T^iU^1$ and $T^iU^2$ lie in the same homogeneity
strip;
  \item[(ii)]  $U^1$ and $U^2$ can be put into a 1-1 correspondence
by a smooth foliation $\{ \gamma_x \}_{x \in U^1}$ of curves
$\gamma_x \in \widehat\W^u$ such that
$\{ T^n\gamma_x \} \subset \widehat\W^u$ creates a 1-1 correspondence between
$T^nU^1$ and $T^nU^2$;
  \item[(iii)]  $|T^i\gamma_x| \le 2 \max \{ |T^iU^1|, |T^iU^2| \}$, for all $x \in U^1$.
\end{enumerate}
Let $J_{U^k}T^n$ denote the stable Jacobian of $T^n$ along the
curve $U^k$ with respect to arclength.

\begin{lemma}
\label{lem:angles}
In the setting above, for $x \in U^1$, define $x^* \in \gamma_x \cap U^2$.  There
exists $C_0 > 0$, independent of $T \in \F$, $W \in \W^s$ and $n \ge 0$ such that
\begin{enumerate}
  \item[(a)]  $d_{\W^s}(U^1, U^2) \le C_0 \Lambda^{-n} d_{\W^s}(W^1, W^2)$;
  \item[(b)] $\ds
\left| \frac{J_{U^1}T^n(x)}{J_{U^2}T^n(x^*)} -1 \right| \; \leq \; C_0[d(T^nx,T^n x^*)^{1/3}
+\theta(T^nx, T^n x^*)] $,
\end{enumerate}
where $\theta(T^nx, T^n x^*) $ is the angle formed by the tangent lines of
$T^nU^1$ and $T^nU_2$ at $T^nx$ and $T^n x^*$, respectively.
\end{lemma}

\begin{proof}
(a)  This is essentially a graph transform argument adapted for this class of maps satisfying
{\bf (H1)}.  What we need to show here is that we do not need to cut curves lying in homogeneity
strips any further in order to get the required contraction and control on distortion.

First notice that due to the uniform expansion of $\gamma_x$ under $T^n$ given by
\eqref{eq:uniform hyp} of {\bf (H1)}, we have $|\gamma_x| \le C_e C_t \Lambda^{-n} d_{\W^s}(W^1, W^2)$, where
$C_t$ is a constant depending only on the minimum angle between $C^u(x)$ and
$C^s(x)$ and between $C^u(x)$ and the horizontal direction.
Again by the transversality of
$\gamma_x$ with $U^1$ and $U^2$, the $r$-intervals on which the functions
$\vf_{U^1}$, $\vf_{U^2}$ describing the curves $U^1$, $U^2$
are defined can differ by no more than
$C_e C_t^2 \Lambda^{-n} d_{\W^s}(W^1, W^2)$.  Letting $I$ denote the intersection
of intervals
on which both functions are defined and recalling the definition of $d_{\W^s}(\cdot, \cdot)$ from
Section~\ref{norms}, it remains to estimate
 $|\vf_{U^1} - \vf_{U^2}|_{\C^1(I)}$.

By the same observation as above, we have
$|\vf_{U^1} - \vf_{U^2}|_{\C^0(I)} \le C_t^2 C_e \Lambda^{-n} d_{W^s}(W^1, W^2)$.
In order to show that the slopes of these curves also contract exponentially, we
make the usual graph transform argument using charts in the adapted norm $\| \cdot \|_*$
from {\bf (H3)}.

Fix $x \in U^1$ and define charts along the
orbit of $x$ so that $x_i := T^ix$, $0 \le i \le n$, corresponds
to the origin in each chart with the stable direction at $x_i$ given by the horizontal axis and the unstable direction by the vertical axis in the charts.  Let $\vartheta < 1$ denote the maximum
absolute value of slopes of stable curves in the chart.  Due to property (iii) before the
statement of the lemma, we may choose the size of the charts to have stable and unstable
diameters $\le C |T^iU1|$ for each $i$, for some uniform constant $C$.  The dynamics
induced by $T^{-1}$ on these charts is defined by
\[
\tilde T^{-1}_{x_i} = \chi^{-1}_{x_{i-1}} \circ T^{-1}  \circ \chi_{x_i}
\]
where $\chi_{x_i}$ are smooth maps with $|\chi_{x_i}|_{\C^2}, |\chi^{-1}_{x_i}|_{\C^2}
\le C$ for some uniform constant $C$.

Note that $D\tilde T_{x_i}^{-1}$ and
$D^2\tilde T_{x_i}^{-1}$ satisfy {\bf (H1)} with possibly larger $C_a$ and $C_e=1$.
In the chart coordinates, since $\tilde T^{-1}_{x_i}(0) = 0$, we have
\[
\tilde T^{-1}_{x_i} (s, t) = (A_i s + \alpha_i(s,t), B_i t + \beta_i(s,t))
\]
where $A_i$ is the expansion at $x_i$ in the stable direction and
$B_i$ is the contraction at $x_i$ in the unstable direction given
by $DT^{-1}_{x_i}(0)$.  The nonlinear functions $\alpha_i, \beta_i$ satisfy
$\alpha_i(0,0)=\beta_i(0,0) =0$ and their Lipschitz constants are bounded by the
maximum of
\begin{equation}
\label{eq:alpha lip}
\| D\tilde T_{x_i}^{-1}(u) - D\tilde T_{x_i}^{-1}(v) \|
\le \| D^2 \tilde T_{x_i}(z) \| \| u- v \|
\end{equation}
where $u, v, z$ range over the chart at $x_i$.

We fix $i$ and let $\vf_1$, $\vf_2$ denote two Lipschitz functions whose graphs
lie in the stable cone of the chart at $x_i$ and satisfy $\vf_j(0)=0$, $j=1,2$.  Define
$L(\vf_1, \vf_2) = \sup_{s \neq 0} \frac{| \vf_1(s) - \vf_2(s) |}{|s|}$.
Let
$\vf_1' = \tilde T_*^{-1} \vf_1$ and $\vf_2'=\tilde T_*^{-1} \vf_2$ denote the graphs of the
images of these two curves in the chart at $x_{i-1}$.
We wish to estimate
$L(\vf_1', \vf_2')$.  For $s$ on the horizontal axis in the chart at $x_i$, we write,
\[
\begin{split}
|\vf_1'(A_i s + & \alpha_i(s, \vf_1(s))) -  \vf_2'(A_is + \alpha_i(s, \vf_1(s))) |
 \le |\vf_1'(A_i s + \alpha_i(s, \vf_1(s))) - \vf_2'(A_is + \alpha_i(s, \vf_2(s))) | \\
& \qquad + |\vf_2'(A_i s + \alpha_i(s, \vf_2(s))) - \vf_2'(A_is + \alpha_i(s, \vf_1(s))) | \\
& \le |B_i| | \vf_1(s) - \vf_2(s) | + | \beta_i(s, \vf_1(s)) - \beta_i(s, \vf_2(s)) |
+ \vartheta | \alpha_i(s, \vf_1(s)) - \alpha_i(s, \vf_2(s)) |   \\
& \le (|B_i| + \mbox{Lip}(\beta_i) + \vartheta \mbox{Lip}(\alpha_i)) | \vf_1(s) - \vf_2(s) |
\end{split}
\]
On the other hand, by \eqref{eq:expansion},
\[
|A_i s +  \alpha_i(s, \vf_1(s))| \ge (|A_i|  - \mbox{Lip}(\alpha_i)(1+\vartheta))|s| .
\]
Putting these together, we see that,
\begin{equation}
\label{eq:lip}
L(\vf_1', \vf_2') \le \sup_{s \neq 0} \frac{(|B_i| + \mbox{Lip}(\beta_i) + \vartheta \mbox{Lip}(\alpha_i)) | \vf_1(s) - \vf_2(s) |}{(|A_i|  - \mbox{Lip}(\alpha_i)(1+\vartheta))|s|} \le
\frac{|B_i| + \mbox{Lip}(\beta_i) + \vartheta \mbox{Lip}(\alpha_i)}{|A_i|  - \mbox{Lip}(\alpha_i)(1+\vartheta)}  L(\vf_1, \vf_2) .
\end{equation}

Suppose that $x_{i-1}$
lies in the homogeneity strip $\Ho_k$ and $x_i$ lies on a curve with index $n$
according to the index given by {\bf (H1)}.  Then by \eqref{eq:alpha lip} and
\eqref{eq:expansion} and \eqref{eq:2 deriv} of {\bf (H1)},  the Lipshitz constants
of  $\alpha_i$ and $\beta_i$ are bounded by
$C_a^{-1} n^2 k^6 (C_a^{-1} n^{-1} k^{-5}) = C_a^{-2} nk$ since the size of the chart
is taken to be on the order of the length of the curve $T^iU^1$ by property (iii) of the
matching.  Thus,
\[
L(\vf_1', \vf_2') \le \frac{\Lambda^{-1} + C_a^{-2} nk (1+ \vartheta)}{C_a nk^2 - C_a^{-2} nk (1+\vartheta)} L(\vf_1, \vf_2) \le \frac{4C_a^{-3}}{k} L(\vf_1, \vf_2),
\]
for large $k$, which can be made smaller than $\Lambda^{-1}$.  Note that since $k \ge c_s n^{\upsilon_0}$ by {\bf (H1)}, this bound is also small for large $n$.  Thus we may choose
$N_0, K_0 >0$ such that the contraction is less than $\Lambda^{-1}$ on all curves
with index $n \ge N_0$ or landing in homogeneity strip $\Ho_k$, $k \ge K_0$.
On the remainder of
$M$, the first and second derivatives of $T^{-1}$ are uniformly bounded by constants
depending on $N_0$ and $K_0$.  For curves in this part of $M$, we choose $\delta_0$,
the maximum length of stable curves in $\W^s$, sufficiently small that the distortion
given by \eqref{eq:alpha lip} is less than
$\frac 12 ( \Lambda^{-1/2} - \Lambda^{-1})$.
Then by \eqref{eq:lip}, since $\vartheta <1$,
the contraction on these pieces is less than $\Lambda^{-1}$ as well.

If $\vf_1$ and $\vf_2$ do not pass through the origin, the exponential contraction in
the $C^0$ norm coupled with the above argument yields the required contraction.

\noindent
(b) It is equivalent to estimate the ratio
$\log \frac{J_{T^nU_1}T^{-n}(T^nx)}{J_{T^n_U2}T^{-n}(T^nx^*)}$.  We write
\begin{equation}
\label{eq:jac split}
\log \frac{J_{T^nU_1}T^{-n}(T^nx)}{J_{T^n_U2}T^{-n}(T^nx^*)}
\le \sum_{i=1}^n \frac{1}{A_i}  | J_{T^iU_1}T^{-1} (T^ix) - J_{T^iU_2}T^{-1}(T^ix^*)|
\end{equation}
where $A_i = \min \{ J_{T^iU_1}T^{-1}(T^ix), J_{T^iU_2}T^{-1}(T^ix^*) \}$.

We estimate the differences one term at a time and assume without loss of generality that
the minimum for $A_i$ is attained at $T^ix$.  Set $x_i = T^ix$, $x_i^* = T^ix^*$.
Let $\vu_1(x_i)$ denote the unit tangent  vector to $T^iU^1$ at $x_i$ and
notice that $J_{T^iU_1}T^{-1}(x_i) = \| DT^{-1}(x_i) \vu_1 \|$.  Define
$\vu_2(x_i^*)$ similarly.  Then
\[
\begin{split}
| \, \| DT^{-1}(x_i) \vu_1 \| - \| DT^{-1}(x_i^*) \vu_2 \| \, |
& \le | \,  \| DT^{-1}(x_i) \vu_1 \| - \| DT^{-1}(x_i) \vu_2 \| \, |  \\
& \qquad + | \, \| DT^{-1}(x_i) \vu_2 \|
- \| DT^{-1}(x_i^*) \vu_2 \|   \, | \\
& \le \| DT^{-1}(x_i) \| \, \| \vu_1 - \vu_2 \| + \| D^2T^{-1}(z_i) \| d(x_i, x_i^*) ,
\end{split}
\]
where $z_i$ is some point on $T^i\gamma_x$.

Suppose $T^{-1}x_i$ lies in the homogeneity strip $\Ho_k$ and $x_i$ lies on some
curve $W_n$ according to the index given in {\bf (H1)}.
Then $\| DT^{-1}(x_i) \| / \| DT^{-1}(x_i) \vu \| \le C$ where $C$ is some uniform constant
for all unit vectors $\vu \in C^s(x_i)$.
Also by {\bf (H1)}, we have $|T^iU^j| \le C_a/nk^5$, $j=1,2$, so
that by property (iii) before the statement of the lemma, $d(x_i, x_i^*) \le 2C_a/(nk^5)$.  Thus
\[
\frac{\| D^2T^{-1}(z_i) \| d(x_i, x_i^*)}{\| DT^{-1}(x_i) \vu_1 \|}
\le \frac{(C_a n^2 k^6) (2C_a/(nk^5))}{C_a^{-1} nk^2} \le \frac{2C_a^3}{k} \le 2C_a^3 d(x_{i-1}, x_{i-1}^*)^{1/3} .
\]
Using these estimates in \eqref{eq:jac split}, we have
\[
\log \frac{J_{T^nU_1}T^{-n}(T^nx)}{J_{T^n_U2}T^{-n}(T^nx^*)}
\le C \sum_{i=1}^n \| \vu_1(x_i) - \vu_2 (x_i^*) \| + d(x_{i-1}, x_{i-1}^*)^{1/3}.
\]
Now $\| \vu_1(x_i) - \vu_2(x_i^*) \| \le \theta(x_i, x_i^*) \le C_0 \Lambda^{i-n} \theta(T^nx, T^nx^*)$
by part (a) of the lemma together with the fact that curves in $\W^s$ have $\C^2$ norm
uniformly bounded above.
Finally, by {\bf (H1)}, $d(x_{i-1}, x_{i-1}^*) \le C_e \Lambda^{i-n-1} d(T^nx, T^nx^*)$,
which completes the proof of the lemma.
\end{proof}


\subsection{Properties of the Banach spaces}
\label{recall property}

We first prove that the weak and strong norms dominate distributional norms on $M$
in the following sense.

\begin{lemma}
\label{lem:distr}
There exists $C >0$ such that for any $h \in \B_w$, $T \in \F$, $n \ge 0$  and
$\psi \in \C^p(T^{-n}\W^s)$,
  \[
  |h(\psi)| \le C |h|_w (|\psi|_\infty + H^p_n(\psi)) .
  \]
\end{lemma}

This is the analogue of Lemma~3.9 of \cite{demers zhang}, but it does not follow
from the argument given there since
condition {\bf (H5)} and the weakened Lasota-Yorke inequalities
in Theorem~\ref{thm:uniform} suggest
that the spectral radius of $\Lp_T$ can be as much as $\eta >1$.  It is a consequence
of Lemma~\ref{lem:distr} that the spectral radius is in fact 1
(see Section~\ref{uniform}, proof of Theorem~\ref{thm:uniform}).

\begin{proof}[Proof of Lemma~\ref{lem:distr}]
On each $M_\ell = \Gamma_\ell \times [-\pi/2, \pi/2]$, we partition the set $\Ho_0$
into finitely many boxes $B_j$ whose boundary curves are elements of $\W^s$
and $\W^u$ as well as the horizontal lines $\pm \pi/2 \mp 1/k_0^2$.
We construct the boxes so that each $B_j$ has diameter $\le \delta_0$ and
is foliated by a smooth family of stable curves
$\{W_\xi\}_{\xi \in E_j} \subset \W^s$, each of whose elements completely crosses $B_j$ in the
approximate stable direction.

We decompose the smooth measure $d\mu = \cos \vf dm$ on $B_j$ into
$d\mu = \hat \mu(d\xi) d\mu_\xi$,
where $\mu_\xi$ is the conditional measure of $\mu$ on $W_\xi$ and
$\hat \mu$ is the transverse measure on $E_j$.  We normalize
the measures so that $\mu_\xi(W_\xi) = \int_{W_\xi} \cos \vf \, dm_{W_\xi}$.	
Since the foliation is smooth,
$d\mu_\xi = \rho_\xi \cos \vf dm_{W_\xi}$ where $|\rho_\xi|_{\C^1(W_\xi)} \le C$ for
some constant $C$ independent of $\xi$.	 Note that
$\hat \mu(E_j) \le C \delta_0$
due to the transversality of curves in $\W^s$ and $\W^u$.
Next we choose on each homogeneity strip
$\Ho_t$, $t \ge k_0$, a smooth
foliation $\{W_\xi \}_{\xi \in E_t} \subset \W^s$ whose elements all have endpoints lying
in the two boundary curves of $\Ho_t$.  We again decompose $\mu$ on $\Ho_t$
into $d\mu = \hat \mu(d\xi) d\mu_\xi$, $\xi \in E_t$, and $d\mu_\xi = \rho_\xi \cos \vf dm_{W_\xi}$ is
normalized as above.  By construction, $\hat \mu(E_t) = \mathcal{O}(1)$.
Given $h \in \C^1(M)$, $\psi \in \C^p(T^{-n}\W^s)$, since $T^{-n}M = M$ {\em (mod 0)}, we have
$h(\psi) = \int_M h \psi \, d\mu = \int _M \Lp^n h \, \psi \circ T^{-n} \, d\mu$.  We split
$M = \cup_\ell M_\ell$ and integrate one $\ell$ at a time.
\[
\begin{split}
& \int_{M_\ell} \Lp^n h \, \psi \circ T^{-n} \, d\mu  = \sum_j \int_{B_j} \Lp^n h \, \psi \circ T^{-n} \, d\mu
+ \sum_{|t|\geq k_0} \int_{\Ho_t} \Lp^n h \, \psi \circ T^{-n} \, d\mu \\
& = \sum_j \int_{E_j} \int_{W_\xi} \Lp^nh \, \psi\circ T^{-n} \, \rho_\xi \, d\mu_W d\hat{\mu}(\xi)
+ \sum_{|t| \geq k_0} \int_{E_t} \int_{W_\xi} \Lp^nh \, \psi\circ T^{-n} \, \rho_\xi \, d\mu_W d\hat \mu(\xi)
\end{split}
\]
We change variables and estimate the integrals on one $W_\xi$ at a time.
Letting $W^n_{\xi,i}$ denote the components of $\G_n(W_\xi)$
defined in Section~\ref{preliminary} and recalling that $J_{W^n_{\xi,i}}T^n$ denotes the stable
Jacobian of $T^n$ along the curve $W^n_{\xi,i}$, we write,
\[
\begin{split}
\int_{W_\xi} & \Lp^n h \, \psi \circ T^{-n} \, \rho_\xi \, d\mu_W = \sum_i \int_{W^n_{\xi,i}}
h \psi (J_\mu T^n)^{-1} J_{W^n_{\xi,i}}T^n \rho_\xi \circ T^n \cos \vf \circ T^n \, dm_W \\
& \leq \sum_i |h|_w  |\psi|_{\C^p(W^n_{\xi,i})}
|(J_\mu T^n)^{-1} J_{W^n_{\xi,i}}T^n|_{\C^p(W^n_{\xi,i})}
| \rho_\xi \circ T^n|_{\C^p(W^n_{\xi,i})}  | \cos \vf \circ T^n|_{\C^p(W^n_{\xi,i})} .
\end{split}
\]
By \eqref{eq:C1 C0}, we have $| \rho_\xi \circ T^n|_{\C^p(W^n_{\xi, i})}
\le C |\rho_\xi|_{\C^p(W_\xi)} \le C$ for some uniform constant $C$.
The disortion bounds given by {\bf (H4)}, equation \eqref{eq:distortion stable}, imply that
\begin{equation}
\label{eq:holder c0}
| (J_\mu T^n)^{-1}J_{W^n_{\xi,i}}T^n|_{\C^p(W^n_{\xi,i})} \leq (1+2C_d) | (J_\mu T^n)^{-1}
J_{W^n_{\xi, i}}T^n|_{\C^0(W^n_{\xi, i})} .
\end{equation}

For $W \in \W^s$, let $\cos W_\xi$ denote the average value of $\cos \vf$ on $W$.  Note that
there exists $C_c >0$, depending only on $k_0$ and the uniform transversality of
$C^s(x)$ with the horizontal direction, such that
$C_c^{-1} \cos W \le \cos \vf(x) \le C_c \cos W$ for all $x \in W$.

Thus
$| \cos \vf \circ T^n|_\infty \le C_c \cos W_\xi$.  Then since
$|\cos \vf \circ T^n(x) - \cos \vf \circ T^n(y)| \le d_W(T^n x, T^n y)$ for $x,y \in W^n_{\xi,i}$,
we have $H^p_{W^n_{\xi,i}} (\cos \vf \circ T^n) \le C_e |W_\xi|^{1-p}$, where we have used
{\bf (H1)}.  If $W_\xi \subset \Ho_t$, then $\cos W_\xi \ge c t^{-2}$ while
$|W_\xi| \le C' t^{-3}$ for uniform constants $c, C' >0$, depending on the minimum angle
between $C^s(x)$ and the horizontal.  Thus since $p \le 1/3$, we have
$|\cos \vf \circ T^n |_{\C^p(W^n_{\xi,i})} \le C \cos W_\xi$ for some uniform constant $C$.

Gathering these estimates together, we have
\begin{equation}
\label{eq:single}
\int_{W_\xi}  \Lp^n h \, \psi \circ T^{-n} \, \rho_\xi \, d\mu_W
\le C |h|_w (|\psi|_\infty + H^p_n(\psi)) \cos (W_{\xi}) \sum_i
|(J_\mu T^n)^{-1} J_{W^n_{\xi,i}}T^n|_{\C^0(W^n_{\xi,i})},
\end{equation}
where $C$ is uniform in $T$ and $n$.
We group the pieces $W^n_{\xi,i} \in \G_n(W_\xi)$ according to most recent long ancestor
$W^k_{\xi, j} \in \G_k(W_\xi)$
as described in Section~\ref{preliminary}.  Then splitting up the Jacobians according to times
$k$ and $n-k$ and using {\bf (H5)}, we have
\begin{equation}
\label{eq:long short}
\begin{split}
\sum_i   & |(J_\mu T^n)^{-1} J_{W^n_{\xi,i}}T^n|_{\C^0(W^n_{\xi,i})}
 \le \sum_{i \in \I_n(W)}  \eta^n |J_{W^n_{\xi,i}}T^n|_{\C^0(W^n_{\xi,i})} \\
& \; \;  \; \; \; +  \sum_{k=1}^n \sum_{j \in L_k(W_\xi)} |(J_\mu T^k)^{-1}J_{W^k_{\xi,j}}T^k|_{\C^0(W^k_{\xi,j})}  \left( \sum_{i \in \I_n(W^k_j)} \eta^{n-k} |J_{W^n_{\xi,i}}T^{n-k}|_{\C^0(W^n_{\xi,i})}
 \right) \\
& \le C_1 (\eta \theta_*)^n
+ \sum_{k=1}^n \sum_{j \in L_k(W_\xi)} |(J_\mu T^k)^{-1}J_{W^k_{\xi,j}}T^k|_{\C^0(W^k_{\xi,j})}
 C_1 (\eta \theta_*)^{n-k}
\end{split}
\end{equation}
where we have used Lemma~\ref{lem:growth}(a) on each of the terms involving
$\I_n(W^k_{\xi,j})$ from time $k$ to time $n$.

For each $k$,  since $|W^k_{\xi,j}| \ge \delta_0/3$, we have by bounded distortion {\bf (H4)},
\[
\begin{split}
\sum_{j \in L_k(W_\xi)} |(J_\mu T^k)^{-1}J_{W^k_{\xi,j}}T^k|_{\C^0(W^k_{\xi,j})}
& \le (1+C_d)^2 3 \delta_0^{-1} \sum_{j \in L_k(W_\xi)} \int_{W^k_{\xi,j}} (J_\mu T^k)^{-1}
J_{W^k_{\xi,j}}T^k \, dm_W \\
& \le C \delta_0^{-1} \int_{W_\xi} (J_\mu T^k)^{-1} \, dm_W .
\end{split}
\]
Putting this estimate together with \eqref{eq:single} and \eqref{eq:long short}
and bringing $\cos W_\xi$ into the integral,
\[
\int_{W_\xi}  \Lp^n h \, \psi \circ T^{-n} \, \rho_\xi \, d\mu_W
\leq C |h|_w (|\psi|_\infty + H^p_n(\psi))
\Big(\cos W_\xi + \sum_{k=1}^n (\eta \theta_*)^{n-k} \int_{W_\xi} (J_\mu T^k)^{-1} \, d\mu_W  \Big)
\]
for some uniform constant $C$.
Thus
\[
\begin{split}
\Big| \int_{M_\ell} \Lp^n h \, & \psi \circ T^{-n} \, dm \Big|
\le C |h|_w (|\psi|_\infty + H^p_n(\psi))
\Big( \sum_j \int_{E_j} \cos W_\xi \, \hat \mu(d\xi)	+ \sum_{|t| \geq k_0} \int_{E_t}
\cos W_\xi \, \hat \mu(d \xi) \\
& \qquad \qquad +  \sum_j \sum_{k=1}^n (\eta \theta_*)^{n-k} \int_{B_j} (J_\mu T^k)^{-1} \, d\mu + \sum_{|t| \geq k_0} \sum_{k=1}^n (\eta \theta_*)^{n-k} \int_{\Ho_t} (J_\mu T^k)^{-1} \, d\mu \\
& \! \! \! \! \! \le C |h|_w ( |\psi|_\infty + H^p_n(\psi)) \Big( \sum_j \hat \mu(E_j)
+ \sum_{|t| \geq k_0} t^{-2} \hat \mu(E_t) + \sum_{k=1}^n (\eta \theta_*)^{n-k}
\int_{M_\ell} (J_\mu T^k)^{-1} d\mu  \Big)
\end{split}
\]
where in the last line we have used the fact that
$\cos W \leq Ct^{-2}$ for $W \subset \Ho_t$.	The first two sums are finite since there are only
finitely many $E_j$ and $\hat \mu(E_t)$ is of order $1$ for each $t$.
Since there are only finitely many $M_\ell$, the first two sums remain finite when we sum over
$\ell$.  For the third sum, we sum over $\ell$ and use the fact that $\int_M (J_\mu T^k)^{-1} \, d\mu = 1$
for each $k \ge 1$.  Thus the contribution from the third sum is uniformly bounded in
$n$ using the fact that $\eta \theta_* <1$ by {\bf (H5)}.
\end{proof}

Several other properties of the spaces $\B$ and $\B_w$ proved in
\cite{demers zhang} do not need to be reproved since their proofs remain essentially
unchanged.  They are as follows.
\begin{itemize}
  \item[(i)] (\cite[Lemma 3.7]{demers zhang})
  $\B$ contains piecewise H\"older continuous functions $h$ with exponent greater than $2\beta$ provided
the discontinuities of $h$ are uniformly transverse to the stable cones $C^s(x)$.  
  \item[(ii)] (\cite[Lemma 2.1]{demers zhang})
$\Lp$ is well-defined as a continuous linear operator on both $\B$ and $\B_w$.
Moreover, there is a sequence of embeddings
$\C^\gamma(M) \hookrightarrow \B \hookrightarrow \B_w \hookrightarrow (\C^p(M))'$,
for all $\gamma > 2\beta$.
  \item[(iii)] (\cite[Lemma 3.10]{demers zhang})
  The unit ball of $(\B, \| \cdot \|_\B)$ is compactly embedded in $(\B_w, | \cdot |_w)$.
\end{itemize}
Lemma~\ref{lem:distr} and items (i) and (ii)
characterize the spaces $\B$ and $\B_w$ as spaces of distributions
containing all H\"older continuous and certain classes of piecewise H\"older continuous
functions.  The last item is necessary in order to deduce the quasi-compactness of
$\Lp_T$ from the Lasota-Yorke inequalities given by Theorem~\ref{thm:uniform}.

There remains one final fact to establish.  As mentioned earlier, since we
identify $h \in \C^1(M)$ with the measure $h\mu$ as an element of $\B$, {\em a priori}
Lebesgue measure may not be in $\B$.  The following Lemma shows that
Lebesgue measure is in fact in $\B$ and therefore so is $h dm$ for any $h \in \C^1(M)$.

\begin{lemma}
\label{lem:lebesgue}
The function $(\cos \vf)^{-1}$ is in  $\B$.  Therefore, Lebesgue measure
$m = (\cos \vf)^{-1} \,\mu$ is also in $\B$ and so is $h m$ for any $h \in \C^1(M)$.  Indeed,
any piecewise H\"older continuous function as in item (i) above times Lebesgue belongs to $\B$.
\end{lemma}

\begin{proof}
In order to show $(\cos \vf)^{-1} \in \B$, we must show that $(\cos \vf)^{-1}$ can be
approximated by functions $h \in \C^1(M)$ in the $\| \cdot \|_\B$ norm.
Since $\| f \|_\B = \sup_{k} \| f|_{\Ho_k} \|_\B$,
our strategy will be to show that $\| (\cos \vf)^{-1}|_{\Ho_k} \|_\B \le Ck^{-1/2}$ for some uniform
constant $C$.  We can then approximate $(\cos \vf)^{-1}$ by $0$ in homogeneity strips
of sufficiently high index.  More precisely, given $\ve >0$, we choose $K$ such that
$CK^{-1/2} < \ve$.  Then on the remaining strips $k < K$,
$(\cos \vf)^{-1}$ has finite $\C^1$-norm
and satisfies the assumptions of \cite[Lemma 3.7]{demers zhang}.  Thus we may find
$f_\ve \in C^1(M)$ as in the proof of that lemma such that
\[
\| (\cos \vf)^{-1} - f_\ve \|_\B \le \sup_{k \ge K} \| (\cos \vf)^{-1} |_{\Ho_k} \|_\B
+ \sup_{k < K} \| ((\cos \vf)^{-1} - f_\ve)|_{\Ho_k} \|_B < 2 \ve ,
\]
proving that $(\cos \vf)^{-1} \in \B$.

It remains to prove the claim $\| (\cos \vf)^{-1}|_{\Ho_k} \|_\B \le Ck^{-1/2}$.
Choose $W \in \W^s$, $W \subset \Ho_k$, and let $\psi \in \C^q(W)$ with
$|\psi|_{W, \alpha , q} \le 1$.  Then,
\[
\int_W (\cos \vf)^{-1} \, \psi \, dm_W \le |(\cos \vf)^{-1} |_{\C^0(W)} |\psi|_{\C^0(W)} |W|
\le |(\cos \vf)^{-1} |_{\C^0(W)} |W|^{1-\alpha}
\]
since $|\psi|_{\C^p(W)} \le |W|^{-\alpha}$.  On $\Ho_k$, we have $|W| \le ck^{-3}$
and $|\cos\vf |^{-1}  \le Ck^2$ for some uniform constants $c$ and $C$ which depend only on the
minimum angle of $C^s(x)$ with the horizontal.  Then, since $\alpha < 1/6$,
\begin{equation}
\label{eq:k balance}
|(\cos \vf)^{-1} |_{\C^0(W)} |W|^{1-\alpha} \le cC k^2 k^{-3+3 \alpha} \le C'k^{-1/2} .
\end{equation}
Taking the suprema over $W \subset \Ho_k$ and $\psi$ with $|\psi|_{W,\alpha, q} \le 1$, we have
$\| (\cos \vf)^{-1} |_{\Ho_k} \|_s \le C' k^{-1/2}$, completing the estimate on the strong stable norm.

To estimate the strong unstable norm, let $\ve \le \ve_0$ and choose two curves in
$\Ho_k$, $W^1, W^2 \in \W^s$, such that $d_{\W^s}(W^1, W^2) \le \ve$.  For $i=1, 2$,
let $\psi_i \in \C^p(W^i)$ with $|\psi_i|_{\C^p(W^i)} \le 1$ and $d_q(\psi_1, \psi_2) \le \ve$.

Recalling the notation of Section~\ref{norms}, denote $W^i =  \{ G_{W^i}(r) = (r, \vf_{W^i}(r)) : r \in I_{W^i} \}$,
$i = 1, 2$ and note that
by definition of $d_{\W^s}(\cdot, \cdot)$, $W^1$ and $W^2$ can be put into one-to-one
correspondence by a foliation of vertical line segments of length at most $\ve$,
except possibly near their endpoints.  Denote by $U^i$ the single matched connected component
of $W^i$ and by $V^i_j$ the at most 2 unmatched components of $W^i$.
We let $\Theta: U^1 \to U^2$ denote the holonomy map along the vertical
foliation.  We estimate,
\begin{equation}
\label{eq:cos split}
\begin{split}
\int_{W^1} (\cos \vf)^{-1} \psi_1 \, dm_W - \int_{W^2} & (\cos \vf)^{-1} \psi_2 \, dm_W
  = \sum_{i,j} \int_{V^i_j} (\cos \vf)^{-1} \psi_i \, dm_W \\
& \;\; \; +  \int_{U^1} (\cos \vf)^{-1} \psi_1 \, dm_W - \int_{U^2} (\cos \vf)^{-1} \psi_2 \, dm_W.
\end{split}
\end{equation}
We first estimate over the unmatched pieces $V^i_j$.  Note that $|V^i_j| \le C\ve$ where
$C$ depends only on the minimum angle of $C^s(x)$ with the vertical.
Recalling that $| \psi_i |_{\C^0(W^i)} \le 1$ and using
\eqref{eq:k balance} since $\beta \le \alpha$, we estimate
\begin{equation}
\label{eq:cosine unmatched}
\left| \sum_{i,j} \int_{V^i_j} (\cos \vf)^{-1} \psi_i \, dm_W \right|
\le \sum_{i,j} |V^i_j|^\beta |V^i_j|^{1-\beta} |(\cos \vf)^{-1}|_{C^0(V^i_j)} \le C \ve^\beta k^{-1/2} .
\end{equation}

To estimate the difference on the matched pieces $U^i$, we change variables to $U^1$ using
$\Theta$,
\[
\begin{split}
 \int_{U^1} (\cos \vf)^{-1} \psi_1 \, dm_W & - \int_{U^2} (\cos \vf)^{-1} \psi_2 \, dm_W
 = \int_{U^1} (\cos \vf)^{-1} \psi_1 - [(\cos \vf)^{-1}  \psi_2] \circ \Theta \, J\Theta \, dm_W \\
 & \le |U^1| |(\cos \vf)^{-1} \psi_1 - [(\cos \vf)^{-1}  \psi_2] \circ \Theta \, J\Theta|_{\C^0(U^1)} .
 \end{split}
 \]
To estimate the $\C^0$ norm of the test function, we split the difference into 3 terms
and use the fact that $|\psi_i|_{\C^0} \le 1$,
\begin{equation}
\label{eq:cosine test}
\begin{split}
|(\cos \vf)^{-1} \psi_1 & - [(\cos \vf)^{-1}  \psi_2] \circ \Theta \, J\Theta|_{\C^0(U^1)}
\le |(\cos \vf)^{-1} - (\cos \vf)^{-1} \circ \Theta|_{\C^0(U^1)} \\
& + |(\cos \vf)^{-1}|_{\C^0(U^2)} |\psi_1 - \psi_2 \circ \Theta|_{\C^0(U^1)}
+ |(\cos \vf)^{-1}|_{\C^0(U^2)} |1 - J\Theta|_{\C^0(U^1)}.
\end{split}
\end{equation}
For the first term above, note that for $x \in U^1$,
$| \cos \vf(x) - \cos \vf \circ \Theta(x) | \le d(x, \Theta(x)) \le \min \{ \ve , Ck^{-3} \}$ for
some uniform constant $C>0$.  Thus
\[
 |(\cos \vf)^{-1}(x) - (\cos \vf)^{-1} \circ \Theta(x)| \le
 \frac{d(x, \Theta(x))}{\cos \vf(x) \cos \vf \circ \Theta(x)} \le C' \frac{\ve^\beta k^{-3(1-\beta)}}{k^{-4}}
 \le C' \ve^\beta k^{3/2} ,
\]
since $\beta \le 1/6$.  To estimate the second term in \eqref{eq:cosine test}, denote $x \in U^1$
by $x = G_{W^1}(r)$ for some $r \in I_{W^1} \cap I_{W^2} =: I$.  Then
$|\psi_1(x) - \psi_2 \circ \Theta(x)| = |\psi_1 \circ G_{W^1}(r) - \psi_2 \circ G_{W^2(r)}| \le \ve$
by definition of $d_q(\cdot , \cdot)$.  Thus
\[
|(\cos \vf)^{-1}|_{\C^0(U^2)} |\psi_1 - \psi_2 \circ \Theta|_{\C^0(U^1)}
\le Ck^2 \ve .
\]
Finally, we estimate the third term of \eqref{eq:cosine test} by noting that
\[
|1-J\Theta| = \left| 1-  \frac{\sqrt{1 + (\vf'_{W^1})^2}}{\sqrt{1 + (\vf'_{W^2})^2}}  \right|
\le |\vf'_{W^1} - \vf'_{W^2}| \le \ve,
\]
where we have used the fact that the derivative of $\sqrt{1+ t^2}$,  $\frac{t}{\sqrt{1+t^2}}$, is
bounded by 1 for $t \geq 0$.

Putting these 3 estimates together in \eqref{eq:cosine test}, we estimate the norm on the
matched pieces by
\[
\left| \int_{U^1} (\cos \vf)^{-1} \psi_1 \, dm_W  - \int_{U^2} (\cos \vf)^{-1} \psi_2 \, dm_W \right|
\le |U^1| C(\ve^\beta k^{3/2} + \ve k^2 + \ve k^2) \le C' \ve^\beta k^{-1},
\]
using the fact that $|U^1| \le Ck^{-3}$.  This, combined with \eqref{eq:cosine unmatched},
yields the required estimate on the strong unstable norm.
\end{proof}


\section{Proof of Theorem~\ref{thm:uniform}}
\label{uniform}

The proof of Theorem~\ref{thm:uniform} relies on the following proposition.

\begin{proposition}
\label{prop:ly}
There exists $C>0$, depending only on {\bf (H1)}-{\bf(H5)}, such that for any
$T \in \F$, $h \in \B$ and $n \ge 0$,
\begin{eqnarray}
|\Lp_T^n h|_w & \le & C \eta^n |h|_w   \label{eq:weak norm} \\
\| \Lp_T^n h \|_s & \le &C ( \theta_1^{(1-\alpha)n} + \Lambda^{-qn}) \eta^n \|h\|_s + C \eta^n \delta_0^{-\alpha}|h|_w
\label{eq:stable norm} \\
\| \Lp_T^n h \|_u &\le & C\eta^n  \Lambda^{-\beta n} \| h\|_u + C \eta^n C_3^n \|h \|_s
\label{eq:unstable norm}
\end{eqnarray}
\end{proposition}

\begin{proof}[Proof of Theorem~\ref{thm:uniform} given Proposition~\ref{prop:ly}]
Choose $1 > \sigma > \eta \max \{ \theta_1^{1-\alpha}, \Lambda^{-q}, \Lambda^{-\beta} \}$
and choose $N \ge 0$ such that
\[
\begin{split}
\| \Lp_T^N h \|_\B &  = \| \Lp^N_T h \|_s + b \| \Lp_T^N h \|_u
\le \frac{\sigma^N}{2} \| h \|_s + C \delta_0^{-\alpha} \eta^N |h|_w + b \sigma^N \| h \|_u + b C\eta^N C_3^N \|h\|_s \\
& \leq \sigma^N \| h \|_\B + C_{\delta_0} \eta^N |h|_w
\end{split}
\]
providing $b$ is chosen sufficiently small with respect to $N$.  This is the required
inequality \eqref{eq:uniform LY}
for Theorem~\ref{thm:uniform} which implies the essential spectrum
of $\Lp_T$ is less than $\sigma$.  Outside the disk of radius $\sigma$, the spectrum
of $\Lp_T$ has finitely many eigenvalues, each with finite multiplicity.  This follows
using the compactness of the unit ball of $\B$ in $\B_w$ \cite[Lemma 3.10]{demers zhang}.

Despite the fact that $\eta$ may be greater than 1, the spectral radius of
$\Lp_T$ equals 1.  To see this, suppose $z \in \mathbb{C}$, $|z|>1$, satisfies
$\Lp_T h = z h$ for some $h \in \B$, $h \neq 0$.  For $\psi \in \C^p(M)$,
Lemma~\ref{lem:distr} implies that,
\[
|h(\psi)| = |z^{-n} \Lp_T^n h(\psi)| = |z^{-n} h(\psi \circ T^n)| \le |z|^{-n} C |h|_w (|\psi|_\infty + H^p_n(\psi \circ T^n)) \xrightarrow[n \to \infty]{} 0
\]
since $H^p_n(\psi \circ T^n) \le C_e \Lambda^{-pn} |\psi|_{\C^p(M)}$ by \eqref{eq:C1 C0}.
Thus $h = 0$, contradicting the assumption on $z$.

The characterization of the peripheral spectrum follows from
Lemmas 5.1 and 5.2 of \cite{demers zhang}.
\end{proof}

To prove Proposition~\ref{prop:ly}, we fix $T \in \F$ and prove the required Lasota-Yorke inequalities
\eqref{eq:weak norm}-\eqref{eq:unstable norm}.
It is shown in \cite[Section 4]{demers zhang} that $\Lp_T$ is a continuous operator
on both $\B$ and $\B_w$
so that it suffices to prove the inequalities for $h \in \C^1(M)$.  They extend to the
completions by continuity.
Since these estimates are similar to those in \cite{demers zhang}, our purpose
in repeating them is to show how they depend explicitly on the uniform constants
given by {\bf (H1)}-{\bf (H5)} and do not require additional information.


\subsection{Estimating the weak norm}
\label{weak norm}

Let $h \in \C^1(M)$, $W \in \W^s$ and $\psi \in \C^p(W)$ such that
$|\psi|_{W,0,p} \leq 1$.         For $n \geq0$, we write,
\begin{equation}
\label{eq:start}
\int_W\Lp^nh \, \psi \, dm_W
=\sum_{W^n_i \in\G_n(W)}\int_{W^n_i}h\frac{J_{W^n_i}T^n}{J_\mu T^n}\psi \circ T^n dm_W
\end{equation}
where $J_{W^n_i}T^n$ denotes the Jacobian of $T^n$ along $W^n_i$.

Using the definition of the weak norm on each $W^n_i$, we estimate \eqref{eq:start} by
\begin{equation}
\label{eq:weak estimate}
\int_W\Lp^nh \, \psi \, dm_W \; \leq \; \sum_{W^n_i \in\G_n} |h|_w
| (J_\mu T^n)^{-1}J_{W^n_i}T^n|_{\C^p(W^n_i)} |\psi\circ T^n|_{\C^p(W^n_i)} .
\end{equation}
For $x,y \in W^n_i$, we use {\bf (H1)} to estimate,
\begin{equation}
\label{eq:C1 C0}
\frac{|\psi (T^nx) - \psi (T^ny)|}{d_W(T^nx,T^ny)^p}\cdot \frac{d_W(T^nx,T^ny)^p}{d_W(x,y)^p} \leq |\psi|_{\C^p(W)} |J_{W^n_i}T^n|^p_{\C^0(W^n_i)}
    \leq C_e \Lambda^{-pn} |\psi|_{\C^p(W)} ,
\end{equation}
so that $|\psi \circ T^n|_{\C^p(W^n_i)} \leq C_e |\psi|_{\C^p(W)} \le C_e$.
We use this estimate together with {\bf (H5)} and \eqref{eq:holder c0} to bound
\eqref{eq:weak estimate} by
\[
\int_W \Lp^nh \, \psi \, dm_W \leq C_e (1+2C_d) \eta^n  |h|_w \sum_{W^n_i \in \G_n}
|J_{W^n_i}T^n|_{\C^0(W^n_i)} \le C' \eta^n |h|_w,
\]
where $C' = C_e (1+2C_d) C_2$ and we have used Lemma~\ref{lem:growth}(b)
for the last inequality.
Taking the supremum over all $W \in \W^s$ and $\psi \in \C^p(W)$ with
$|\psi|_{W,0,p} \leq 1$
yields \eqref{eq:weak norm} expressed with uniform constants given by {\bf (H1)}-{\bf (H5)}.


\subsection{Estimating the strong stable norm}
\label{stable norm}

Let $W \in \W^s$ and let $W^n_i$ denote the elements of $\G_n(W)$ as defined above.
For $\psi \in \C^q(W)$, $|\psi|_{W,\alpha, q} \le 1$, define
$\bp_i = |W^n_i|^{-1} \int_{W^n_i} \psi \circ T^n \, dm_W$.
Using equation \eqref{eq:start}, we write
\begin{equation}
\label{eq:stable split}
\int_W\Lp^nh\, \psi \, dm_W      =
\sum_{i} \int_{W^n_i}h \, \frac{J_{W^n_i}T^n}{J_\mu T^n} \,(\psi \circ T^n- \bp_i) \, dm_W
 + \bp_i  \int_{W^n_i}h \, \frac{J_{W^n_i}T^n}{J_\mu T^n} \, dm_W .
\end{equation}

To estimate the first term of \eqref{eq:stable split},
we first estimate $|\psi \circ T^n - \bp_i|_{\C^q(W^n_i)}$.
If $H_W^q(\psi)$ denotes the H\"older constant of $\psi$ along $W$, then
equation~\eqref{eq:C1 C0} implies
\begin{equation}
\label{eq:H^q}
\frac{|\psi (T^nx) - \psi (T^ny)|}{d_W(x,y)^q} \leq C_e \Lambda^{-nq}  H_W^q(\psi)
\end{equation}
for any $x,y \in W^n_i$.  Since $\bp_i$ is constant on $W^n_i$, we have
$H^q_{W^n_i}(\psi \circ T^n - \bp_i) \leq C_e \Lambda^{-qn} H^q_W(\psi)$.
To estimate the $\C^0$ norm, note that $\bp_i = \psi \circ T^n(y_i)$ for some
$y_i \in W^n_i$.  Thus for each $x \in W^n_i$,
\[
| \psi \circ T^n(x) - \bp_i|
    = |\psi \circ T^n(x) - \psi \circ T^n(y_i)|
    \leq H^q_{W^n_i}(\psi \circ T^n) |W^n_i|^q \leq C_e H^q_W(\psi) \Lambda^{-nq}  .
\]
This estimate together with \eqref{eq:H^q} and the fact that $|\vf|_{W,\alpha,q} \leq 1$, implies
\begin{equation}
\label{eq:C^q small}
|\psi \circ T^n - \bp_i|_{\C^q(W^n_i)} \leq C_e \Lambda^{-nq} |\psi|_{\C^q(W)}
\leq C_e \Lambda^{-qn} |W|^{-\alpha} .
\end{equation}

We apply \eqref{eq:holder c0}, \eqref{eq:C^q small}
and the definition of the strong stable norm to the first term of \eqref{eq:stable split},
\begin{equation}
\label{eq:first stable}
\begin{split}
\sum_i \int_{W^n_i} h & \frac{J_{W^n_i}T^n}{J_\mu T^n}  \, (\psi \circ T^n - \bp_i) \, dm_W  \leq
(1+2C_d) C_e \sum_i \|h\|_s \frac{|W^n_i|^\alpha}{|W|^\alpha}
\left| \frac{J_{W^n_i}T^n}{J_\mu T^n} \right|_{C^0(W^n_i)} \Lambda^{-qn}    \\
& \leq \; \eta^n (1+2C_d) C_e \Lambda^{-qn} \|h\|_s \sum_i \frac{|W^n_i|^\alpha}{|W|^\alpha}
|J_{W^n_i}T^n|_{\C^0(W^n_i)}
\; \leq \; C_4 \eta^n \Lambda^{-qn} \|h\|_s ,
\end{split}
\end{equation}
where $C_4 = (1+2C_d) C_e C_2^{1-\alpha}$ and
in the second line we have used {\bf (H5)} and
Lemma~\ref{lem:growth}(c) with $\varsigma = \alpha$.

For the second term of \eqref{eq:stable split}, we use the fact that
$|\bp_i| \leq |W|^{-\alpha} $ since $|\psi|_{W, \alpha, q} \le 1$.
Recall the notation introduced before the statement of
Lemma~\ref{lem:growth}.
Grouping the pieces $W^n_i \in \G_n(W)$
according to most recent long ancestors $W^k_j \in L_k(W)$, we have
\[
\begin{split}
\sum_{i}  |W|^{-\alpha} \int_{W^n_i}h \frac{J_{W^n_i}T^n}{J_\mu T^n} \, dm_W
= & \sum_{k=1}^n \sum_{j\in L_k(W)}\sum_{ i\in \I_n(W^k_j)}
     |W|^{-\alpha} \int_{W^n_i}h \frac{J_{W^n_i}T^n}{J_\mu T^n} \, dm_W \\
&  + \sum_{ i\in \I_n(W)}
     |W|^{-\alpha} \int_{W^n_i} h \frac{J_{W^n_i}T^n}{J_\mu T^n} \, dm_W
\end{split}
\]
where we have split up the terms involving $k=0$ and $k \geq 1$.
We estimate the terms with $k \geq 1$
by the weak norm and the terms with $k=0$ by the strong stable norm.
Using again \eqref{eq:holder c0} and {\bf (H5)},
\[
\begin{split}
\sum_{i}  |W|^{-\alpha} \int_{W^n_i}h \frac{J_{W^n_i}T^n}{J_\mu T^n} \, dm_W
& \leq
\eta^n (1+2C_d) \sum_{k=1}^n\sum_{j\in L_k(W)}
\sum_{i\in \I_n(W^k_j)} |W|^{-\alpha} |h|_w | J_{W^n_i}T^n |_{\C^0(W^n_i)}\\
&  +
\eta^n (1+2C_d) \sum_{\ i\in \I_n(W)} \frac{|W^n_i|^\alpha }{|W|^\alpha } \|h\|_s |J_{W^n_i}T^n|_{\C^0(W^n_i)} .
\end{split}
\]

In the first sum above corresponding to $k\geq 1$, we write
\[
|J_{W^n_i}T^n|_{\C^0(W^n_i)} \leq |J_{W^n_i}T^{n-k}|_{\C^0(W^n_i)}
|J_{W^k_j}T^k|_{\C^0(W^k_j)} .
\]
Thus using Lemma~\ref{lem:growth}(a) from time $k$ to time $n$,
\[
\begin{split}
\sum_{k=1}^n \sum_{j \in L_k} \sum_{i \in \I_n(W^k_j)} |W|^{-\alpha} |J_{W^n_i}T^n|_{\C^0(W^n_i)}
& \leq \sum_{k=1}^n \sum_{j \in L_k(W)} |J_{W^k_j}T^k|_{\C^0(W^k_j)} |W|^{-\alpha}
\sum_{i  \in \I_n(W^k_j)} |J_{W^n_i}T^{n-k}|_{\C^0(W^n_i)} \\
& \le 3 \delta_0^{-\alpha}  \sum_{k=1}^n \sum_{j \in L_k(W)} |J_{W^k_j}T^k|_{\C^0(W^k_j)} \,
\frac{|W^k_j|^\alpha}{|W|^\alpha} C_1 \theta_*^{n-k},
\end{split}
\]
since $|W^k_j| \ge \delta_0/3$.
The inner sum is bounded by $C_2^{1-\alpha}$ for each $k$ by Lemma~\ref{lem:growth}(c)
while the outer sum is bounded by $C_1/(1-\theta_*)$ independently of $n$.

Finally, for the sum corresponding to $k=0$, since
\[
|J_{W^n_i}T^n|_{\C^0(W^n_i)} \le (1+C_d)  |T^nW^n_i| |W^n_i|^{-1} \le (1+C_d) |J_{W^n_i}T^n|_{\C^0(W^n_i)},
\]
we use Jensen's inequality and Lemma~\ref{lem:growth}(a) to estimate,
\[
\sum_{i \in \I_n(W)} \frac{|W^n_i|^\alpha}{|W|^\alpha} |J_{W^n_i}T^n|_{\C^0(W^n_i)}
\le (1+C_d)
\left( \sum_{i \in \I_n(W)} \frac{|T^nW^n_i|}{|W^n_i|} \right)^{1-\alpha}
\le (1+C_d)C_1 \theta_*^{n(1-\alpha)}.
\]

Gathering these estimates together, we have
\begin{equation}
\label{eq:second stable}
\sum_{i}  |W|^{-\alpha}\left| \int_{W^n_i}h(J_\mu T^n)^{-1}J_{W^n_i}T^n \, dm_W\right|
\; \leq \; C_5 \eta^n \delta_0^{-\alpha}|h|_w + C_6 \|h\|_s \eta^n \theta_*^{n(1-\alpha)} ,
\end{equation}
where $C_5 = 3(1+2C_d) C_1C_2^{1-\alpha} /(1-\theta_*)$ and
$C_6 = (1+2C_d)^2 C_1$.
Putting together \eqref{eq:first stable} and \eqref{eq:second stable} proves
\eqref{eq:stable norm},
\[
\|\Lp^n h\|_s \leq C'  \eta^n \left( \Lambda^{-q n}+\theta_*^{n(1-\alpha)}\right)\|h\|_s
         + C' \eta^n \delta_0^{-\alpha} |h|_w ,
\]
with $C' = \max \{ C_4, C_5, C_6 \}$, a uniform constant depending only on
{\bf (H1)}-{\bf (H5)}.


\subsection{Estimating the strong unstable norm}
\label{unstable norm}

Fix $\ve \le \ve_0$ and
consider two curves $W^1, W^2 \in\W^s$ with $d_{\W^s}(W^1,W^2) \leq \ve$.
For $n \geq 1$, we describe how to partition $T^{-n}W^\ell$ into
``matched'' pieces $U^\ell_j$ and
``unmatched'' pieces $V^\ell_k$, $\ell=1,2$.
In the what follows, we use $C_t$ to
denote a transversality constant which depends only on the minimum angle
between various transverse directions:  the minimum angle between $C^s(x)$ and
$C^u(x)$, between $S^T_{-n}$ and $C^s(x)$, and
between $C^s(x)$ and the vertical and horizontal directions.

Let $\omega$ be a connected component of $W^1 \setminus \Si_{-n}^T$
such that $T^{-n}\omega \in \G_n(W)$.
To each
point $x \in T^{-n}\omega$, we associate a smooth curve
$\gamma_x \in \widehat\W^u$ of length
at most $C_t C_e \Lambda^{-n}\ve$ such that its image $T^n\gamma_x$,
if not cut by a singularity or the boundary of a homogeneity strip,
will have length $C_t \ve$.
By {\bf (H2)}, $T^i\gamma_x \in \widehat\W^u$ for each $i \ge 0$.

Doing this for each connected component of
$W^1 \setminus \Si_{-n}^T$, we subdivide $W^1 \setminus \Si_{-n}^T$ into
a countable
collection of subintervals of points for which $T^n\gamma_x$ intersects
$W^2 \setminus \Si_{-n}^T$ and subintervals for which this is not the case.
This in turn induces a corresponding partition on
$W^2 \setminus \Si_{-n}^T$.

We denote by $V^\ell_k$ the pieces in $T^{-n}W^\ell$ which are not matched up by this process
and note that
the images $T^nV^\ell_k$ occur either at the endpoints of $W^\ell$ or because the
curve $\gamma_x$ has been cut by a singularity.
In both cases, the length of the
curves $T^nV^\ell_k$ can be at most $C_t \ve$ due to the uniform transversality of
$\Si_{-n}^T$ with $C^s(x)$ and of $C^s(x)$ with $C^u(x)$.

In the remaining
pieces the foliation $\{ T^n\gamma_x \}_{x \in T^{-n}W^1}$ provides a one to one correspondence
between points in $W^1$ and $W^2$.
We partition these pieces in
such a way that the lengths of their images under $T^{-i}$
are less than $\delta_0$ for each $0 \le i \le n$ and
the pieces are pairwise matched by
the foliation $\{\gamma_x\}$. We call these matched pieces $\widetilde U^\ell_j$
and note that $T^i \widetilde U^\ell_j \in \G_{n-i}(W^\ell)$ for each $i = 0, 1, \ldots n$.
For convenience, we further trim the $\widetilde U^\ell_j$ to pieces
$U^\ell_j$ so that $U^1_j$ and $U^2_j$ are both defined on the same arclength
interval $I_j$.  The at most two components of $T^n(\widetilde U^\ell_j \setminus
U^\ell_j)$ have length less than  $C_t \ve$ due to the uniform transversality
of $C^s(x)$ with the vertical direction.  We attach these trimmed pieces to
the adjacent $U^\ell_i$ or $V^\ell_k$ as appropriate so as not to create any
additional components in the partition.

We further relabel any pieces $U^\ell_j$ as $V^\ell_j$
and consider them unmatched if for some $i$, $0 \le i \le n$, $|T^i\gamma_x| > 2 |T^iU^\ell_j|$.
i.e. we only consider pieces matched if at each intermediate step, the distance between
them is at most of the same order as their length.  We do this in order to be able to
apply Lemma~\ref{lem:angles} to the matched pieces.  Notice that since the distance between
the curves at each intermediate step is at most $C_t C_e \ve$ and due to the uniform
contraction of stable curves going forward, we have $|T^nV^\ell_k| \le C_t C_e^2 \ve$ for
all such pieces considered unmatched by this last criterion.

In this way we write $W^\ell = (\cup_j T^nU^\ell_j) \cup (\cup_k T^nV^\ell_k)$.
Note that the images $T^nV^\ell_k$ of the unmatched pieces must have length $\le C_v \ve$
for some uniform constant $C_v$
while the images of the matched pieces
$U^\ell_j$
may be long or short.

Recalling the notation of Section~\ref{norms}, we have arranged a pairing
of the pieces $U^\ell_j$ with the following property:
\begin{equation}
\label{eq:match}
\begin{split}
\mbox{If } \; &
U^1_j = G_{U^1_j}(I_j) = \{ (r, \vf_{U^1_j}(r)) : r \in I_j \}, \\
\mbox{then } \; & U^2_j = G_{U^2_j}(I_j) = \{ (r, \vf_{U^2_j}(r)) : r \in I_j \},
\end{split}
\end{equation}
so that the
point $x = (r, \vf_{U^1_j}(r)) \in U^1_j$ can associated with the point
$\bar x = (r, \vf_{U^2_j}(r)) \in U^2_j$ by the vertical
line $\{(r,s)\}_{s\in[-\pi/2, \pi/2]}$, for each $r \in I_j$.  In addition, the $U^\ell_j$ satisfy
the assumptions of Lemma~\ref{lem:angles}.

Given $\psi_\ell$ on $W^\ell$ with $|\psi_\ell|_{W^\ell,0,p} \leq 1$ and
$d_q(\psi_1, \psi_2) \leq \ve$,
with the above construction we must estimate
\begin{align}
\label{eq:unstable split}
& \left|\int_{W^1} \Lp^nh \, \psi_1 \, dm_W - \int_{W^2} \Lp^nh \, \psi_2 \, dm_W \right|
  \; \leq \; \sum_{\ell,k} \left|\int_{V^\ell_k} h (J_\mu T^n)^{-1}J_{V^\ell_k}T^n \psi_\ell \circ T^n \, dm_W \right|\nonumber\\
  & + \sum_j \left| \int_{U^1_j} h (J_\mu T^n)^{-1}J_{U^1_j}T^n \psi_1\circ T^n \, dm_W
    - \int_{U^2_j} h (J_\mu T^n)^{-1}J_{U^2_j}T^n \psi_2\circ T^n \, dm_W \right|
\end{align}
We do the estimate over the unmatched pieces $V^\ell_k$ first
using the strong stable norm.  Note that
by \eqref{eq:C1 C0}, $|\psi_\ell \circ T^n|_{\C^q(V^\ell_k)}
\leq C_e |\psi_\ell|_{\C^p(W^\ell)} \leq C_e $.
We estimate as
in Section~\ref{stable norm}, using the fact that $|T^nV^\ell_k| \le C_v \ve$, as noted above,
\begin{equation}
\label{eq:first unstable}
\begin{split}
 \sum_{\ell,k} \Big| \int_{V^\ell_k}h (J_\mu T^n)^{-1}J_{V^\ell_k}T^n\psi_\ell\circ T^n \, & dm_W \Big|
\leq
C_e \sum_{\ell,k} \|h\|_s |V^\ell_k|^\alpha |(J_\mu T^n)^{-1}J_{V^\ell_k}T^n|_{\C^q(V^\ell,k)}  \\
& \leq C_e (1+2C_d) \eta^n
\| h \|_s \sum_{\ell,k} |V^\ell_k|^\alpha |J_{V^\ell_k}T^n|_{\C^0(V^\ell_k)} \\
& \leq C' \ve^{\alpha} \eta^n \|h\|_s \sum_{\ell,k} |J_{V^\ell_k}T^n|_{\C^0(V^\ell_k)}^{1-\alpha}
\le 2 C' \ve^\alpha \eta^n \|h\|_s C_3^n ,
\end{split}
\end{equation}
with $C' = C_e (1+2C_d)^2 C_v^\alpha$,
where we have applied Lemma~\ref{lem:growth}(d)
with $\varsigma = 1-\alpha > \varsigma_0$
since there are at most two $V^\ell_k$ corresponding to
each element $W^{\ell, n}_i \in \G_n(W^\ell)$ as defined in Section~\ref{preliminary} and
$|J_{V^\ell_k}T^n|_{\C^0(V^\ell_k)} \leq |J_{W^{\ell,n}_i}T^n|_{\C^0(W^{\ell,n}_i)}$
whenever $V^\ell_k \subseteq W^{\ell,n}_i$.

Next, we must estimate
\[
\sum_j\left|\int_{U^1_j}h(J_\mu T^n)^{-1}J_{U^1_j}T^n \, \psi_1\circ T^n \, dm_W -
\int_{U^2_j}h (J_\mu T^n)^{-1}J_{U^2_j}T^n \, \psi_2\circ T^n \, dm_W \right|  .
\]
 We fix $j$ and estimate the difference.
Define
\[
\phi_j = ((J_\mu T^n)^{-1}J_{U^1_j}T^n \, \psi_1 \circ T^n) \circ G_{U^1_j} \circ G_{U^2_j}^{-1} .
\]
The function $\phi_j$ is well-defined on $U^2_j$ and we can write,
\begin{equation}
\label{eq:stepone}
\begin{split}
&\left|\int_{U^1_j}h(J_\mu T^{-1}J_{U^1_j}T^n \, \psi_1\circ T^n-
\int_{U^2_j}h (J_\mu T^n)^{-1}J_{U^2_j}T^n \, \psi_2\circ T^n\right|\\
&\leq \left|\int_{U^1_j}h (J_\mu T^n)^{-1}J_{U^1_j}T^n \, \psi_1\circ T^n-
\int_{U^2_j}h \,\phi_j \right|
+\left|\int_{U^2_j}h (\phi_j - (J_\mu T^n)^{-1}J_{U^2_j}T^n \, \psi_2\circ
T^n)\right| .
\end{split}
\end{equation}

We estimate the first term on the right hand side of~\eqref{eq:stepone} using the
strong unstable norm.
Using {\bf (H5)},  \eqref{eq:holder c0} and \eqref{eq:C1 C0},
\begin{equation}
\label{eq:c1-unst 1}
| (J_\mu T^n)^{-1}J_{U^1_j}T^n\cdot \psi_1 \circ
T^n|_{\C^p(U^1_j)}
\le C_e (1+2C_d) \eta^n |J_{U^1_j}T^n|_{\C^0(U^1_j)}.
\end{equation}
Notice that
\begin{equation}
\label{eq:graph bound}
|G_{U^1_j} \circ G^{-1}_{U^2_j}|_{\C^1(U^2_j)}
\le \sup_{r \in U^2_j}
\frac{\sqrt{1 + (d\vf_{U^1_j}/dr)^2}}{\sqrt{1 + (d\vf_{U^2_j}/dr)^2}} \le
\sqrt{1 + \Gamma^2} =: C_g,
\end{equation}
where $\Gamma$ is the maximum slope of curves in $\W^s$
given by {\bf (H2)}.  Using this, we estimate as in \eqref{eq:c1-unst 1},
\[
|\phi_j|_{\C^p(U^2_j)}
 \le C_g C_e (1+2C_d) \eta^n |J_{U^1_j}T^n|_{\C^0(U^1_j)}.
\]
By the definition of $\phi_j$ and $d_q(\cdot, \cdot)$,
\[
d_q((J_\mu T^n)^{-1}J_{U^1_j}T^n\psi_1\circ T^n, \phi_j)
= \left| \left[  (J_\mu T^n)^{-1}J_{U^1_j}T^n\psi_1\circ T^n \right] \circ G_{U^1_j}
  - \phi_j \circ G_{U^2_j} \right|_{\C^q(I_j)} \; = \; 0 .
\]

By Lemma~\ref{lem:angles}(a), we have
$d_{\W^s}(U^1_j,U^2_j)\leq C_0\Lambda^{-n} \ve =: \ve_1$.
In view of \eqref{eq:c1-unst 1} and following, we renormalize the test functions by
$R_j = C_7 \eta^n |J_{U^1_j}T^n|_{\C^0(U^1_j)}$
where $C_7 = C_g C_e (1+ 2C_d)$.
Then we apply the definition of the strong unstable norm with
$\ve_1$ in place of $\ve$.      Thus,
\begin{equation}
\label{eq:second unstable}
\sum_j \left|\int_{U^1_j}h  J_\mu T^n)^{-1} J_{U^1_j}T^n \, \psi_1\circ T^n -
\int_{U^2_j} h  \, \phi_j \; \right|
\leq C_7 C_0^\beta \ve^\beta \Lambda^{-\beta n} \eta^n \|h\|_u \sum_j |J_{U^1_j}T^n|_{\C^0(U^1_j)}
\end{equation}
where the sum is $\le C_2$ by Lemma~\ref{lem:growth}(b) since there is at most
one matched piece $U^1_j$ corresponding to each element
$W^{1,n}_i \in \G_n(W^1)$ and $|J_{U^1_j}T^n|_{\C^0(U^1_j)} \le |J_{W^{1,n}_i}T^n|_{\C^0(W^{1,n}_i)}$
whenever $U^1_j \subseteq W^{1,n}_i$.

It remains to estimate the second term in \eqref{eq:stepone} using the
strong stable norm.
\begin{equation}
\label{eq:unstable strong}
 \left|\int_{U^2_j}h(\phi_j - (J_\mu T^n)^{-1}J_{U^2_j}T^n \psi_2 \circ T^n) \right|
\leq \;  \|h\|_s |U^2_j|^\alpha
       \left|\phi_j - (J_\mu T^n)^{-1} J_{U^2_j}T^n \psi_2 \circ T^n\right|_{\C^q(U^2_j)} .
\end{equation}
In order to estimate the $\C^q$-norm of the function in \eqref{eq:unstable strong}, we split
it up into two differences.      Since $|G_{U^\ell_j}|_{\C^1} \le C_g$ and
$|G_{U^\ell_j}^{-1}|_{\C^1} \le 1$, $\ell = 1,2$,
we write
\begin{equation}
\label{eq:diff}
\begin{split}
& | \phi_j - ((J_\mu T^n)^{-1}J_{U^2_j}T^n)\cdot \psi_2 \circ T^n|_{\C^q(U^2_j)} \\
\leq& \; \left |\left[ ((J_\mu T^n)^{-1}J_{U^1_j}T^n)\cdot\psi_1 \circ
T^n\right]\circ G_{U^1_j} - \left[((J_\mu T^n)^{-1}J_{U^2_j}T^n)\cdot\psi_2
\circ T^n\right] \circ G_{U^2_j}\right|_{\C^q(I_j)}\\
\leq& \; \left | ((J_\mu T^n)^{-1}J_{U^1_j}T^n)\circ G_{U^1_j} \left[ \psi_1
\circ T^n\circ G_{U^1_j} -\psi_2 \circ T^n\circ
G_{U^2_j}\right]\right|_{\C^q(I_j)}\\
&+ \left|\left[((J_\mu T^n)^{-1}J_{U^1_j}T^n) \circ
G_{U^1_j}-((J_\mu T^n)^{-1}J_{U^2_j}T^n) \circ G_{U^2_j}\right]\psi_2\circ
T^n\circ G_{U^2_j}\right|_{\C^q(I_j)}\\
\leq& \; C_g (1+2C_d) | (J_\mu T^n)^{-1}J_{U^1_j}T^n|_{\C^0(U^1_j)} \left|\psi_1 \circ T^n\circ
G_{U^1_j} -\psi_2 \circ T^n\circ G_{U^2_j}\right|_{\C^q(I_j)}\\
&+ C_g C_e \left|((J_\mu T^n)^{-1}J_{U^1_j}T^n) \circ
G_{U^1_j}-((J_\mu T^n)^{-1}J_{U^2_j}T^n) \circ
G_{U^2_j}\right|_{\C^q(I_j)}
\end{split}
\end{equation}
To bound the two differences above, we need the following lemma.

\begin{lemma}
\label{lem:test}
There exist constants $C_8, C_9>0$, depending only on {\bf (H1)}-{\bf (H5)}, such that,
\begin{itemize}
  \item[(a)] $\displaystyle
|((J_\mu T^n)^{-1} J_{U^1_j}T^n)\circ G_{U^1_j}
-(J_\mu T^n)^{-1} J_{U^2_j}T^n)\circ G_{U^2_j}|_{\C^q(I_j)}\leq C_8 | (J_\mu T^n)^{-1}
J_{U^2_j}T^n|_{C^0(U^2_j)} \ve^{1/3-q};$
  \item[(b)] $\displaystyle
|\psi_1 \circ T^n \circ G_{U^1_j} - \psi_2 \circ T^n \circ G_{U^2_j} |_{\C^q(I_{r_j})}
\le C_9  \ve^{p-q} .$
\end{itemize}
\end{lemma}

We postpone the proof of the lemma to Section~\ref{lemma proofs} and
show how this completes the estimate on the strong unstable norm.
Note that using \eqref{eq:comparable J}, we may replace $|J_{U^1_j}T^n|_{\C^0(U^1_j)}$
by $C|J_{U^1_j}T^n|_{\C^0(U^1_j)}$ where it appears in our estimates for some
uniform constant $C$.
Starting from \eqref{eq:unstable strong}, we apply Lemma~\ref{lem:test} to \eqref{eq:diff}
 to obtain,
\begin{equation}
\label{eq:unstable three}
\begin{split}
& \sum_j \Big| \int_{U^2_j} h(\phi_j     - (J_\mu T^n)^{-1} J_{U^2_j}T^n \psi_2 \circ T^n ) \, dm_W \Big| \\
& \le \bar C \|h\|_s \sum_j |U^2_j|^\alpha
|(J_\mu T^n)^{-1} J_{U^2_j}T^n|_{\C^0(U^2_j)} \, \ve^{p -q}
\le \bar C \eta^n \|h\|_s \ve^{p-q} \sum_j  |J_{U^2_j}T^n|_{\C^0(U^2_j)},
\end{split}
\end{equation}
for some uniform constant $\bar C$
where again the sum is finite as in \eqref{eq:second unstable}.
This completes the estimate on the second term in \eqref{eq:stepone}.
Now we use this bound, together with \eqref{eq:first unstable} and
\eqref{eq:second unstable} to estimate \eqref{eq:unstable split}
\[
\begin{split}
 \left|\int_{W^1} \Lp^nh \, \psi_1 \, dm_W - \int_{W^2} \Lp^nh \, \psi_2 \, dm_W \right|
  \; \leq \; CC_3^n \eta^n  \|h\|_s \ve^\alpha + C \|h\|_u \Lambda^{-\beta n} \eta^n \ve^\beta
  + C \eta^n \|h\|_s \ve^{p-q} ,
\end{split}
\]
where again $C$ depends only on {\bf (H1)}-{\bf (H5)} through the estimates above.
Since $p-q \ge \beta$ and $\alpha \ge \beta$, we divide through by $\ve^\beta$ and take
the appropriate suprema to complete the proof of \eqref{eq:unstable norm}.


\subsubsection{Proof of Lemma~\ref{lem:test}}
\label{lemma proofs}

First we prove the following general fact and then use it to prove Lemma~\ref{lem:test}.

\begin{lemma}
\label{lem:general}
Let $(N,d)$ be a metric space and let $0<r<s \le 1$.
Suppose $g_1, g_2 \in \C^s(N, \mathbb{R})$ satisfy $|g_1 - g_2|_{\C^0(N)} \le D_1 \ve^s$ for some
constant $D_1>0$.
Then $|g_1 - g_2|_{\C^r(N)} \le 3 \ve^{s-r} \max \{ D_1, H^s(g_1)+ H^s(g_2) \}$,
where $H^s(\cdot)$
denotes the H\"older constant with exponent $s$ on $N$.
\end{lemma}
\begin{proof}
Since $| \cdot |_{\C^r(N)} = | \cdot |_{\C^0(N)} + H^r(\cdot)$, we must estimate $H^r(g_1 - g_2)$.
Let $x, y \in N$.  Then on the one hand, since $|g_1 -g_2| \le D_1 \ve^s$, we have
\[
\frac{|(g_1(x) - g_2(x)) - (g_1(y)-g_2(y)|}{d(x,y)^r} \leq 2 D_1 \ve^s d(x,y)^{-r}
\]
On the other hand, using the fact that $g_1, g_2 \in \C^s(N)$, we have
\[
\frac{|(g_1(x) - g_2(x)) - (g_1(y)-g_2(y)|}{d(x,y)^r} \leq (H^s(g_1)+H^s(g_2))  d(x,y)^{s-r} .
\]
These two estimates together imply that the H\"older constant of $g_1 - g_2$ is bounded by
\[
H^r(g_1 - g_2) \le \sup_{x,y \in N} \min \{ 2 D_1 \ve^s d(x,y)^{-r}, (H^s(g_1)+H^s(g_2))  d(x,y)^{s-r} \}.
\]
This expression is maximized when $ 2 D_1 \ve^s d(x,y)^{-r} = (H^s(g_1)+H^s(g_2))  d(x,y)^{s-r}$,
i.e., when $d(x,y) = \ve \left( \frac{2D_1}{H^s(g_1) + H^s(g_2)} \right)^{1/s}$.  Thus the H\"older constant
of $g_1-g_2$ satisfies,
\[
H^r(g_1-g_2) \le \ve^{s-r} (2D_1)^{1-\frac rs}(H^s(g_1) + H^s(g_2))^{\frac rs} .
\]
\end{proof}

\begin{proof}[Proof of Lemma~\ref{lem:test}(a)]
Throughout the proof, for ease of notation
we write $J_\ell^n$ for $(J_\mu T^n)^{-1} J_{U^\ell_j}T^n$.

For any $r \in I_j$, $x=G_{U^1_j}(r)$ and $\bar x=G_{U_j^2}(r)$ lie
on a common vertical segment.
By the construction at the beginning of Section~\ref{unstable norm},
$U^1_j$, $U^2_j$ lie in two homogeneous stable
curves $\widetilde U^1_j$ and
$\widetilde U^2_j$ which are connected by the foliation $\{ \gamma_x \}$.
Thus $x^* := \gamma_x \cap \widetilde U^2_j$ is uniquely defined
for all $x \in U^1_j$.
Then $T^n(x)$ and $T^n(x^*)$ lie on the element $T^n\gamma_x \in
\W^u$ which intersects $W^1$ and $W^2$ and
has length at most $C_t \varepsilon$.
By \eqref{eq:D u dist} and Lemma~\ref{lem:angles}(b),
\[
|J_1^n(x) - J_2^n(x^*)|
\leq C_d C_0  |J_2^n|_{\C^0(U^2_j)} (d_W(T^nx,T^n x^*)^{1/3} + \theta(T^nx, T^n x^*)) ,
\]
where $\theta(T^nx, T^n x^*)$ is the angle between the tangent line to $W^1$ at
$T^nx$ and the tangent line to $W^2$ at $T^n x^*$.
Let $y \in W^2$ be the unique point in $W^2$ which lies on the same
vertical segment as $T^nx$.
Since by assumption $d_{\W^s}(W^1, W^2) \le \ve$, we have
$\theta(T^nx, y) \le \ve$.  Due to the uniform transversality of curves in $\W^u$ and $\W^s$
and the fact that $W^2$ is the graph of a $\C^2$ function with
$\C^2$ norm bounded by $B$ from {\bf (H2)}, we have $\theta(y, T^n x^*) \le BC_t\ve$ and so
$\theta(T^nx, T^n x^*) \le (1+BC_t)\ve$.  Thus
\begin{equation}
\label{eq:inter}
|J_1^n(x) - J_2^n(x^*)| \leq C_d C_0 (C_t+1+BC_t) \ve^{1/3} |J_2^n|_{\C^0(U^2_j)}  .
\end{equation}
Also, by \eqref{eq:distortion stable}, since $x^*$ and $\bar x$ are both on
$\widetilde U^2_j$, we have
$|J_2^n(x^*) - J_2^n(\bar x)| \le C_d^2 |J_2^n|_{\C^0(U^2_j)}d_W(x^*, \bar x)^{1/3}$.
Putting this together with \eqref{eq:inter} and using the fact that $d_W(x^*,\bar x) \le C_t \ve$ by the transversality of $\gamma_x$ with $\W^s$ yields,
\begin{equation}
\label{eq:J C0}
|J_1^n(x) - J_2^n(\bar x)| \leq C' \ve^{1/3} |J_2^n|_{\C^0(U^2_j)},
\end{equation}
where $C' = C_d C_0 (2 C_t + 1 + BC_t)$.

Now using the fact that $|G_{U^\ell_j}|_{\C^1(I_j)} \le C_g$ from \eqref{eq:graph bound},
we apply Lemma~\ref{lem:general} with $D_1 = C_1 |J_2^n|_{\C^0(U^2_j)}$
and $g_i = J_i^n \circ G_{U^i_j}$, $i=1,2$.
By \eqref{eq:J C0}, we have
\begin{equation}
\label{eq:comparable J}
|J_1^n|_{\C^0(U^1_j)} \le (1+C' \ve^{1/3}) |J_2^n|_{\C^0(U^2_j)} ,
\end{equation}
and invoking \eqref{eq:distortion stable}, we complete the proof of (a).
\end{proof}

\begin{proof}[Proof of (b)]
Let $\vf_{W^\ell}$ be the
function whose graph is $W^\ell$,  defined for $r \in I_{W^\ell}$, and set $f^\ell_j:=G_{W^\ell}^{-1}\circ
T^n\circ G_{U^\ell_j}$, $k=1,2$.  Notice that since
$|G_{W^\ell}^{-1}|_{\C^1} \le 1$ and $|G_{U^\ell_j}|_{\C^1} \le C_g$, and due to
the uniform contraction along stable curves, we have
Lip$(f^\ell_j) \le C_f$, where $C_f$ is independent of $W^\ell$, $T$ and $j$.
We may assume that $f^\ell_j(I_j) \subset I_{W^1} \cap I_{W^2}$ since if not, by the
transversality of $C^u(x)$ and $C^s(x)$, we must be in a neighborhood of
one of the endpoints of $W^\ell$ of length at most $C_t \ve$; such short pieces
may be estimated as in \eqref{eq:first unstable} using the strong stable norm.
Thus
\begin{equation}
\begin{split}
\label{eq:test split}
|\psi_1 \circ T^n \circ G_{U^1_j} - \psi_2 \circ T^n \circ G_{U^2_j} |_{\C^q(I_j)}
& \le |\psi_1 \circ G_{W^1} \circ f^1_j - \psi_2 \circ G_{W^2} \circ f^1_j|_{\C^q(I_j)} \\
&+ |\psi_2 \circ G_{W^2} \circ f^1_j - \psi_2 \circ G_{W^2} \circ f^2_j |_{\C^q(I_j)} .
\end{split}
\end{equation}
Using the above observation about $f^1_j$, we estimate the first term of
\eqref{eq:test split} by
\begin{equation}
\label{eq:test 2}
|\psi_1 \circ G_{W^1} \circ f^1_j - \psi_2 \circ G_{W^2} \circ f^1_j |_{\C^q(I_j)}
\leq C_f |\psi_1 \circ G_{W^1} - \psi_2 \circ G_{W^2}|_{\C^q(f^1_j(I_j))} \le C_f \ve ,
\end{equation}
since $d_q(\psi_1, \psi_2) \le \ve$.
To estimate the second term of \eqref{eq:test split}, notice that since
$U^1_j$ and $U^2_j$ are joined by the transverse foliation
$\{ \gamma_x \} \subset \widehat \W^u$ and
using the uniform contraction along stable curves under $T^n$, we have
$|f^1_j - f^2_j|_{\C^0(I_j)} \le \tilde C \ve$ for a constant $\tilde C$ depending
only on the uniform hyperbolicity of {\bf (H1)} and the uniform transversality
conditions in {\bf (H2)}.  Thus for $r \in I_j$,
\begin{equation}
\label{eq:test C0}
|\psi_2 \circ G_{W^2} \circ f^1_j(r) - \psi_2 \circ G_{W^2} \circ f^2_j(r)|
\le C_g |\psi_2|_{\C^p} |f^1_j(r) - f^2_j(r)|^p \leq C_g \tilde C |\psi_2|_{\C^p} \ve^p .
\end{equation}
Now we again apply Lemma~\ref{lem:general} to obtain
\[
|\psi_2 \circ G_{W^2} \circ f^1_j - \psi_2 \circ G_{W^2} \circ f^2_j|_{\C^q(I_j)}
\le C |\psi_2|_{\C^p} \ve^{p-q} ,
\]
for a uniform constant $C$.
This estimate combined with \eqref{eq:test 2} proves part (b)
since $|\psi_2|_{\C^p(W^2)} \le 1$.
\end{proof}


\section{Proof of Theorem~\ref{thm:close}}
\label{close}

Fix $\ve < \ve_0$ and
suppose $T_1, T_2 \in \F$ with $d_\F(T_1, T_2) \le \ve$.
We denote by $\Si^\ell_{-n}$ the singularity sets for $T_\ell$, $\ell=1,2$.
Let $h \in \C^1(M)$, $\| h \|_\B \le 1$, and $W \in \W^s$.  Let $\psi \in \C^p(W)$ with
$|\psi|_{W, 0 , p} \le 1$.  We must estimate
\begin{equation}
\label{eq:L diff}
\begin{split}
\int_W (\Lp_1 h & - \Lp_2 h) \psi \, dm_W = \int_W \Lp_1 h \psi \, dm_W - \int _W \Lp_2 h \psi \, dm_W \\
& = \int_{T_1^{-1}W} h \, \psi \circ T_1 (J_\mu T_1)^{-1} J_{T_1^{-1}W}T \, dm_W
- \int_{T_2^{-1}W} h \, \psi \circ T_2 (J_\mu T_2)^{-1} J_{T_2^{-1}W}T_2 \, dm_W .
\end{split}
\end{equation}

Notice that the estimate required is similar to that done in
Section~\ref{unstable norm},
except that instead of two close stable curves iterated under the same map, we have one
stable curve iterated under two different maps.

We partition $T_1^{-1}W$ and $T_2^{-1}W$ into matched and unmatched pieces
as in the beginning of Section~\ref{unstable norm}.
Let $\G_1^\ell(W)$, $\ell=1,2$, denote the elements of $T_\ell^{-1}W$ as described in
Section~\ref{preliminary}.
Let $\omega \in \G_1^1(W)$.
Due to {\bf (C1)},
to each
point $x \in \omega$, we associate a curve $\gamma_x \in \widehat W^u$ of length
at most $C_t \ve$ which terminates on a piece of
$T_2^{-1}W$ that lies in the same homogeneity
strip, if one exists.
We also require that $\gamma_x$ is not cut by $\Si_{1}^1 \cup \Si_{1}^2$.

We denote by $V^\ell_k$ those components of $T_\ell^{-1}W$ not matched
by this process.  We also include in the set of $V^\ell_k$ all images of connected components
of $W \cap N_{\ve}(\Si^1_{-1} \cup \Si^2_{-1})$ under $T_\ell^{-1}$.
Note that
the $T_\ell V^\ell_k$ occur either at the endpoints of $W$ or near a singularity or the boundary of
$N_\ve(\Si^1_{-1} \cup \Si^2_{-1}$).  In all cases, the length of the
curves $T_\ell V^\ell_k$ can be at most $C_t C_e \ve$ due to the uniform transversality of
$\Si^\ell_{-1}$ with $C^s$ and of $C^s$ with $C^u$.

In the remaining
pieces the foliation $\{ \gamma_x \}$ provides a one-to-one correspondence
between points in $T_1^{-1}W$ and $T_2^{-1}W$.
We further partition these pieces in
such a way that their lengths
are between $\delta_0/2$ and $\delta_0$ and
the pieces are pairwise matched by
the foliation $\{\gamma_x\}$. We call these matched pieces $\widetilde U^\ell_j$.
As in Section~\ref{unstable norm}, we trim the $\widetilde U^\ell_j$ to
pieces $U^\ell_j$ so that $U^1_j$ and $U^2_j$ are defined on the same arclength
interval $I_j$.  The at most two components of $T_\ell(\widetilde U^\ell_j \setminus U^\ell_j)$
have length at most $C_t C_e \Lambda^{-1} \ve$.
We adjoin these trimmed pieces to the adjacent
 $U^\ell_i$ or $V^\ell_k$ as appropriate so as not to create more pieces in the partition
of $T_\ell^{-1} W$.

In this way, we write $T^{-1}_\ell W = (\cup_j U^\ell_j) \cup (\cup_k V^\ell_k)$ and note that
the images $T_\ell V^\ell_k$ have length at most $C_v \ve$ for some uniform constant
$C_v$, $\ell = 1,2$.

Now using \eqref{eq:L diff}, we have
\begin{equation}
\label{eq:L close split}
\begin{split}
\int_W (\Lp_1 h & - \Lp_2 h) \psi \, dm_W
= \sum_{\ell,k} \int_{V^\ell_k} h \, \psi \circ T_\ell \, (J_\mu T_\ell)^{-1} J_{V^\ell_k}T_\ell \, dm_W \\
& \; \; \; \; + \sum_j \int_{U^1_j} h \, \psi \circ T_1 \, (J_\mu T_1)^{-1} J_{U^1_j}T_1 \, dm_W
- \int_{U^2_j} h \, \psi \circ T_2 \, (J_\mu T_2)^{-1} J_{U^2_j}T_2 \, dm_W  .
\end{split}
\end{equation}
We estimate the integral on short pieces $V^\ell_k$ first using the strong stable norm.
By \eqref{eq:C1 C0}, we have  $|\psi \circ T_\ell|_{\C^q(V^\ell_k)}
\leq C_e |\psi|_{\C^p(W)} \leq C_e$.  Following the estimate in \eqref{eq:first unstable}, we have
\begin{equation}
\label{eq:first close}
\begin{split}
\sum_{\ell, k} & \left|\int_{V^\ell_k}h (J_\mu T_\ell)^{-1}J_{V^\ell_k}T_\ell \, \psi \circ T_\ell \, dm\right|
\leq C \ve^{\alpha}  \|h\|_s   \sum_{\ell, k} |J_{V^\ell_k}T_\ell |_{\C^0(V^\ell_k)}^{1-\alpha} .
\end{split}
\end{equation}
The sum is finite by \eqref{eq:weakened step1} of {\bf (H3)} with $\varsigma = 1-\alpha$
since there are at most two $V^\ell_k$ corresponding to
each element $W^{\ell,1}_i \in \G^\ell_1(W)$ as defined in Section~\ref{preliminary} and
$|J_{V^\ell_k}T_\ell |_{\C^0(V^\ell_k)} \leq |J_{W^{\ell, 1}_i}T_\ell |_{\C^0(W^{\ell, 1}_i)}$
whenever $V^\ell_j \subseteq W^{\ell, 1}_i$.
The constant $C$ above depends only on properties {\bf (H1)}-{\bf (H5)},
but for brevity we do not write out the explicit dependence since these estimates
are similar to those done in Section~\ref{unstable norm} and the constants are the same.

Next, we must estimate
\[
\sum_j\left|\int_{U^1_j}h \, (J_\mu T_1)^{-1}J_{U^1_j}T_1 \, \psi \circ T_1 \, dm_W -
\int_{U^2_j}h \, (J_\mu T_2)^{-1}J_{U^2_j}T_2 \, \psi \circ T_2 \, dm_W \right|      .
\]
Using notation analogous to \eqref{eq:match},
we fix $j$ and estimate the difference.
Define
\[
\phi_j = ((J_\mu T_1)^{-1}J_{U^1_j}T_1 \, \psi \circ T_1) \circ G_{U^1_j} \circ G_{U^2_j}^{-1} .
\]
The function $\phi_j$ is well-defined on $U^2_j$ and we can write,
\begin{equation}
\label{eq:L stepone}
\begin{split}
&\left|\int_{U^1_j}h \, (J_\mu T_1)^{-1}J_{U^1_j}T_1 \, \psi \circ T_1 -
\int_{U^2_j}h \, (J_\mu T_2)^{-1}J_{U^2_j}T_2 \, \psi \circ T_2\right|\\
&\leq \left|\int_{U^1_j}h \, (J_\mu T_1)^{-1}J_{U^1_j}T_1 \, \psi \circ T_1 -
\int_{U^2_j}h \,\phi_j \right|
+\left|\int_{U^2_j}h (\phi_j - (J_\mu T_2)^{-1}J_{U^2_j}T_2 \, \psi \circ
T_2)\right| .
\end{split}
\end{equation}

To estimate the two terms above, we need the following adaptation of Lemma~\ref{lem:test}.
\begin{lemma}
\label{lem:close}
There exists $\bar C>0$, independent of $W \in \W^s$ and $T_1, T_2 \in \F$,
such that for each $j$,
\begin{itemize}
  \item[(a)]  $d_{\W^s}(U^1_j,U^2_j)\leq \bar C \ve^{1/2}$ ;
  \item[(b)]  $ \displaystyle
|((J_\mu T_1)^{-1} J_{U^1_j}T_1)\circ G_{U^1_j}
-((J_\mu T_2)^{-1} J_{U^2_j}T_2)\circ G_{U^2_j}|_{\C^q(I_j)}\leq \bar C | (J_\mu T_2)^{-1}
J_{U^2_j}T_2|_{C^0(U^2_j)} \ve^{1/3-q}$ ;
  \item[(c)] $ \displaystyle
|\psi \circ T_1 \circ G_{U^1_j} - \psi \circ T_2 \circ G_{U^2_j} |_{\C^q(I_j)}
\le \bar C  \ve^{p-q} $ .
\end{itemize}
\end{lemma}

We estimate the first term in equation~\eqref{eq:L stepone} using the strong
unstable norm.
The estimates \eqref{eq:holder c0} and
\eqref{eq:C1 C0} and property {\bf (H5)} imply that
\begin{equation}
\label{eq:L c1-unst 1}
| (J_\mu T_1)^{-1}J_{U^1_j}T_1\cdot \psi \circ T_1|_{U^1_j,0,p}
\le \eta C_e  |J_{U^1_j}T_1|_{\C^0(U^1_j)}.
\end{equation}
Similarly, since by \eqref{eq:graph bound},
$|G_{U^1_j} \circ G_{U^2_j}^{-1}|_{\C^1} \le C_g$,
we have
$ |\phi_j|_{U^2_j, 0,p}
\le C_g \eta C_e |J_{U^1_j}T_1|_{\C^0(U^1_j)}$.
By the definition of $\phi_j$ and $d_q(\cdot, \cdot)$,
\[
d_q((J_\mu T_1)^{-1}J_{U^1_j}T_1\psi \circ T_1, \phi_j)
= \left| \left[  (J_\mu T_1)^{-1}J_{U^1_j}T_1\psi \circ T_1 \right] \circ G_{U^1_j}
  - \phi_j \circ G_{U^2_j} \right|_{\C^q(I_j)} \; = \; 0 .
\]
In view of \eqref{eq:L c1-unst 1} and following, we renormalize the test functions by
$R_j = \eta C_g C_e |J_{U^1_j}T_1|_{\C^0(U^1_j)}$.
Then we apply the definition of the strong unstable norm using Lemma~\ref{lem:close}(a)
to obtain,
\begin{equation}
\label{eq:second close}
\sum_j \left|\int_{U^1_j}h  \, (J_\mu T_1)^{-1} J_{U^1_j}T_1 \, \psi_1\circ T_1 -
\int_{U^2_j} h  \, \phi_j \; \right|
\leq C \ve^{\beta/2} \|h\|_u \sum_j |J_{U^1_j}T_1|_{\C^0(U^1_j)},
\end{equation}
where the sum is $\le C_2$ by Lemma~\ref{lem:growth}(b) since there is at most
one matched piece $U^1_j$ corresponding to each curve
$W^{1}_i \in \G_1^1(W)$.

We estimate the second term in \eqref{eq:L stepone} using the
strong stable norm.
\begin{equation}
\label{eq:L unstable strong}
\left|\int_{U^2_j}h(\phi_j - (J_\mu T_2)^{-1}J_{U^2_j}T_2 \psi \circ T) \right|
\leq \; C \|h\|_s |U^2_j|^\alpha
       \left|\phi_j - (J_\mu T_2)^{-1} J_{U^2_j}T_2 \psi \circ T_2\right|_{\C^q(U^2_j)} .
\end{equation}
In order to estimate the $\C^q$-norm of the function in \eqref{eq:L unstable strong},
we split it up into two differences.
Following \eqref{eq:diff} line by line, we obtain
\begin{equation}
\label{eq:close diff}
\begin{split}
 | \phi_j & - (J_\mu T_2)^{-1}J_{U^2_j}T_2\cdot \psi)\circ T_2|_{\C^q(U^2_j)} \\
& \leq \; C |\, (J_\mu T_1)^{-1}J_{U^1_j}T_1|_{\C^0(U^1_j)} \left|\psi \circ T_1\circ
G_{U^1_j} -\psi \circ T_2\circ G_{U^2_j}\right|_{\C^q(I_j)}\\
&\qquad + C \left|((J_\mu T_1)^{-1}J_{U^1_j}T_1) \circ
G_{U^1_j}-((J_\mu T_2)^{-1}J_{U^2_j}T_2) \circ
G_{U^2_j}\right|_{\C^q(I_j)}
\end{split}
\end{equation}
Notice that $|J_{U^1_j}T_1|_{\C^0(U^1_j)}
\le C |J_{U^2_j}T_2|_{\C^0(U^2_j)}$ by \eqref{eq:J est}.
Then using Lemma~\ref{lem:close}(b) and (c) together with
\eqref{eq:close diff} yields by \eqref{eq:L unstable strong}
\[
\sum_j \Big| \int_{U^2_j} h(\phi_j - (J_\mu T_2)^{-1} J_{U^2_j}T_2 \psi \circ T_2 ) \, dm_W \Big|
\le C \|h\|_s \ve^{p-q} \sum_j |J_{U^2_j}T_2|_{\C^0(U^2_j)} ,
\]
where again the sum is finite by Lemma~\ref{lem:growth}(b).
This completes the estimate on the second term in \eqref{eq:L stepone}.
Now we use this bound, together with \eqref{eq:first close} and
\eqref{eq:second close} to estimate \eqref{eq:L close split}
\begin{equation}
\label{eq:final close}
 \left|\int_{W} \Lp_1 h \, \psi \, dm_W - \int_{W} \Lp_2 h \, \psi \, dm_W \right|
  \; \leq \; C \|h\|_s \ve^\alpha + C \|h\|_u \ve^{\beta/2}
  + C \|h\|_s \ve^{p-q} .
\end{equation}
Since $p-q \ge \beta$ and $\alpha \ge \beta$, the theorem is proved.


\subsection{Proof of Lemma~\ref{lem:close}}
\label{close lemma proofs}

\begin{proof}[Proof of (a)]
Note that by construction $U^1_j$ and $U^2_j$ lie in the same homogeneity strip.  Also, they
are both defined on the same interval $I_j$ so the length of the symmetric difference
of their $r$-intervals is 0.  Recalling the definition of $d_{\W^s}(U^1_j, U^2_j)$, we see
that it remains only to estimate
$|\vf_{U^1_j} - \vf_{U^2_j}|_{\C^1(I_j)}$ for their defining functions $\vf_{U^k_j}$.

For $x = (r, \vf_{U^1_j}(r))$, define  $\bar x = (r, \vf_{ U^2_j}(r))$
and $x_\ve = T_2^{-1} \circ T_1(x)$.
We may assume that $x_\ve$ lies on $\widetilde U^2_j$ since otherwise, we would be
$\ve$-close to one of the $V^2_k$ and such short curves can be estimated as in
\eqref{eq:first close} using the strong stable norm.
Since $x$ and $x_\ve$ are images
of the same point $u \in W$ under $T_1^{-1}$ and $T_2^{-1}$ respectively, it follows
from {\bf (C1)} that $x$ and $x_\ve$ are at most $\ve$ apart.	Then since all
vectors in the stable cone have slope bounded away from $\pm \infty$, it follows that
$x$ and $\bar x$ are at most $C\ve$ apart (and so by the triangle inequality,
also $\bar x$ and $x_\ve$ are at most $C\ve$ apart).

This proves that $|\vf_{U^1_j} - \vf_{U^2_j}|_{\C^0(I_j)} \le C \ve$.
It remains to estimate $|\vf'_{U^1_j} - \vf'_{U^2_j}|$, where $\vf_{U^\ell_j}'$ denotes the derivative
of $\vf_{U^\ell_j}$ with respect to $r$.

Let $\vec{v}_W(u)$ be the unit tangent vector to $W$ at $u := T_1(x) = T_2(x)$, as before.
The tangent vector
to $U^\ell_j$ is given by $DT_\ell^{-1}(u)\vec{v}_W(u)$, $\ell =1,2$.  By {\bf (C4)},
\begin{equation}
\label{eq:close slope}
\| DT_1^{-1}(u) \vec{v}_W(u) - DT_2^{-1}(u) \vec{v}_W(u) \| \le \ve^{1/2} .
\end{equation}
Then since $\| DT_\ell^{-1}(u) \vec{v}_W(u) \| \ge C_e^{-1}$ by {\bf (H1)}, we have
$\theta(x, x_\ve) \le C_e \ve^{1/2}$, where $\theta(x, x_\ve)$ is the angle
between the tangent vectors to $U^1_j$ and $U^2_j$ at $x$ and $x_\ve$ respectively.

For $y \in U^\ell_j$, let $\phi(y)$ denote the
angle that $G_{U^\ell_j}$ makes with the positive $r$-axis at $y$.
Then
\[
|\vf'_{U^1_j}(x) - \vf'_{U^2_j}(\bar x) | = |\tan \phi(x) - \tan \phi(\bar x)|
\leq \big[ \sup_{z \in U^\ell_j} \sec^2 \phi(z) \big] |\phi(x) - \phi(\bar x)|
=  \big[ \sup_{z \in U^\ell_j} \sec^2 \phi(z) \big] \theta(x, \bar x) .
\]
Since the slopes of curves in $C^s(x)$ are uniformly bounded away from $\pm \infty$,
we have $\sec^2 \phi(z)$ uniformly bounded above for any $z \in U^k_j$.
The proof of the lemma is completed by writing
$\theta(x, \bar x) \le \theta(x, x_\ve) + \theta(x_\ve, \bar x)$.  The first term
is $\le C \ve^{1/2}$ using \eqref{eq:close slope} and the second term is $\le C \ve$ since
$x_\ve$ and $\bar x$ both lie on $\widetilde U^2_j$ and stable curves have a uniform
$\C^2$ bound by {\bf (H2)}.
 \end{proof}

\begin{proof}[Proof of (b)]
We prove that the closeness condition {\bf (C3)} implies the existence
of a constant $C>0$, independent of $W \in \W^s$ and $T_1, T_2 \in \F$, such that
\begin{equation}
\label{eq:J close}
|J_{U^1_j}T_1 \circ G_{U^1_j} - J_{U^2_j}T_2 \circ G_{U^2_j}|_{\C^q(I_j)} \le
C |J_{U^2_j}T_2|_{\C^0(U^2_j)} \ve^{1/3-q}.
\end{equation}
The analogous statement concerning $(J_\mu T_k)^{-1}$ follows from
condition {\bf (C2)}.  Then since
\[
|f_1 g_1 - f_2 g_2|_{\C^q} \le |f_1|_{\C^q} |g_1 - g_2|_{\C^q} + |g_2|_{\C^q}
|f_1 - f_2|_{\C^q},
\]
for any $\C^q$ functions $f_1, g_1, f_2, g_2$, part (b) of
the lemma follows from these two estimates
using the fact that $| \cdot |_{\C^q} \le (1+C_d) | \cdot |_{\C^0}$ by bounded distortion
for the functions we are estimating.  We proceed to prove \eqref{eq:J close}.

For any $r \in I_j$, we write
\begin{equation}
\label{eq:split J}
\begin{split}
|J_{U^1_j}T_1 \circ G_{U^1_j}(r) - J_{U^2_j}T_2 \circ G_{U^2_j}(r)|
\; \; \le \; \; & |J_{U^1_j}T_1 \circ G_{U^1_j}(r) - J_{U^1_j}T_2 \circ G_{U^1_j}(r)| \\
& + |J_{U^1_j}T_2 \circ G_{U^1_j}(r) - J_{U^2_j}T_2 \circ G_{U^2_j}(r)|
\end{split}
\end{equation}
The first term above is $\le |J_{U^1_j}T_2|_{\C^0(I_j)} \ve$ by {\bf (C3)}.

Recall that $U^1_j, U^2_j$ lie inside the longer curves
$\widetilde U^1_j, \widetilde U^2_j$ which are matched by the foliation
$\{ \gamma_x \}_{x \in \widetilde U^1_j} \subset \widehat \W^u$.
Thus $|J_{U^1_j}T_2|_{\C^0(I_j)} \le C |J_{U^2_j}T_2|_{\C^0(I_j)}$ by
the same argument used to prove \eqref{eq:comparable J}, completing the estimate on the
first term of \eqref{eq:split J}.

The second term of \eqref{eq:split J} is $\le C' \ve^{1/3} |J_{U^2_j}T_2|_{\C^0(I_j)}$
using \eqref{eq:J C0} since it involves the Jacobian of a single map in $\F$ evaluated
on two stable curves that are matched by a foliation of unstable curves.  Thus
\begin{equation}
\label{eq:J est}
|J_{U^1_j}T_1 \circ G_{U^1_j}(r) - J_{U^2_j}T_2 \circ G_{U^2_j} (r)| \leq C \ve^{1/3}
|J_{U^2_j}T_2|_{\C^0(U^2_j)}  .
\end{equation}
This implies in particular that $|J_{U^1_j}T_1|_{\C^0(U^1_j)} \le C
|J_{U^2_j}T_2|_{\C^0(U^2_j)}$.  Now we use \eqref{eq:distortion stable}
and the fact that
 $|G_{U^\ell_j}|_{\C^1(I_j)} \le C_g$ to apply Lemma~\ref{lem:general} and complete
 the proof of \eqref{eq:J close}.
\end{proof}

\begin{proof}[Proof of (c)]
Let $x = (r, \vf_{U^1_j}(r))$ and as above, define $\bar x = (r, \vf_{U^2_j}(r))$ and
$x_\ve = T_2^{-1} \circ T_1(x)$.
Since $\bar x$ and $x_\ve$ are at
most $C \ve$ apart and lie on $\widetilde U^2_j$, we have
$d_W(T_2 \bar x, T_2 x_\ve) \le C \ve$
by the uniform contraction given by
{\bf (H1)}.  Thus,
\begin{equation}
\label{eq:psi close}
|\psi \circ T_1 \circ G_{U^1_j}(r) - \psi \circ T_2 \circ G_{U^2_j}(r)|
\leq |\psi|_{\C^p(W)} d_W(T_1x,T_2 \bar x)^p .
\end{equation}
Since $d_W(T_1 x, T_2 x_\ve)= 0$, we may use the triangle inequality to
conclude that the difference above is bounded by $C |\psi|_{\C^p(W)} \ve^p$.

Again applying Lemma~\ref{lem:general} with $|\psi|_{\C^p(W)} \le 1$
completes the proof of part (c).
\end{proof}


\section{Proofs of Applications: Movements and Deformations of Scatterers and Random Perturbations}
\label{perts}

In this section we prove Theorems~\ref{thm:F1}, \ref{thm:deform} and \ref{thm:random}
and leave Theorems~\ref{thm:C1} and \ref{thm:C2} regarding external forces and kicks
to Section~\ref{kick} since they require more background material.

\subsection{Proof of Theorem~\ref{thm:F1}}

We fix constants $\tau_*, \K_* >0$ and $E_* < \infty$ and denote $\F_1(\tau_*, \K_*, E_*)$
as simply $\F_1$ for brevity.  Note that every $T \in \F_1$ is a billiard map corresponding
to a standard Lorentz gas with convex scatterers so that we may recall known facts
about such maps to establish {\bf (H1)}-{\bf (H5)} with constants depending only on the
three quantities $\tau_*$, $\K_*$ and $E_*$.

\smallskip
\noindent
{\em (H1)}.  For $x \in M$, define
\[
\begin{split}
C^s(x) & = \{ (dr, d\vf) \in \mathcal{T}_xM : - \K_*^{-1} - \tau_*^{-1} \le d\vf/dr \le - \K_* \} \\
\mbox{and} \; \;
C^u(x) & = \{ (dr, d\vf) \in \mathcal{T}_xM : \K_* \le d\vf/dr \le \K_*^{-1} + \tau_*^{-1} \} .
\end{split}
\]
Then for any $T \in \F_1$, $DT_xC^u(x) \subset C^u(Tx)$ and
$DT^{-1}_xC^s(u) \subset C^s(T^{-1}x)$ whenever $DT_x$ and $DT^{-1}_x$ are
defined.  Moreover, \eqref{eq:uniform hyp} is satisfied with
$\Lambda = 1 + 2\K_* \tau_*$ and
\[
C_e = \frac{2\tau_* \K_*}{\Lambda}\frac{\sqrt{1 + \K_*^2}}{\sqrt{1 + (\K_*^{-1}+\tau_*^{-1})^2}},
\]
(see \cite[Section 4.4]{chernov book}).  Notice that $C^s$ and $C^u$ are uniformly
transverse to each other and to the vertical and horizontal directions in $M$ as required.

The bounds on the first and second derivatives of $T$ required by
\eqref{eq:expansion} and \eqref{eq:2 deriv} are standard for such maps
(\cite[Section 4.4]{chernov book}).  Here, the index $n$ corresponds to the free flight
time $\tau(T^{-1}x)$.  For finite horizon, this has a uniform upper bound, while for infinite
horizon, the relation between $k$ and $n$ is satisfied with $\upsilon_0 = 1/4$
(\cite[Section 5.10]{chernov book}).

\smallskip
\noindent
{\em (H2)}.  We say a $\C^2$ curve $W$ in $M$ is stable if its
tangent vectors $\mathcal{T}_xW$ lie in $C^s(x)$ as defined above for each $x \in W$.
We call a stable
curve homogeneous if it is contained in a single homogeneity strip $\Ho_k$.
Since each stable curve $W$ has slope bounded away from infinity, we may identify
$W$ with the graph of a function of $r$, which we denote by $\vf_W(r)$.

By \cite[Proposition 4.29]{chernov book}, we may choose $B$ depending only
on $\tau_*$, $\K_*$ and $E_*$ such that if $\frac{d^2\vf_W}{dr^2} \le B$, then
each smooth component $W'$ of $T^{-1}W$ satisfies
$\frac{d^2\vf_{W'}}{dr^2} \le B$.

We define $\widehat \W^s$ to be the set of all stable homogeneous curves $W$
such that $\frac{d^2\vf_W}{dr^2} \le B$.  The invariance of the family $\C^s(x)$ as
well as the choice of $B$ guarantee that $\widehat \W^s$ is invariant as required.
The set of unstable curves $\widehat \W^u$ is defined similarly.

\smallskip
\noindent
{\em (H3)}.  Following \cite[Section 5.10]{chernov book}, we define the adapted norm
in the tangent space at $x \in M$ by
\[
\| v \|_* = \frac{\K(x) + |\mathcal{V}|}{\sqrt{1 + \mathcal{V}^2}} \| v \|, \; \; \;
\forall v \in C^s(x) \cup C^u(x)
\]
where, $v = (dr, d\vf)$ is a tangent vector, $\mathcal{V} = d\vf/dr$ and $\K(x)$
is the curvature of the scatterer at $x$.
Since the slopes of vectors in $C^s(x)$ and $C^u(x)$ are bounded away from $\pm \infty$,
we may extend $\| \cdot \|_*$ to all of $\mathbb{R}^2$ in such a way
that $\| \cdot \|_*$ is uniformly equivalent
to $\| \cdot \|$.   It is straightforward to check that for $v \in C^u(x)$,
\[
\frac{\| DT(x) v \|_*}{\| v \|_*} \ge 1 + \K_* \tau_* = \Lambda .
\]
Uniform expansion in $C^s(x)$ under $DT^{-1}(x)$ follows similarly.
Now \eqref{eq:step1} follows from \cite[Lemma 5.56]{chernov book} and
\eqref{eq:weakened step1} follows from \cite[Sublemma 3.5]{demers zhang}
with $\varsigma_0 = 1/6$.
From this point forward, we consider $k_0$ to be fixed.

\smallskip
\noindent
{\em (H4)}.  The bounded distortion constant $C_d$ in \eqref{eq:distortion stable}
and \eqref{eq:D u dist} depends only on the choice of $k_0$ from {\bf (H3)} and
the uniform hyperbolicity constants  $C_e$ and $\Lambda$
(\cite[Lemma 5.27]{chernov book}).

\smallskip
\noindent
{\em (H5)}.  For maps in $\F_1$, $DT(x) \equiv 1$ so we may take $\eta=1$.


\subsection{Proof of Theorems~\ref{thm:deform}}

Fix constants $\tau_*, \K_* >0$ and $E_* < \infty$ and consider
a configuration $Q_0 \in \Q_1(\tau_*, \K_*, E_*)$ with scatterers
$\Gamma_1, \ldots, \Gamma_d$.  Choose
$\gamma \le \frac 12 \min \{ \tau_*, \K_* \}$ and let
$\tQ \in \F_B(Q_0, E_*; \gamma)$ with scatterers $\tGamma_1, \ldots, \tGamma_d$.
Since $\ell(I_i) = |\partial \Gamma_i| = |\partial \tGamma_i|$ we may take
the corresponding functions $u_i, \tu_i$ to be arclength parametrizations
of $\partial \Gamma_i$ and $\partial \tGamma_i$ respectively.
We denote by $u'_i$ and $u''_i$ the first and second derivatives of
$u_i$ with respect to the arclength parameter $r$.
Then the curvature of $\partial \Gamma_i$ is simply given by
$\K(r) = \| u''_i(r) \|$ at each point $u_i(r) \in \partial \Gamma_i$,
and similarly for $\partial \tGamma_i$.

Thus on $\partial \tGamma_i$, we have by assumption on $\tQ$ and $\gamma$,
\[
\tilde{\K}(r) = \| \tu''_i \| = \| u''_i + \tu''_i - u''_i \| \ge \K(r) - \gamma \ge \K_*/2.
\]
Also, $\tau_{\min}(\tQ) \ge \tau_{\min}(Q_0) - \gamma \ge \tau_*/2$
since $\| u_i - \tu_i \| \le \gamma$.  Thus
$\F_A(Q_0, E_*; \gamma) \subset \F_1(\tau_*/2, \K_*/2, E_*)$.

Next we must show that $\tQ \in \F_A(Q_0, E_*; \gamma)$ represents a small
perturbation in the distance $d_\F(\cdot, \cdot)$.  We do this by first fixing
$\Gamma_2, \ldots, \Gamma_d$ and considering a deformation of
$\Gamma_1$ into $\tGamma_1$ such that $| u_1 - \tu_1 |_{\C^2} \le \gamma$.

Let $T_0$ be the map corresponding to $Q_0$ and let $T_1$ be the map corresponding
to $\tQ$.  We fix $x = (r, \vf) \in I_1 \times [-\pi/2, \pi/2]$ and compare $T_0^{-1}x$
with $T_1^{-1}x$.  To do this, we let $\Phi^0_t$ and $\Phi^1_t$ denote the flow on the tables
$Q_0$ and $\tQ$ respectively.  We denote by $\pi_0(x)$ the projection of $x$ onto the
flow space $\mathbb{T}^2 \times S^1$ corresponding to $Q_0$ and
by $\pi_0^q$ and $\pi_0^\theta$ the projections onto the position and angular coordinates
respectively.
Let $\tau_0(x)$ denote the
free flight time of $x$ under $\Phi^0_t$ and let $\K_0(\cdot)$ denote the curvature
of the scatterers in $Q_0$.  The analogous objects, $\pi_1, \pi_1^q, \pi_1^\theta,
\tau_1(\cdot)$ and $\K_1(\cdot)$ are defined for the table $\tQ$.

First suppose that $T_0^{-1}x$ and $T_1^{-1}x$ lie on the same scatterer $\Gamma_j$.
Notice that the trajectories $\Phi^0_{-t}(\pi_0x)$ and $\Phi^1_{-t}(\pi_1x)$ begin
from two points in $\mathbb{T}^2$ at most $\gamma$ apart and make an angle
of at most $\gamma$ with one another.  We decompose this motion into the sum of
(I) two parallel trajectories starting a distance $\gamma$ apart and (II) two
trajectories starting at the same point and making an angle $\gamma$.

\smallskip
\noindent
{\em I.  Parallel trajectories.} It is an elementary estimate that two parallel lines
a distance $\gamma$ apart will intersect a convex scatterer at a distance at most
\begin{equation}
\label{eq:parallel}
d_{\mathbb{T}^2}(\pi_0^q(T^{-1}_0x), \pi_1^q(T_1^{-1}x)) \le \sqrt{3 \gamma/\K_{\min}(\Gamma_j)}
\le \sqrt{3 \gamma/ \K_*},
\end{equation}
where $d_{\mathbb{T}^2}$ denotes distance on $\mathbb{T}^2$.

\smallskip
\noindent
{\em II. Nonparallel trajectories making an angle $\gamma \neq 0$.}  After time $t$ under the flow, the two trajectories will be
at most $t\gamma$ apart in $\mathbb{T}^2$.  Let $\tau(x_{-1}) = \max\{ \tau_0(T_0^{-1}x), \tau_1(T_1^{-1}x) \}$.
Then in the case of a finite horizon
Lorentz gas, by the same estimate as in \eqref{eq:parallel},
\begin{equation}
\label{eq:finite horizon}
d_{\mathbb{T}^2}(\pi_0^q(T^{-1}_0x), \pi_1^q(T_1^{-1}x)) \le \sqrt{3 \gamma \tau(x_{-1})/\K_{\min}(\Gamma_j)}
\le \sqrt{3 \gamma \tau_{\max} / \K_*}.
\end{equation}
In the infinite horizon case, define $\hat \tau = \gamma^{-1/3}$.  If
$\tau(x_{-1}) \le \hat \tau$, then 
\eqref{eq:finite horizon} implies
$d_{\mathbb{T}^2}(\pi_0^q(T^{-1}_0x), \pi_1^q(T_1^{-1}x)) \le \sqrt{3 / \K_*} \gamma^{1/3}$.
On the other hand, suppose $\tau_0(T_0^{-1}x) > \hat \tau$. Then $x$ lies in a cell $D_n$
such that $c^{-1} n \le \tau_0(T_0^{-1}y) \le c n$ for some $c >0$ and all $y \in D_n$, and
the width of $D_n$ in the stable direction is at most $C'/n$ (see \cite[Section 4.10]{chernov book}).  Thus
\beq
\label{eq:tau}
d_M(x, \Si_{-1}^{T_0}) \le C' n^{-1} \le C' c \tau_0^{-1}(T_0^{-1}x) \le C' c \hat \tau^{-1} \le C' c \gamma^{1/3} .
\eeq
An identical estimate holds if $\tau_1(T_1^{-1}x) > \hat \tau$.
Thus either $x \in N_{C\gamma^{1/3}}(\Si_{-1}^{T_0} \cup \Si_{-1}^{T_1})$ or
\begin{equation}
\label{eq:infinite horizon}
d_{\mathbb{T}^2}(\pi_0^q(T^{-1}_0x, \pi_1^q(T_1^{-1}x)) \le \sqrt{3  / \K_*} \gamma^{1/3}.
\end{equation}

\smallskip
\noindent
Concatenating these two estimates (I) and (II), we see that in terms of position coordinates,
$T^{-1}_0x$ and $T^{-1}_1x$ in $I_j$
are of order $\gamma^{1/2}$ in the finite horizon case and of order $\gamma^{1/3}$ in the
infinite horizon case.  Since the normal direction of $\Gamma_j$
varies smoothly with the position, we have
$d_M(T_0^{-1}x, T_1^{-1}x)$
of the same order.  Similar estimates hold when starting from $x \in \Gamma_j$ and
comparing images in $\Gamma_1$ and $\tGamma_1$.

In the case when $T_0^{-1}x$ and $T_1^{-1}x$ do not lie on the same scatterer $\Gamma_j$,
we must have $x \in N_{C\gamma^{1/3}}(\Si_{-1}^{T_0} \cup \Si_{-1}^{T_1})$
by the preceding arguments where $C = 4 \K_*^{-3/2}$ is sufficient.  We have
thus shown {\bf (C1)} holds with $\ve = C \gamma^{1/3}$.  Indeed, {\bf (C1)} holds with
$\ve = C \gamma^b$ for any $0 < b \le 1/3$ by the same argument.

We can consider the deformation of $d$ scatterers as the concatenation of errors induced
by deforming one scatterer at a time.  The preceding analysis holds with $C$ increased
by a factor of $d$.

Condition {\bf (C2)} is trivial to check since $J_\mu T_i \equiv 1$ for $i = 0,1$.

Next we prove {\bf (C4)}.  By \cite[eq. (2.26)]{chernov book},
$DT^{-1}_0(x) = \frac{-1}{\cos \vf(T_0^{-1}x)} A_0(x)$, where
\[
\scriptsize
A_0(x) =
\left[
\begin{array}{cc}
  \tau_0(T_0^{-1}x) \K_0(x) + \cos \vf(x) & - \tau_0(T_0^{-1}x) \\
  - \K_0(T_0^{-1}x) ( \tau_0(T_0^{-1}x) \K_0(x) + \cos \vf(x) ) - \K_0(x) \cos \vf(T_0^{-1}x)
  & \tau_0(T_0^{-1}x) \K_0(T_0^{-1}x) + \cos \vf(T_0^{-1}x) 
\end{array}
\right],
\]
and $DT_1^{-1}x = \frac{-1}{\cos \vf(T_1^{-1}x)} A_1(x)$, with a similar definition for $A_1(x)$.
Thus
\begin{equation}
\label{eq:DT diff}
\| DT_0^{-1}(x) - DT_1^{-1}x \| \le \Big| \frac{1}{\cos \vf(T_0^{-1}x)} - \frac{1}{\cos \vf(T_1^{-1}x)} \Big|
  \| A_0(x) \| + \frac{1}{\cos \vf(T_1^{-1}x)} \| A_0(x) - A_1(x) \|
\end{equation}
Note that $\| A_i(x) \|$ is bounded by a uniform constant times $\tau_i(T_i^{-1}x)$ and
\beq
\label{eq:A bound}
\| A_0(x) - A_1(x) \| \le K \tau(x_{-1})(d_M(T_0^{-1}x, T_1^{-1}x) + \gamma)
\eeq
where $K$ depends on $\tau_*$, $E_*$ and $\K_*$
and the $\gamma$ term is
due to possible differences in the curvatures $\K_0$ and $\K_1$ at the same point.
Notice that if $W \in \W^s$, then $|T_i^{-1}W| \le C |W|^{1/3}$ in the infinite horizon case
and $|T_i^{-1}W| \le C |W|^{1/2}$ in the finite horizon case.  Thus for $\delta < 1/k_0$, if
$T^{-1}_ix \in N_\delta(\Si_0)$, then $d_M(x, \Si_{-1}^{T_i}) \le C_t\delta^2$ where
$C_t$ is a uniform constant depending on the transversality of $C^s(x)$ with the
horizontal direction and of $\Si_{-1}^{T_i}$ with $C^s(x)$.

Now choose $\ve = \gamma^{a}$, where $a \le 1/3$ will be determined shortly.  Suppose
$x \notin N_\ve(\Si_{-1}^{T_0} \cup \Si_{-1}^{T_1})$.  Then by the above observation,
$\cos \vf(T_i^{-1}x) \ge C \ve^{1/2}$, $i =0,1$, and also by \eqref{eq:tau},
$\tau(x_{-1}) \le C \ve^{-1}$.  Thus recalling that
$d_M(T_0^{-1}x, T_1^{-1}x) \le C \gamma^{1/3}$, we estimate the
first term of \eqref{eq:DT diff},
\beq
\label{eq:cos diff}
\begin{split}
\| A_0(x) \| \Big| \frac{1}{\cos \vf(T_0^{-1}x)} - \frac{1}{\cos \vf(T_1^{-1}x)} \Big|
& \le \frac{K \tau(T_0^{-1}x) }{\cos \vf(T_0^{-1}x) \cos \vf(T_1^{-1}x)} | \cos \vf(T_1^{-1}x) - \cos \vf(T_0^{-1}x)| \\
& \le C\ve^{-2}  d_M(T_0^{-1}x, T_1^{-1}x) \le C' \gamma^{1/3-2a}  .
\end{split}
\eeq
To estimate the second term of \eqref{eq:DT diff}, we use \eqref{eq:A bound} to estimate,
\[
 \frac{1}{\cos \vf(T_1^{-1}x)} \| A_0(x) - A_1(x) \|
  \le C \ve^{-3/2} \gamma^{1/3} = C \gamma^{1/3 - 3a/2} .
\]

Putting these estimates together, we have
\[
\| DT_0^{-1}(x) - DT_1^{-1}(x) \| \le C''  \gamma^{1/3 - 2a}.
\]
Choosing $a = 2/15$ establishes {\bf (C4)}.

Condition {\bf (C3)} follows similarly
using the fact that
the stable Jacobian along $W \in \W^s$ is simply the norm of the tangent vector to $W$ times
$DT_i(x)$, $i = 0,1$.  The improved estimate in {\bf (C3)} comes from the fact that
instead of estimating \eqref{eq:cos diff} as above, we
must estimate instead
\[
\tau(x_{-1}) \left| \frac{\cos \vf(T_0^{-1}x)}{\cos \vf(T_1^{-1}x)} - 1 \right| \le C \ve^{-3/2}
d_M(T_0^{-1}x, T_1^{-1}x) \le C' \gamma^{1/3 -3a/2} = C' \gamma^{2/15} = C' \ve
\]
with our choice of $a = 2/15$.

If we restrict perturbations to the finite horizon case with horizons uniformly bounded by some
$\tau_{\max} < \infty$, then our estimates above improve by omitting a factor of $\ve^{-1}$
and $d(T_0^{-1}x , T_1^{-1}x) \le C \gamma^{1/2}$ by \eqref{eq:finite horizon}.  In
this case, the optimal choice of $b = 1/3$.


\subsection{Proof of Theorem~\ref{thm:random}}
\label{random proof}

We fix a class of maps $\F$ for which {\bf (H1)}-{\bf (H5)} hold with uniform constants
and choose $T_0 \in \F$.  Define
$X_\ve(T_0) = \{ T \in \F : d_{\F}(T,T_0) \le \ve \}$.
Recall the transfer operator $\Lp_{(\nu,g)}$ associated with the random process drawn
from $X_\ve(T_0)$ as defined in Section~\ref{random}.
Our first lemma is a generalization of Theorem~\ref{thm:close} which shows that the transfer
operator $\Lp_{(\nu,g)}$ is close to $\Lp_{T_0}$ in the norms we have defined.

\begin{lemma}
\label{lem:random close}
There exists $C>0$ such that if $\ve \le \ve_0$, then
$||| \Lp_{(\nu,g)} - \Lp_{T_0} ||| \leq C A \ve^\beta$.
\end{lemma}

\begin{proof}
Let $h \in \C^1(M)$, $W \in \W^s$ and $\psi \in \C^p(W)$ with $|\psi|_{W,0,p} \le 1$.
Then using \eqref{eq:final close},
\[
\begin{split}
\left| \int_W \Lp_{(\nu,g)}h \, \psi \, dm_W -  \int_W \Lp_{T_0}h \, \psi \, dm_W \right| &
= \left| \int_{\Omega} \int_W ( \Lp_{T_\omega} h(x) - \Lp_{T_0}h(x)) \psi(x) g(\omega, T_\omega^{-1}x) \, dm_W d\nu \right| \\
& \le \int_\Omega Cb^{-1} \ve^{\beta/2} \| h \| |g(\omega, \cdot)|_{\C^1(M)} d\nu(\omega)
\le Cb^{-1} A \ve^{\beta/2} \| h\|,
\end{split}
\]
where we have interchanged order of integration since $\int_W \Lp_{T_w}( h) \, \psi \, g(\omega, \cdot) \, dm_W$ is uniformly and absolutely integrable for each $\omega \in \Omega$
by Theorem~\ref{thm:uniform}.
\end{proof}

It remains to prove the uniform Lasota-Yorke inequalities for $\Lp_{\nu,g}$.
Let $\ob_n = (\omega_1, \ldots, \omega_n) \in \Omega^n$ and define
$T_{\ob_n} = T_{\omega_n} \circ \cdots \circ T_{\omega_1}$.
We first prove that the random compositions $T_{\ob_n}$ have the same properties
{\bf (H1)}-{\bf (H5)} as the maps $T_\omega \in \F$, with possibly modified
constants.

The singularity sets for $T_{\ob_n}$ are
$\Si_n^{T_{\ob_n}} = \cup_{k=1}^n T_{\omega_1}^{-1} \circ \cdots \circ T_{\omega_k}^{-1} \Si_0$,
for $n\ge 0$, and similarly for $\Si_{-n}^{T_{\ob_n}}$.  Thus the transversality properties
{\bf (H1)}
of $\Si_{-n}^{T_{\ob_n}}$ with respect to $C^s$ and $C^u$ hold due to the
uniformity of this transversality for all maps in $\F$.  The family $\W^s$ is
preserved under $T_{\ob_n}^{-1}$ since it is preserved by each map in the composition.

The uniform
expansion given by \eqref{eq:uniform hyp} of {\bf (H1)} also holds since
$DT_{\ob_n} = \prod_{k=1}^n DT_{\omega_k} \circ T_{\ob_{k-1}}$ and in the adapted
metric $\| \cdot \|_*$ given by {\bf (H3)}, the expansion holds with $C_e=1$
for each map in the composition.  Translating to the Euclidean norm at the last step,
we get {\bf (H1)} with $C_e$ depending only on the uniform constant relating the
adapted and Euclidean metrics.  Equations \eqref{eq:expansion} and \eqref{eq:2 deriv}
also hold trivially since they concern only one iterate of a map drawn from $\F$.
{\bf (H5)} follows for the same reason.

Due to the uniform expansion along stable and unstable leaves,
\eqref{eq:distortion stable} and \eqref{eq:D u dist} of {\bf (H4)} hold with a possibly
larger distortion constant $C_d^*$, again using the bounded distortion of each
map in the composition $T_{\ob_n}$.

Finally, we establish that the iteration of the one-step expansion
given in {\bf (H3)} holds for random sequences of maps in the class $\F$.
As in Section~\ref{preliminary}, for $W \in \W^s$ we define the $n$th generation
$\G_n^{\ob_n}(W) \subset \W^s$ of smooth curves in $T_{\ob_n}^{-1}W$.
The elements of $\G_n^{\ob_n}(W)$ are denoted by $W^n_i$ as before
and long and short pieces are defined similarly.
Analogously,
$\I^{\ob_n}_n(W^k_j)$ denotes the set of indices $i$ in generation $n$ such that
$W^k_j$ is the most recent long ancestor of $W^n_i$ under $T_{\ob_n}$.
Thus $\I_n^{\ob_n}(W)$ denotes
the set of curves that are never part of a curve that has grown to length $\delta_0/3$
at each time step $1 \le k \le n$.

\begin{lemma}
\label{lem:random growth}
Let $W \in \W^s$ and for $n \geq 0$, let $\I_n^{\ob_n}(W)$ and $\G_n^{\ob_n}(W)$
be defined as above.
There exist constants $C_1, C_2, C_3 >0$, independent of $W \in \W^s$ and
$\ob_n \in \Omega^n$, such that for
any $n\geq 0$,
\begin{itemize}
  \item[(a)] $\ds
\sum_{i \in \I_n^{\ob_n}(W)} |J_{W^n_i}T_{\ob_n}|_{\C^0(W^n_i)} \leq C_1 \theta_*^n $;
 \item[(b)] $\ds
\sum_{W^n_i \in \G_n^{\ob_n}(W)}  |J_{W^n_i}T_{\ob_n}|_{\C^0(W^n_i)} \le C_2     $;
 \item[(c)]  for any $0 \leq \varsigma \leq 1$,
$\ds
\sum_{W^n_i \in \G_n^{\ob_n}(W)} \frac{|W^n_i|^\varsigma}{|W|^\varsigma}
|J_{W^n_i}T_{\ob_n}|_{\C^0(W^n_i)} \le C_2^{1-\varsigma} $;
 \item[(d)] for $\varsigma > \varsigma_0$,
  $\ds
\sum_{W^n_i \in \G_n^{\ob_n}(W)} |J_{W^n_i}T_{\ob_n}|_{\C^0(W^n_i)}^\varsigma \le C_3^n$.
\end{itemize}
\end{lemma}

\begin{proof}
(a)   Fix $W\in \W^s$ and for $\ob_n \in \Omega^n$, define
$ \Z_n(W)
=\sum_{i\in \mathcal I_n^{\ob_n}(W)} |J_{W^n_i}T_{\ob_n}|_*$,
where $|J_{W^n_i}T_{\ob_n}|_*$ denotes the least contraction on $W^n_i$ under
$T_{\ob_n}$ measured in the metric induced by the adapted norm.
We will prove by induction on $n\in \mathbb{N}$ that  $\Z_{n}(W)\leq \theta_*^{n}$.
Then, since $\| \cdot \|_*$ is equivalent to $\| \cdot \|$, statement (a) follows.

Note that at each iterate between $1$ and $n$, every piece $W^n_i$, $i \in \I_n^{\ob_n}(W)$, is created by genuine cuts due to singularities and
homogeneity strips and not by any artificial subdivisions, since those are only made when a piece
has grown to length greater than $\delta_0$.
Thus we may apply
the one-step expansion \eqref{eq:step1} to conclude,
\begin{equation}
\label{thetaW}
 \Z_1(W)\leq \theta_*, \; \; \; \forall \; W \in \W^s.
\end{equation}

 Assume that $\Z_{n}(W)\leq \theta_*^{n}$ is proved for some $n\geq 1$ and all $W \in \W^s$.
We apply it to each component $W^1_i \in \G_1^{\omega_1}(W)$ such that $i \in \I_1^{\omega_1}(W)$.
Then
$ \Z_n(W^1_i) \leq \theta_*^{n}$ since $W^1_i \in \W^s$.

Given $\ob_n \in \Omega^n$, we use the notation
$\ob_{n-k}' = (\omega_n, \ldots, \omega_{n-k+1})$ so that we may split up
compositions $\ob_n = (\ob_{n-k}', \ob_k)$ into two pieces.
Given a sequence $\ob_{n+1}$,
we group the components of $W_i^{n+1}\in \G_{n+1}^{\ob_{n+1}}(W)$ with
 $i\in \I_{n+1}^{\ob_{n+1}}(W)$
 according to elements with index in $\I_1^{\omega_1}(W)$. More precisely,
 for $j \in \I_1^{\omega_1}(W)$, let
 $A_j = \{ i : W^{n+1}_i \in \G^{\ob_{n+1}}_{n+1}(W), T_{\ob_n'}W^{n+1}_i \subset W^1_j \}$.
 Note that $|J_{W^{n+1}_i}T_{\ob_{n+1}}|_* \le |J_{W^{n+1}_i}T_{\ob_n'}|_*
 |J_{W^1_j}T_{\omega_1}|_*$ whenever $T_{\ob_n'}W^{n+1}_i \subseteq W^1_j$.
Combining this and \eqref{thetaW} with the inductive hypothesis, we get
 \[
 \begin{split}
 \Z_{n+1}(W) &
=  \sum_{j \in \I_1^{\omega_1}(W)}\sum_{i \in A_j}  |J_{W^{n+1}_i}T_{\ob_{n+1}}|_*
\; \leq \; \sum_{j \in \I_1^{\omega_1}(W)}
\left( \sum_{i \in A_j} |J_{W^{n+1}_i}T_{\ob_n'}|_* \right)
|J_{W^1_j}T_{\omega_1}|_*   \\
&=\sum_{j \in \I_1^{\omega_1}(W)} \Z_n(W^1_j)\,\cdot
|J_{W^1_j}(T_{\omega_1})|_*
\; \leq \; \theta_*^{n+1}  .
\end{split}
\]

\smallskip
\noindent
(b)   Fix $W \in \W^s$ and $\ob_n \in \Omega^n$.
For any $0 \le k \le n$ and
$W^n_i \in \G_n^{\ob_n}(W)$, we have
\begin{equation}
\label{eq:long piece bound}
|J_{W^n_i}T_{\ob_n}|_{\C^0(W^n_i)} \le |J_{W^n_i}T_{\ob_{n-k}'}|_{\C^0(W^n_i)} |J_{W^k_j}T_{\ob_k}|_{\C^0(W^k_j)},
\end{equation}
whenever $T_{\ob_{n-k}'}W^n_i \subseteq W^k_j \in \G_k^{\ob_k}(W)$.

Now grouping $W^n_i \in \G_n^{\ob_n}(W)$ by most recent long ancestor
$W^k_j \in L_k^{\ob_k}(W)$ as described
in Section~\ref{preliminary} and using \eqref{eq:long piece bound},
we have
\[
\begin{split}
\sum_i & |J_{W^n_i}T_{\ob_n}|_{\C^0(W^n_i)}
= \sum_{k =0}^n \sum_{W^k_j \in L_k^{\ob_k}(W)} \sum_{i \in \I_n^{\ob_n}(W^k_j)}
|J_{W^n_i}T_{\ob_n}|_{\C^0(W^n_i)}    \\
& \le \sum_{k=1}^{n} \sum_{W^k_j \in L_k^{\ob_k}(W)} \Big( \sum_{i \in \I_n^{\ob_n}(W^k_j)}
 |J_{W^n_i}T_{\ob_{n-k}'}|_{\C^0(W^n_i)} \Big) |J_{W^k_j}T_{\ob_k}|_{\C^0(W^k_j)}
 \; + \; \sum_{i \in \I_n^{\ob_n}(W)}  |J_{W^n_i}T_{\ob_n}|_{\C^0(W^n_i)} ,
 \end{split}
\]
where we have split off the terms involving $k=0$ that have no long ancestor.  We have
\[
|J_{W^k_j}T_{\ob_k}|_{\C^0(W^k_j)} \le (1+C_d^*)|T_{\ob_k}W^k_j| |W^k_j|^{-1} \le 3 \delta_0^{-1}
(1+C_d^*)|T_{\ob_k}W^k_j|
\]
since $|W^k_j| \ge \delta_0/3$.
Since $\I_n^{\ob_n}(W^k_j)$ and $\I_{n-k}^{\ob_{n-k}'}(W^k_j)$ correspond to the same set of short pieces in the
$(n-k)^{\mbox{\scriptsize th}}$ generation of $W^k_j$, we apply part (a) of this lemma
to each of these sums.  Thus,
\[
\begin{split}
\sum_i & |J_{W^n_i}T_{\ob_n}|_{\C^0(W^n_i)}
\le \sum_{k=0}^{n-1} \sum_{W^k_j \in L_k^{\ob_k}(W)} C_1 \theta_*^{n-k} |J_{W^k_j}T_{\ob_k}|_{\C^0(W^k_j)} \; + \; C_1 \theta_*^n \\
 & \le C \delta_0^{-1} \sum_{k=0}^{n-1} \sum_{W^k_j \in L_k^{\ob_k}(W)} \theta_*^{n-k}|T_{\ob_k}W^k_j| + C \theta_*^n
 \; \le \; C \delta_0^{-1} |W| \sum_{k=0}^{n-1} \theta_*^{n-k} + C \theta_*^n ,
 \end{split}
\]
which is uniformly bounded in $n$.

\smallskip
\noindent
(c) follows from (b) by an application of Jensen's inequality and (d) follows from
{\bf (H3)} using an inductive argument similar to the proof of (a).
\end{proof}

We complete the proof of Theorem~\ref{thm:random} via the following proposition.
The uniform Lasota-Yorke inequalities of Theorem~\ref{thm:uniform} then follow
from the argument given at the beginning of Section~\ref{uniform}.

\begin{proposition}
\label{prop:random ly}
Choose $\ve \le \ve_0$ sufficiently small that $\sigma(1+\ve)<1$ and let $\Delta(\nu,g) \le \ve$.
There exists a constant $C$, depending on $a$, $A$, and {\bf (H1)}-{\bf (H5)} such that
for $h \in \B$ and $n \ge 0$,
\begin{eqnarray}
|\Lp_{(\nu,g)}^n h|_w & \le & C \eta^n |h|_w   \label{eq:random weak norm} \\
\| \Lp_{(\nu,g)}^n h \|_s & \le &C \eta^n ( \theta_*^{(1-\alpha)n} + \Lambda^{-qn}) \|h\|_s + C \delta_0^{-\alpha} \eta^n |h|_w
\label{eq:random stable norm} \\
\| \Lp_{(\nu,g)}^n h \|_u &\le & C \eta^n \Lambda^{-\beta n} \| h\|_u + C \eta^n C_3^n \|h \|_s
\label{eq:random unstable norm}
\end{eqnarray}
\end{proposition}

\begin{proof}
We record for future use,
\[
\Lp^n_{(\nu,g)}h(x) = \int_{\Omega^n} h \circ T_{\ob_n}^{-1} (J_\mu T_{\ob_n}\circ T_{\ob_n}^{-1})^{-1}
\prod_{j=1}^n g(\omega_j, T_{\omega_j}^{-1} \circ \cdots \circ T_{\omega_n}^{-1}x) d\nu^n(\ob_n)
.
\]
The proofs of the inequalities are the same as in Section~\ref{uniform} except that we have the
additional function $g(\omega,x)$.  We show how to adapt the estimates of Section~\ref{uniform}
to the operator $\Lp_{(\nu,g)}$ in the case of the strong stable norm.  The other estimates
are similar.

\medskip
\noindent
{\bf Estimating the Strong Stable Norm.}
Following Section~\ref{stable norm}, we write,
\begin{equation}
\begin{split}
\label{eq:random split}
\int_W & \Lp^n_{(\nu,g)} h \, \psi dm_W
= \int_{\Omega^n} \sum _i \left\{ \int_{W^n_i} h (\psi \circ T_{\ob_n} - \bp_i)
(J_\mu T_{\ob_n})^{-1} J_{W^n_i}T_{\ob_n} \prod_{j=1}^ng(\omega_j, T_{\ob_{j-1}}x) dm_W \right. \\
& \left. + \bp_i \int_{W^n_i} h (J_\mu T_{\ob_n})^{-1} J_{W^n_i}T_{\ob_n} \prod_{j=1}^n
g(\omega_j, T_{\ob_{j-1}}x) dm_W \right\} d\nu^n(\ob_n),
\end{split}
\end{equation}
where $\bp_i = |W^n_i|^{-1} \int_{W^n_i} \psi \circ T_{\ob_n} \, dm_W$.
Since for each $\ob_n$, $T_{\ob_n}$ satisfies properties {\bf (H1)}-{\bf (H5)}
with uniform constants, we may use the estimates of Section~\ref{uniform}.
Accordingly, $|\psi \circ T_{\ob_n} - \bp_i|_{\C^q(W^n_i)} \le C \Lambda^{-qn} |W|^{-\alpha}$
using \eqref{eq:C1 C0}.  Define $G_{\ob_n}(x) = \prod_{j=1}^n g(\omega_j, T_{\ob_{j-1}}x)$.
We estimate the first term of \eqref{eq:random split} using \eqref{eq:first stable}
\begin{equation}
\begin{split}
\label{eq:random first}
\sum_i & \int_{W^n_i} h (\psi \circ T_{\ob_n} - \bp_i) \, (J_\mu T_{\ob_n})^{-1} J_{W^n_i}T_{\ob_n}
G_{\ob_n} \, dm_W \\
& \le \sum_i C \| h\|_s |W_i|^\alpha |(J_\mu T_{\ob_n})^{-1} J_{W^n_i}T_{\ob_n}|_{\C^q(W^n_i)}
|\psi \circ T_{\ob_n} - \bp_i|_{\C^q(W^n_i)} |G_{\ob_n}|_{\C^q(W^n_i)} \\
& \le C \| h \|_s \Lambda^{-qn} \eta^n \sum_i \frac{|W^n_i|^\alpha}{|W|^\alpha}
|J_{W^n_i}T_{\ob_n}|_{\C^0(W^n_i)}  |G_{\ob_n}|_{\C^q(W^n_i)} .
\end{split}
\end{equation}
The only new term here is $|G_{\ob_n}|_{\C^q(W^n_i)}$ which is addressed by the following
lemma.

\begin{sublem}
\label{lem:G}
There exists $C>0$, independent of $W$ and $\ob_n$, such that if $W^n_i \in \G_n^{\ob_n}(W)$, then
\[
|G_{\ob_n}|_{\C^1(W^n_i)} \le  C G_{\ob_n}(x) \; \; \mbox{for any } x \in W^n_i .
\]
\end{sublem}
\begin{proof}[Proof of Sublemma]
For $x,y \in W^n_i$,
\[
\begin{split}
\log \frac{\prod_{j=1}^n g(\omega_j, T_{\ob_{j-1}} x)}{\prod_{j=1}^n g(\omega_j, T_{\ob_{j-1}}y)}
& \le \sum_{j=1}^n a^{-1} |g(\omega_j, \cdot)|_{\C^1(M)} d(T_{\ob_{j-1}}x, T_{\ob_{j-1}}y) \\
& \le \sum_{j=1}^\infty a^{-1} A C_e \Lambda^{-n}d(x,y) =: c_0 d(x,y) ,
\end{split}
\]
using properties (i) and (iii) of $g$.  The distortion bound yields the lemma with $C = c_0 e^{c_0}$.
\end{proof}
We estimate \eqref{eq:random first} using the sublemma and Lemma~\ref{lem:random growth}(c),
\begin{equation}
\label{eq:random contract}
\sum_i  \int_{W^n_i} h (\psi \circ T_{\ob_n} - \bp_i) \, (J_\mu T_{\ob_n})^{-1} J_{W^n_i}T_{\ob_n}
G_{\ob_n} \, dm_W
\le C \| h \|_s \eta^n \Lambda^{-qn} G_{\ob_n}(x_0),
\end{equation}
where $x_0$ is some point in $T_{\ob_n}^{-1}W$.

Similarly, we estimate the second term in \eqref{eq:random split} using
\eqref{eq:second stable}.   In each term, $G_{\ob_n}$ plays the role of a test function and we
replace the occurrences of $|G_{\ob_n}|_{\C^p(W^n_i)}$ and $|G_{\ob_n}|_{\C^q(W^n_i)}$
as appropriate according to Sublemma~\ref{lem:G}.  Thus following \eqref{eq:second stable},
we write,
\[
\sum_i \bp_i \int_{W^n_i} h (J_\mu T_{\ob_n})^{-1} J_{W^n_i}T_{\ob_n} G_{\ob_n} dm_W
\le C(\delta_0^{-\alpha} \eta^n |h|_w + \theta_*^{(1-\alpha)n} \eta^n \|h\|_s) G_{\ob_n}(x_0) ,
\]
choosing the same $x_0$ as in \eqref{eq:random contract}.
Now Combining this expression with \eqref{eq:random contract} and \eqref{eq:random split},
we obtain
\[
\int_W \Lp_{T_{\ob_n}} h \psi \, dm_W \le C \eta^n
(\|h\|_s(\Lambda^{-qn} + \theta_*^{(1-\alpha)n})
+ \delta_0^{-\alpha} |h|_w) \prod_{j=1}^n g(\omega_j, T_{\ob_{j-1}}x_0) .
\]
We integrate this expression one $\omega_j$ at a time, starting with $\omega_n$.
Notice that $\int_\Omega g(\omega_n, T_{\ob_{n-1}}x_0) d\nu(\omega_n) =1$ by
property (ii) of $g$ since $T_{\ob_{n-1}}$ is independent of $\omega_n$.  Similarly,
each factor in $G_{\ob_n}(x_0)$ integrates to 1 so that
\[
\| \Lp^n_{(\nu,g)} h \|_s \le C \|h\|_s \eta^n (\Lambda^{-qn} + \theta_*^{(1-\alpha)n})
+ C \delta_0^{-\alpha} \eta^n |h|_w
\]
which is the required inequality for the strong stable norm.  The inequalities for
the weak norm and the strong unstable norm follow similarly, always using
Sublemma~\ref{lem:G}.
\end{proof}


\section{Proofs of Applications:  External Forces with Kicks and Slips}
\label{kick}

In this section we prove Theorem \ref{thm:C1} and \ref{thm:C2} for the perturbed dispersing billiards under external forces with kicks and slips.
To simplify the analysis, for any fixed force $\mathbf{F}$, we will consider our  system,
denoted as $T_{\bF,\bG}$, as a perturbation of the map $T_{\bF, \mathbf{0}}$.
We say a constant $C$ is uniform if $C=C(\eps_1, \tau_*, \cK_*, E_*)$, where
$\eps_1, \tau_*, \cK_*$ and $E_*$ are from {\bf (A2)} and {\bf (A3)}.

We begin by reviewing some properties of $T_\bF = T_{\bF, \mathbf{0}}$ proved in
\cite{Ch01} and proving some additional ones that we shall need.

\subsection{Properties of $T_\bF$}
\label{flow review}

We assume the setup described in Section~\ref{concrete}.B, which is the billiard flow
given by \eqref{flowf} and \eqref{reflectiong} with $\bG = \mathbf{0}$.

Let $\bx=(\bq, \theta)\in \cM$ be any phase point with position $\bq$, and $V\in \cT_\bx \cM$ a tangent vector at $\bx$. Pick a small number $\delta_0>0$ and a $C^{3}$ curve $c_s(0)=(\bq_s, \theta_s)\subset \cM$ tangent to the vector $V$, such that $c_0=\bx$ and $\frac{d c_s}{ds}|_{s=0}=V$, and $s\in [0,\delta_0]$. Now we define $c_s(t)=\Phi^t c_s(0)$, for any $t\geq 0$. Since $\tau$ is the free path function, we have  $d\tau=p dt$. In the calculation below, we denote differentiation with respect to $s$ by primes and that with respect to  $\tau$ by dots. In particular, $\dot{c}_s(t)=(\dot \bq, \dot \theta)=( \bv, h)$, where
$\bv = \p/p = (\cos \theta, \sin \theta)$ and
$h=h(\bq,\theta)$ is  the geometric curvature of the billiard trajectory with initial condition $(\bq,\theta)$ on the table.

If we assume $t_s$ to be the time that the trajectory of $c_s(0)$ hits the wall of the billiard table, then $\{c_s(t)\,|\, t\in [0, t_s], s\in [0, \delta_0]\}$ is a $C^{3}$ smooth $2$-d manifold in $\cM$. We introduce two quantities $u=\bq'\cdot \bv$, and $w=\bq'\cdot \bv^{\perp}$, where $\bv^{\perp}=(-\sin\theta, \cos\theta)$. Clearly $\bq'=u \bv+w \bv^{\perp}$. Now let  $\kappa=(\theta'-uh)/w$. We consider two vectors of the surface $U=(\bv,h)$ and $R=(\bv^{\perp}, \kappa)$. Clearly $\dot c_s=U$ and $c_s'=u U+w R$.   Define  $p_U=\text{grad}(p) \cdot U$, $p_R=\text{grad} (p)\cdot R$, and $h_U=\text{grad} (h) \cdot U$, $h_R=\text{grad} (h)\cdot R$, respectively.
Then it is straight forward to check that
\beq
\label{p'h'}
p'=\text{grad}(p)\cdot c'_s = p_U u+p_R w \qquad
h'=h_U u+h_Rw \qquad \text{and } \; \; \theta'=\kappa w+h u .
\eeq
In addition $ \dot p= p_U $ and $\dot h= h_U $.
The derivation of these formulas can be found in  \cite{Ch01}.
The following lemma was proved in \cite[Lemmas 3.1, 3.2]{Ch01}.

\begin{lemma}[\cite{Ch01}]
The evolution of the quantities $\kappa$ and $w$ between collisions is given by the equations
\beq\label{kappadot}
\dot\kappa=-\kappa^2 + a + b\kappa \,\,\,\,\,\,\text{ and }\,\,\,\,\,\,\,\dot w=\kappa w,
\eeq
where $a=a(h)$, $b=b(h)$ are smooth functions whose $\C^0$ norms are bounded by $c_0\ve_1$
for some uniform $c_0>0$.  Furthermore, at the moment of collision,
\beq\label{pm}
u^+=u^-, \; \; w^+=-w^-\,\,\,\,\text{ and }\,\,\,\, \kappa^+=\kappa^-+\frac{2\cK(r)+(h^++h^-)\sin\varphi}{\cos\varphi} .
\eeq
In addition the derivative of $r$ and $\varphi$ satisfies
\beq\label{drdphi}
dr/ds=\mp w^{\pm}/\cos\varphi\,\,\,\,\,\,\,\text{ and }\,\,\,\,\,\,\,d\varphi/dr=\mp\cK(r)+\kappa^{\pm}\cos\varphi\mp h^{\pm}\sin\varphi .
\eeq
\end{lemma}

We will calculate the differential of the map $T_{\bF}$ (which is not contained in
\cite{Ch01}). It follows from (\ref{kappadot}) that
\beq
\label{ddotw}
\frac{d \dot w}{d\tau}=\frac{d}{d\tau}(\kappa w)=\dot\kappa w+\kappa \dot w=\kappa \dot w-\kappa^2 w+(a + b\kappa) w= aw + b\dot w .
\eeq
   This implies that
 \beq\label{contw1}   \left\{
     \begin{array}{ll}
      \dot w(\tau)=\dot w(0)+\int_{0}^{\tau} aw + b\dot w\, d\gamma \\
     w(\tau)=w(0)+\dot w(0) \tau+\int_0^{\tau} \int_0^{\xi} aw + b\dot w \,d\gamma\,d\xi
     \end{array}
   \right.   .
   \eeq
At the moment of collision, (\ref{pm}) implies that
\beq\label{dotwpm}
 \left\{
     \begin{array}{ll}
      w^+=-w^-\\
      \dot w^+=-\dot w^--\frac{2\cK+(h^++h^-)\sin\varphi}{\cos\varphi}w^-
    \end{array}
   \right.   .
      \eeq

In addition (\ref{drdphi}) implies that
\beq
\label{dphids}
\frac{d\varphi}{ds}=\frac{\cK(r)+h^{\pm}\sin\varphi}{\cos\varphi}w^{\pm} \mp \dot w^{\pm} .
\eeq

\begin{lemma}
\label{wtaubound}
For $x = (r, \vf)$, let $\tau_1(x)$ denote the distance to the next collision under the flow.
There exist constants $\hat C_1, \hat C_2 > 0$ independent of $x$, such that
$|w(\tau)|$ and $|\dot w(\tau)|$ are uniformly bounded from above by
$\hat C_1 |w^+(0)| + \hat C_2|\dot w^+(0)|$ for $\tau \in [0,\tau_1(x)]$.
\end{lemma}
\begin{proof}
We fix $x$ and abbreviate $\tau_1(x)$ as $\tau_1$.
We begin by adapting \cite[Lemma 3.4]{Ch01}, to show that if  for some
 $\tau_0 \in [0,\tau_1)$,
 $\kappa(\tau_0)$ is bounded away from zero, then
 $\kappa$ is bounded away from zero and infinity on $[\tau_0,\tau_1]$.
 More precisely, (\ref{kappadot}) implies that if $\kappa > 0$, then
$$
-(\kappa + \ve_2)^2 \le
\dot \kappa=-\kappa^2+b \kappa+a=-(\kappa-\frac{b}{2})^2+\frac{b^2}{4}+a\leq -(\kappa-c_0\eps_1)^2+ \eps_2^2
$$
 where $\eps_2^2 = 2c_0\eps_1.$

So if we assume that  for some $\tau_0 \in [0, \tau_1)$, $\kappa^+(\tau_0) > c_1$ for a fixed
$c_1 > 5 \sqrt{\ve_0}$, then
we may integrate these inequalities to obtain
$$
\frac{1}{(\kappa^+(\tau_0) + \ve_2)^{-1} + (\tau - \tau_0)} - \ve_2 \le \kappa(\tau)
\le \ve_2 \frac{Ae^{2 \ve_2 (\tau - \tau_0)} + 1}{Ae^{2\ve_2(\tau- \tau_0)} -1} +  c_0\ve_1,
$$
where $A = (\kappa^+(\tau_0) - c_0 \ve_1 + \ve_2)/(\kappa^+(\tau_0) - c_0 \ve_1 - \ve_2)$.
Then since $\ve_0$ is small compared to $\kappa^+(\tau_0)$, this reduces to
\beq
\label{kappabound}
\frac{1}{(\kappa^+(\tau_0))^{-1} + (\tau - \tau_0)} - \ve_3 \le \kappa(\tau) \le
\frac{1}{(\kappa^+(\tau_0))^{-1} + (\tau - \tau_0)} +  \ve_3
\eeq
where $\ve_3 = 2\ve_2 + 2c_0 \ve_1$.

 Now (\ref{kappadot}) implies that for any $0 \le \tau'  < \tau \le \tau_1$,
 \beq
 \label{convw}
 w(\tau)=w(\tau')\exp\left(\int_{\tau'}^{\tau} \kappa d\gamma \right).
 \eeq
Also, (\ref{ddotw}) implies that
 $\frac{\dot w}{w}d\ln \dot w=(a+b\kappa) \,d\tau$ and
since $\dot w=\kappa w$, we integrate this to obtain,
\beq
\label{1stdotwt}
\dot w(\tau)=\dot w(0) \exp\left(\int_{0}^{\tau} (\frac{a}{\kappa}+b)\, d\gamma \right)
\; \; \; \mbox{for any $\tau \in [0,\tau_1]$.}
\eeq
Integrating again, it follows that
\beq
\label{2ndwtau}
w(\tau)=w(0)+\dot w(0)\int_0^{\tau} \exp\left(\int_{0}^{\xi} (\frac{a}{\kappa}+b)\, d \gamma \right)\,d\xi .
\eeq
This implies that both $w(\tau), \dot w(\tau)$  are functions of $(w^+(0), \dot w^+(0))$.

 To show that $|w|$ and $|\dot w|$ are uniformly bounded, we consider three cases.

 \smallskip
 \noindent
 Case I:  $\kappa$ is finite on $[0, \tau_1)$ and $\kappa(\tau) < 1/\tau_{\min}$
 for all $\tau \in [0, \tau_1)$ ($\kappa$ can be positive or negative).
 Then by \eqref{convw}, $|w(\tau)| \le |w(0)| e^{\tau/\tau_{\min}}
 \le |w(0)| e^{\tau_{\max}/\tau_{\min}}$ for all $\tau \in [0, \tau_1]$.

Once we know $|w|$ is bounded on $[0,\tau_1]$, we may use it to bound $|\dot w|$ as follows.
We integrate \eqref{ddotw} using
the integrating factor $\exp(- \int_0^\tau b \, d\gamma)$ to obtain,
\beq
\label{eq:wint}
\dot w(\tau) = \dot w(0) e^{\int_0^\tau b \, d\gamma} + e^{\int_0^\tau b \, d\gamma}
\int_0^\tau aw(\xi)
e^{ - \int_0^\xi b \, d\gamma } \, d\xi .
\eeq
Thus
\beq
\label{eq:dotw}
| \dot w(\tau)| \le |\dot w^+(0)| e^{c_0\ve_1 \tau_{\max}} + |w^+(0)|e^{(2 c_0 \ve_1 + 1/\tau_{\min})\tau_{\max}}  c_0 \ve_1 \tau_{\max} =: C_1 |\dot w^+(0)| + C_2 |w^+(0)| .
\eeq

 \smallskip
 \noindent
 Case II: $\kappa$ is finite on $[0,\tau_1)$, $\kappa(\tau_0) \ge  1/\tau_{\min}$ for some
 $\tau_0 \in [0, \tau_1]$ and
 $\tau_0$ is the least $\tau$ in the interval with this property.  Then
 by \eqref{kappabound}, $\kappa(\tau) \geq (\tau_{\min} + 2\tau_{\max})^{-1}$ for
 all $\tau \in [\tau_0, \tau_1]$.  As a consequence, by \eqref{2ndwtau},
 \beq
 \label{half}
 |w(\tau)| \le |w(\tau_0)| + |\dot w(\tau_0)| \tau_{\max}
 e^{c_0 \ve_1 (\tau_{\min} + 2\tau_{\max} + 1)\tau_{\max}}
 \eeq
 for each $\tau \in [\tau_0, \tau_1]$.  On the other hand, for $\tau \in [0, \tau_0]$, we have
 $\kappa(\tau) \le 1/\tau_{\min}$, so that both $|w(\tau)|$ and $|\dot w(\tau)|$ are uniformly
 bounded on this interval by Case I.  This together with \eqref{half} proves
 Case II for $|w|$.  The estimate for $|\dot w|$ follows again from \eqref{eq:wint} and \eqref{eq:dotw}.

\smallskip
\noindent
Case III: $\kappa(\tau_0) = \pm \infty$ for some $\tau_0 \in (0, \tau_1)$.
According to \eqref{kappadot} and \eqref{kappabound}, the only way this case can occur is if
$\kappa$ reaches $-\infty$ in finite time and changes from $-\infty$ to $\infty$ at $\tau_0$.
\eqref{convw} implies in particular that $w(\tau_0)=0$.

On the interval $[0, \tau_0]$, $\kappa$ clearly satisfies the assumption of Case I so that
both $|w|$ and $|\dot w|$ are uniformly bounded as in the statement of the lemma on this interval.
Indeed, this is true on any interval in which $\kappa$ remains negative.  Thus the only case
left to consider is when $\kappa(\tau) > 0$ for $\tau \in (\tau_0, \tau_1]$.

In this case, \eqref{kappadot} guarantees that $\kappa$ initially decreases and
\eqref{kappabound} guarantees that $\kappa(\tau) \geq \tau_{\min}^{-1}$ on this interval.
Thus by \eqref{2ndwtau}, we estimate as in \eqref{half} to bound
$|w|$ by a linear combination of $|w(\tau_0)|$ and $|\dot w(\tau_0)|$.  But since
these two quantities are in turn bounded by $|w^+(0)|$ and $|\dot w^+(0)|$ by the
previous paragraph, the proof of Case III is complete for $|w|$.  The estimate
on $|\dot w|$ now follows again from \eqref{eq:wint} and \eqref{eq:dotw}.
\end{proof}

Combining the above facts,  we can show the following.

\begin{lemma} If we denote $x_1=(r_1, \varphi_1)=T_{\bF} x$, then there exits $C=C( \cK_*,\tau_*)>0$ such that for any unit vector $(dr/ds, d\varphi/ds)$,
\beq\label{DTeps0}
\left\{
     \begin{array}{ll}
    -\cos\varphi_1\frac{dr_1}{ds}=\left(\cos\varphi+\tau \cK+a_1\right)\frac{dr}{ds}+(\tau+a_2)\frac{d\varphi}{ds}\\
    -\cos\varphi_1\frac{d\varphi_1}{ds}=\left(\tau \cK_1\cK+\cK_1\cos\varphi+\cK\cos\varphi_1 +b_1\right) \frac{dr}{ds}+\left(\cK_1\tau+\cos\varphi_1+b_2\right)\frac{d\varphi}{ds}     \end{array}
   \right.
\eeq
where $a_i\leq C\eps_1$ and $b_i\leq C\eps_1$, for $i=1,2$.
In addition
\beq
\label{detTF}
(1-C\eps_1)\frac{\cos\varphi}{\cos\varphi_1}\leq |\det D_xT_\bF|
\leq (1+C\eps_1)\frac{\cos\varphi}{\cos\varphi_1}
\eeq
\end{lemma}
\begin{proof} Let $x_1=T_{\bF} x$, and $\tau_1(x)$ be the length of the free path of $x$.
By \eqref{1stdotwt} and \eqref{2ndwtau}, there exists a linear transformation $D_x$
such that
\beq
\label{Dx}
D_x(w^+, \dot w^+)^T=(w^-_1,\dot w^-_1)^T
\eeq
where $w^-_1=w^-(\tau_1)$ and $\dot w^-_1=\dot w^-(\tau_1)$.
Indeed, by Lemma~\ref{wtaubound}, there exist smooth
functions $c_i$, $i=1, \ldots 4$ with $|c_i| \le C \ve_1$ for some $C = C(\K_*, \tau_*) >0$ such that
\beq
\label{III}
I:=\int_0^{\tau} aw + b \dot w \, d\gamma = c_1 w^+(0)+c_2\dot w^+(0),\,\,\,\,
II:=\int_0^{\tau}\int_0^{\xi} aw + b \dot w \,d\gamma d\xi=c_3 w^+(0)+c_4\dot w^+(0),
\eeq
so that using \eqref{contw1}, we may write $D_x$ as
\beq
\label{Dxg}
D_{x}=\left(\begin{array}{cc}1+c_3&\tau+c_4\\c_1&1+c_2\\\end{array}\right) .
\eeq

Using (\ref{dotwpm}) and (\ref{dphids}), the differential of $DT_{\bF}$ satisfies
\beq
\label{DTepsmatrix}
DT_{\bF} = N_{x_1}^{-1} L_{x_1} D_x N_{x}
\eeq
where 
$$N_{x}=-\left(\begin{array}{cc}\cos\varphi &0\\\cK+h^+\sin\varphi
    &1\\\end{array}\right)$$ 
    is the coordinate transformation matrix on $\cT_x M$, such that $(w^+(0), \dot w^+(0))^T=N_x (dr/ds, d\varphi/ds)^T$, and
\beq
\label{LU}
L_{x_1}=\left(\begin{array}{cc}-1&0\\-\tfrac{2\cK_1+(h_1^++h_1^-)\sin\varphi_1}{\cos\varphi_1}&-1\\\end{array}\right)\,\,\,\,\,\,\,\text{ and }\,\,\,\,\, N_{x_1}^{-1}=\left(\begin{array}{cc} -\frac{1}{\cos\varphi_1}&0\\ \frac{\cK_1+h_1^+\sin\varphi_1}{\cos\varphi_1}&-1\\
\end{array}\right) .
\eeq

Now combining (\ref{contw1}) with \eqref{III} and  (\ref{DTepsmatrix}), we get
\begin{align*}
    -\cos\varphi_1\frac{dr_1}{ds}&=\left(\cos\varphi+\tau \cK+\tau h^+\sin\varphi\right)\frac{dr}{ds}+\tau\frac{d\varphi}{ds}-II\\
        &=\left(\cos\varphi+\tau \cK+\tau h^+\sin\varphi\right)\frac{dr}{ds}+\tau\frac{d\varphi}{ds}-c_3w^+-c_4\dot w^+\\
            &=\left(\cos\varphi+\tau \cK+a_1\right)\frac{dr}{ds}+(\tau+a_2)\frac{d\varphi}{ds}
    \end{align*}
    where $a_1=c_3\cos\varphi+c_4(\cK+h^+\sin\varphi)+\tau h^+\sin\varphi$ and $a_2=c_4$.
 Similarly we obtain
    \begin{align*}-\cos\varphi_1\frac{d\varphi_1}{ds}&=-(\cK_1+h_1^-\sin\varphi_1)w^-_1-\cos\varphi_1\dot w^-_1\\
    &=-(\cK_1+h_1^-\sin\varphi_1)(w^++\dot w^+\tau+II)-\cos\varphi_1(\dot w^++I)  \\
    &=\left[(\cK_1+h_1^-\sin\varphi_1)\cos\varphi+(\tau(\cK_1+h_1^-\sin\varphi_1)+\cos\varphi_1)(\cK+h^+\sin\varphi)\right] \frac{dr}{ds}\\
    &\,\,\,\,\,\,+\left(\tau\cK_1+\tau h_1^-\sin\varphi_1+\cos\varphi_1\right)\frac{d\varphi}{ds}-II(\cK_1+h_1^-\sin\varphi_1)-I\cos\varphi_1\\
      &=\left(\tau \cK_1\cK+\cK_1\cos\varphi+\cK\cos\varphi_1+b_1 \right) \frac{dr}{ds}+\left(\cK_1\tau+\cos\varphi_1+b_2\right)\frac{d\varphi}{ds}     \end{align*}
where
\begin{align*}
b_1&=(\cos\varphi+\tau\cK)h_1^-\sin\varphi_1+\cos\varphi_1\left(c_1\cos\varphi+c_2\cK+(1+c_2)h^+\sin\varphi\right)\\
&\,\,\,\,\,\,+\left(c_3\cos\varphi+\tau h^+\sin\varphi+c_4(\cK+h^+\sin\varphi)\right)(\cK_1+h_1^-\sin\varphi_1)\end{align*}
and $b_2=  (\tau+c_4) h_1^-\sin\varphi_1
+c_4\cK_1+c_2\cos\varphi_1$.

Now we use the assumption that the quantities $\cK, \tau$ are uniformly bounded from above, and $|h^{\pm}|=\cO(\eps_1)$, to obtained that for any unit vector $(dr/ds, d\varphi/ds)$, the quantities $|a_i|\leq C\eps_1$ and $|b_i| \leq C\eps_1$, $i=1,2$, for some uniform $C>0$.

Finally we use (\ref{DTepsmatrix}) to calculate the determinant of the differential $D_x T_{\bF}$,
\beq
\label{eq:full detTf}
\begin{split}
\det D_x T_{\bF}&=\det N_{x_1}^{-1} \cdot \det L_{x_1} \cdot \det D_x\cdot \det N_x=\frac{\cos\varphi}{\cos\varphi_1} \det D_x\\
&=\frac{\cos\varphi}{\cos\varphi_1}\left((1+c_2)(1+c_3) - c_1(\tau+c_4)\right)
\end{split}
\eeq
which implies the last inequality (\ref{detTF}).
\end{proof}

It follows from the above lemma that  the differential $D_xT_\bF:\cT_x M\to \cT_{x_1}M$  at any  point $ x =(r, \varphi)\in M$ is the $2\times 2$ matrix:
\beq\label{DTf}
DT_{\bF}(x)=-\frac{1}{\cos\varphi_1}\left(
  \begin{array}{cc}
	\tau\K +\cos\varphi+a_1& \tau+a_2 \\
	 \K( r_1)(\tau\K+\cos\varphi)+\K\cos\varphi_1+b_1 & \tau\K( r_1)+\cos\varphi_1+b_2 \\
  \end{array}
\right)
\eeq
where $ x_1=T_\bF(x)=( r_1, \varphi_1)$.

 Furthermore  it was shown in [Ch01] that the  map $T_{\bF}$ has two families of cones
$\bar \cC^u( x )$ (unstable) and $\bar \cC^s( x )$ (stable) in the tangent spaces
${\cT}_{ x } M$, for all ${ x }\in M$. More precisely, the unstable cone $\bar\cC^u(x)$ contains all tangent vectors based at ${ x }$ whose images  generate dispersing wave fronts:
\begin{equation}
\label{eq:unstable cone}
\bar\cC^u(x)=\{(dr, d\varphi)\in \cT_{ x } M:\, B_0^{-1}\leq d\varphi/dr\leq	 B_0\} .
\end{equation}
The unstable cone $\bar\cC^u(x)$ is strictly invariant under $DT_F$.
Similarly the stable cone
$$\bar\cC^s(x)=\{(dr, d\varphi)\in \cT_{ x } M:\, -B_0^{-1}\geq d\varphi/dr\geq	 -B_0\}$$
is strictly invariant under $DT_F^{-1}$.
Here $B_0=B_0(\eps_1, \tau_*,\cK_*)>1$ is a uniform constant.
Indeed, there exists a uniform constant $C>0$ such that we can choose
$B_0 = \cK^{-1}_* + 2\tau_*^{-1} + C\ve_1$ for all $\ve_1$
sufficiently small.

Let $dx = (dr, d\vf) \in \mathcal{T}_xM$.
Following \cite[Section 5.10]{chernov book}, we define the adapted norm
$\| \cdot \|_*$ by
\beq
\label{inducenorm}
\| dx \|_* = \frac{\K(x) + |\mathcal{V}|}{\sqrt{1 + \mathcal{V}^2}} \| dx \|, \; \; \;
\forall dx \in C^s(x) \cup C^u(x) ,
\eeq
where $\| dx \| = \sqrt{dr^2 + d\vf^2}$ is the Euclidean norm.
Since the slopes of vectors in $C^s(x)$ and $C^u(x)$ are bounded away from $\pm \infty$,
we may extend $\| \cdot \|_*$ to all of $\mathbb{R}^2$ in such a way
that $\| \cdot \|_*$ is uniformly equivalent
to $\| \cdot \|$.   It is straightforward to check that for $dx \in C^u(x)$,
\beq
\label{eq:Fexp}
\frac{\| dx_1\|_*}{\| dx \|_*} \ge \hat\Lambda:=1 + \K_{\min} \tau_{\min}/2.
\eeq
Finally, a simple calculation using \eqref{DTf} shows that there exists a constant
$B_1= B_1(\K_*, \tau_{\min}, \tau_{\max}) > 0$ such that
\beq
\label{eq:cos exp}
\frac{B_1^{-1}}{\cos \vf(x_1)} \le \frac{\| dx_1 \|}{\| dx \|}  \le \frac{B_1}{\cos \vf(x_1)},
\; \; \; \mbox{for all } dx \in C^u(x) .
\eeq
Uniform expansion in $C^s(x)$ under $DT^{-1}(x)$ follows similarly.
(See also \cite[Sect. 3]{Ch01}.)


\subsection{Hyperbolicity of the perturbed map $T_{\bF,\bG}$}
\label{hyp tfg}

We are now ready to verify conditions {\bf (H1)}-{\bf (H5)} for the map $T_{\bF, \bG}$.
We do this fixing $\bF$, $\bG$ satisfying assumptions {\bf (A1)}-{\bf (A4)}
with $|\bF|_{\C^1}, |\bG|_{\C^1} \le \ve$ for some $\ve \le \ve_1$.  We then
compare $T = T_{\bF, \bG}$ with the related map $T_{\bF} = T_{\bF, \mathbf{0}}$.

Since  $\mathbf{G}$ preserves tangential collisions,
the discontinuity set of $T$ is the same as that of $T_{\bF}$, which  comprises the
preimage of $\cS_0:=\{\varphi=\pm \pi/2\}$.   Similarly, the singularity sets
of $T^{-1}$ and $T_\bF^{-1}$ are the same due to {\bf (A4)}.
But the singular sets for higher iterates are not the same.
Let $\cS_{\pm n}^{T}= \cup_{i=0}^n T ^{\mp i}\cS_{0,H}$
with $n\in \mathbb{N}$.
Then $T^{\pm n}$ is smooth on $M\setminus \cS_{\pm n}^{T}$.

For any phase point $ x =(r,\varphi)\in M$, let $T x =(\bar r_1,\bar \varphi_1)$
and $T_{\bF}  x =(r_1, \varphi_1)$. According to {\bf (A3)} and {\bf (A4) } and since we
are on a fixed integral surface, we may
express $\bG$ in local coordinates via two smooth functions
$g^1$ and $g^2$ such that $g^i(r,\pm \pi/2)=0$, $i=1,2$, and
\beq
\label{rphi}
\bar r_1= r_1+g^1( r_1, \varphi_1)\,\,\,\,\,\text{ and }\,\,\,\,\,\,\bar \varphi_1=\varphi_1+g^2( r_1, \varphi_1)
\eeq
where $g^i$ is a  $C^2$ function with $C^1$ norm uniformly bounded from above by $c_g\eps$, for some uniform constant $c_g>0$.

According to  (\ref{rphi}), the differential of $T$ satisfies
\beq
\label{drphi}
d\bar r_1=\left(1+g^1_{1}( r_1, \varphi_1)\right) dr_1+g_{2}^1( r_1, \varphi_1) d\varphi_1\,\,\,\,\,\text{ and }\,\,\,\,\,\,d\bar \varphi_1=g^2_1( r_1, \varphi_1) dr_1+\left(1+g_2^2( r_1, \varphi_1)\right)d\varphi_1
\eeq
where $g^i_1(r_1,\varphi_1)=\partial g^i/\partial r_1$ and $g^i_2(r_1,\varphi_1)=\partial g^i/\partial \varphi_1$. This  implies
\begin{align}
\label{DTepsg}
D T(x)&= \left(
  \begin{array}{cc}
  1+g^1_1( r_1, \varphi_1) & g^1_2( r_1, \varphi_1) \\
 g^2_1( r_1, \varphi_1) & 1+g^2_2( r_1, \varphi_1) \\
  \end{array}
  \right)D T_{\bF}(x)
  \end{align}
Note that $T$ is not  a $\C^1$ perturbation of $T_\bF$ around the boundary of $M$.
Furthermore, $T$ no longer preserves $\mu_{\bF}$, the SRB measure for $T_\bF$.
However, it follows from \eqref{detTF} and \eqref{DTepsg} that
\begin{align}
\label{differential}
|\det DT(x)|\leq \frac{\cos\varphi(x)}{\cos\bar \varphi_1(x)}
\frac{\cos \bar \varphi_1(x)}{\cos \varphi_1(x)}(1+ C\eps)
\le  \frac{\cos\varphi(x)}{\cos\bar \varphi_1(x)} (1 + C_1 \ve)
\end{align}
since by \eqref{rphi},
\beq
\label{eq:equiv cos}
\frac{\cos \bar \varphi_1(x)}{\cos \varphi_1(x)} =
\frac{\cos (\varphi_1(x) + g^2(x_1))}{\cos \varphi_1(x)} \le (1 + C' \ve)
\eeq
since $g^2(r, \pm \pi/2) = 0$ and $|\nabla g^2| \le C \ve$.
Clearly this implies condition {\bf (H5)}.

The next proposition shows that although the perturbed maps do not have the same families of stable/unstable manifolds, they do share common families of stable and unstable cones.
\begin{proposition}\label{cones} There exist two families of cones
$C^u(x)$ (unstable) and $C^s(x)$ (stable) in the tangent spaces
${\cT}_{ x} M$ and $\Lambda>1$, such that for all ${ x}\in M$:
\begin{itemize}
\item[(1)] $D T  (C^u(x))\subset C^u(T  x)$ and $D T
	(C^s (x))\supset C^s( T  x)$ whenever $D T  $ exists.
	\item[(2)] These
	families of cones are	continuous on $M$
 and the angle between $C^u(x)$ and $C^s(x)$ is uniformly
 bounded away from
zero.
\item[(3)]	$\|D_{ x} T (v)\|_*\geq \Lambda \|v\|_*, \forall
	v\in C^u(x) \quad\text{and}\quad   \|D_{ x}T ^{-1}(v)\|_*\geq
	\Lambda \|v\|_*, \forall v\in C^s(x)$.
 \end{itemize}
\end{proposition}
\begin{proof}
For $ x\in M$ and any unit vector $d x\in \cT_{ x}M$, let $d{ x}_1=D_{ x}T_\bF  d{ x}$.
Then by (\ref{drphi})	
the slope $\bar\V_1$ of the vector
$d{ \bar x}_1$ at $\bar x_1:=T { x}=(\bar r_1, \bar \varphi_1)$ satisfies
\begin{align}\label{cVeps}
\bar\cV_1&=\frac{g^2_1+(1+g^2_2)\cV_1}{1+g^1_1+g^1_2 \cV_1}=\cV_1+\cO(\eps)
\end{align}
So the cone $\bar C^u(x)$ from \eqref{eq:unstable cone}
may not be invariant under $DT(x)$. Accordingly, we define a slightly bigger cone,
$$C^u(x)=\{(dr, d\varphi)\in \cT_{x} M:\, B_0^{-1}(1-c_1\eps_1))\leq d\varphi/dr\leq B_0(1+c_2\eps_1)$$
for some constants $c_1, c_2>0$, and  we use assumption (\textbf{A2}) to ensure that $c_i\eps_1<1/2$, $i=1,2$. By (\ref{DTf}), $DT_{\bF}$ maps the first and third quadrants strictly
inside themselves and shrinks any cones larger than the unstable cones.
More precisely, let $V$ be a unit vector on the upper boundary of $C^u(x)$,
with slope $\cV=B_0(1+c_2\eps_1)$.  Then by (\ref{DTf}) the slope of $DT_{\bF} V$ satisfies
$\cV_1=\frac{C+D\cV}{A+B\cV} $,
where we denote $$\left(\begin{array}{cc}A & B \\C & D \end{array}\right)	
=\left(\begin{array}{cc}\tau\K +\cos\varphi+a_1& \tau+a_2 \\
	 \K( r_1)(\tau\K+\cos\varphi)+\K\cos\varphi_1+b_1 & \tau\K( r_1)+\cos\varphi_1+b_2\end{array}\right) . $$	
It follows from the invariance of $\bar\cC^u$ that  $\frac{C+D B_0}{A+B B_0}<B_0$. One can easily check that  	 $$\cV_1=\frac{C+D B_0(1+c_2\eps_1)}{A+B B_0(1+c_2\eps_1)}<\cV=B_0(1+c_2\eps_1)$$ Similarly we can check the lower boundary of the cone is also mapped inside the cone $C^u$.
	  Thus $C^u$ is invariant under $DT$.

Similarly we define the stable cone	 $C^s(x)$ as
$$C^s(x)=\{(dr, d\varphi)\in \cT_{ x} M:\, -B_0^{-1}(1-c_1\eps_1))\geq d\varphi/dr\geq -B_0(1+c_2\eps_1)\}.$$
Then one can check that the stable cone $\cC^{s}$ is strictly invariant under $DT^{-1} $ whenever $DT^{-1} $ exists for any $T\in \cF$.  From the definitions of $C^s(x)$ and $C^u(x)$,
it is clear that the angle between them is bounded away from 0 on $M$.  Thus items
(1) and (2) of the lemma are proved.

To prove (3), note that \eqref{inducenorm} implies,
$$
\frac{\|d \bar x_1\|_*}{\|d x\|_*}=\frac{\|d \bar x_1\|_*}{\|d x_1\|_*}\frac{\|d x_1\|_*}{\|dx\|_*}=\frac{\|d x_1\|_*}{\|dx\|_*} \frac{\cK(\bar r_1)+|d\bar \vf_1|}{\cK(r_1)+ |d\vf_1|}  .
$$
Using \eqref{rphi}, \eqref{drphi}, \eqref{eq:Fexp}
and the fact that $\cK(\cdot)$ is a $\C^1$ function on $M$,
we conclude that for $\epsilon_0=1$ small enough,
\beq
\label{reducenorm}
\frac{\|d \bar x_1\|_*}{\|d x\|_*}\geq \Lambda := 1 + \cK_{\min} \tau_{\min}/3 .
\eeq
Similarly, one can show property (3) for stable cones, which we will not repeat here.	 
\end{proof}

Near grazing collisions, we have also using \eqref{eq:cos exp} and \eqref{eq:equiv cos}
along with  \eqref{rphi} and \eqref{drphi},
\beq
\label{eq:cos tgf}
\frac{B_1^{-1}(1-C\ve_1)}{\cos \bar \vf_1} \le
\frac{\| d \bar x_1 \|}{\| dx \|} = \frac{\|d \bar x_1\|}{\|d x_1\|}\frac{\|d x_1\|}{\|dx\|}
\le \frac{B_1(1+ C \ve_1)}{\cos \bar \vf_1} ,
\eeq
which establishes \eqref{eq:expansion} in {\bf (H1)} since in the finite horizon case, there
are only finitely many singularity curves so we may take $n$ in that formula to be 1.

The last formula (\ref{eq:2 deriv}) in \textbf{(H1)} (again with $n=1$) follows directly
from differentiating
(\ref{DTepsg}) and using (\ref{DTf}) to recover this standard estimate for the unperturbed
billiard (see \cite{katok} or \cite[Sect. 9.9]{Ch01} for the classical result).
This finishes the verification of \textbf{(H1)}.


\subsection{Regularity of stable and unstable curves}
\label{stable curves}

It follows from Proposition \ref{cones} that we may define common families of stable and
unstable cones for all perturbations $T\in \cF_B(Q_0, \tau_*, \ve_1)$.  Recall the homogeneity strips $\bH_k$ defined
in Section~\ref{class of maps} and that a homogeneous curve in $M$ is a curve
that lies in a single homogeneity strip.
In this subsection we will show that	there is a class of $C^2$ smooth unstable homogeneous
curves $\widehat\W^u$ in $M$ which is invariant under any $T\in \cF $.
Furthermore these curves are regular in the sense that they have uniformly
bounded curvature and distortion bounds.  Similarly, there is an invariant
class of homogeneous stable curves, $\widehat \W^s$.

\subsubsection{Curvature bounds}

  The next lemma, proved in $T_{\bF}$ in  \cite{Ch01},
  states that the images of an unstable curve are essentially flattened under the map $T_\bF$.
\begin{lemma}\label{curvbdT0}  Let $W\subset M$ be a $C^2$-smooth unstable curve with equation $\varphi_0=\varphi_0(r_0)$ such that $T_\bF^iW$ is a homogeneous unstable curve
for each $0 \le i \le n$.   Then
$T_\bF^n W$ has equation $\varphi_{n}=\varphi_n(r_n)$ which satisfies:
\beq\label{curvt0}|\frac{d^2\varphi_n}{dr^2_n}|\leq C_1+ \theta^{3n}|\frac{d^2\varphi_0}{d r_0^2}|\leq C_2\eeq
where $C_i=C_i(Q)$, $i=1,2$ is a constant and $\theta\in (0,1)$.
Furthermore, for any regular unstable curve $W$, there exists $n_W\geq 1$, such that for any $n>n_W$, every smooth curve of $T^n W$ has uniformly bounded curvature.
\end{lemma}

One can obtain a similar bounded curvature property for the perturbed map $T $.

\begin{proposition}
\label{curvbd}(Curvature bounds)
 Let $W$ be any $C^2$ smooth unstable curve. Then there exists $n_W\geq 1$ and
 $C_b >0$ such that 	every smooth curve $W'\subset T ^n W$ with equation
 $\bar\varphi_n=\bar\varphi_n(\bar r _n)$	satisfies
	 \beq\label{slopebd}|d^2\bar\varphi_n/d\bar r _n^2|\leq C_b, \; \; \;
	 \mbox{for } n > n_W .
	 \eeq
\end{proposition}
 \begin{proof}	
 We fix any phase point $ \bar x_0:=x \in W$, denote $x_n = (r_n, \vf_n) =T_{\bF}^n x$ and
 $\bar x_n = (\bar r_n, \bar \vf_n) =T^n x$.
 According to \eqref{drphi},	the slope of the vector $DT \,d \bar x $ satisfies
\begin{align}
\label{dbvarphidbr1}
\frac{d\bar\varphi_1}{ d\bar r_1}=\frac{g^2_1+(1+g^2_2)\cV_1}{1+g^1_1+g^1_2 \cV_1}=\cV_1+\frac{g^2_1+g^2_2\cV_1-g^1_1\cV_1-g^1_2 \cV_1}{1+g^1_1+g^1_2 \cV_1},
\end{align}
where $\cV_1=d\varphi_1/dr_1$, $\bar \cV_1=d\bar\varphi_1/d\bar r_1$.
 We differentiate the above equality  with respect to $r_1$, using the fact that by
 (\ref{drphi}),
 $\displaystyle \frac{d\bar r_1}{dr_1}=1+g^1_1 +g^1_2\cV_1 $.
Now  use the same notation as in Lemma \ref{curvbdT0} to get for some
$C_0>0$ and $C_3>0$
\beq
\label{evodphi}
|\frac{d^2\bar\varphi_1}{d\bar r_1^2}|
\leq C_0+(1+C_3\eps_1)\theta^{3}|\frac{d^2\bar\varphi_0}{d\bar r ^2_0}|,
\eeq
since $d^2\bar \vf_0/ d\bar r_0^2 = d^2 \vf_0 / dr_0^2$.
By choosing $\eps_1$ small one can make
$(1+\eps_1 C_3)\theta^2<1$. Then we have for any $n\geq 1$,
$$|\frac{d^2\bar\varphi_n}{d\bar r_n^2}|
\leq \frac{C_0}{1-\theta}+\theta^{n}|\frac{d^2\bar\varphi_0}{d\bar r_0^2}|
$$
 Since $W$ is $\C^2$, there exists $C_1=C_1(W)>0$ such that
 $|\frac{d^2\bar\varphi_0}{d\bar r_0^2}|<C_1$.
 We fix a constant $C_b=C_b(Q)>0$ and define
 $$n_W=|\frac{\ln (C_b/C_1)}{\ln \theta}|.$$
 Then for any $n>n_W$, connected components of $T^n W$ have equation
 $\bar\varphi_n=\bar\varphi_n(\bar r_n)$ with second derivative bounded from above by $C_b$.
\end{proof}

We now fix the constant $C_b>0$, then define $\widehat \cW^u$ be the class of
all homogeneous unstable curves $W$ whose  curvature is uniformly bounded by $C_b$.
It follows from Propositions~\ref{cones} and \ref{curvbdT0} that the class
$\widehat\cW^u$ is invariant under any $T \in\cF$.
Any unstable curve $W\in \widehat\cW^u$ is called a regular unstable curve.
Similarly one defines $\widehat\cW^s$. This verifies condition \textbf{(H2)}.


\subsubsection{Distortion bounds}

In this section, we establish the distortion bounds for $T$ required by {\bf (H4)}.
For any stable curve $W \in \widehat \W^s$ and $ x \in W$, denote by
$J_W T_\bF( x )$ (resp. $J_W T ( x )$) the Jacobian of $T_\bF$
(resp. $T$) along $W$ at $ x \in W$.
It was shown in \cite{Ch01}
that there exists $C_1>0$, such that for any regular stable curve $W$ for which
$T_\bF W$ is also a regular stable curve,
\beq
\label{distT0}
|\ln J_W T_\bF( x )-\ln J_W T_\bF( y)|
\leq C_1 d_W( x , y)^{\frac{1}{3}}
\eeq
where $d_W( x , y)$ is the arclength between $ x $ and $ y$ along $W$.
We show that $T$ has the same properties on the set of all regular stable curves $\widehat \W^s$.
\begin{lemma}\label{distorbd}(Distortion bounds)
Let $T \in \F$ and $W \in \widehat \W^s$ be such that $T$ is smooth on $W$ and
$TW \in \widehat \W^s$.  There exists $C_J>0$ independent of $W$ and $F$
such that
$$|\ln J_W T ( x )-\ln J_W T ( y)|\leq  C_J d_W( x , y)^{\frac{1}{3}} .$$
\end{lemma}
\begin{proof}	Fix $T \in \F$ and $W \in \W^s$ for which $TW \in \widehat \W^s$.  This implies in particular that both $T$ and $T_\bF$ are smooth on $W$.
For any $x=(r,\varphi)\in W$, let $x_1:=T_\bF x=(r_1,\varphi_1)$
and $\bar x_1=T x=(\bar r_1,\bar \varphi_1)$.
Similarly, let $dx = (dr, d\vf) \in \mathcal{T}_xW$ be a unit vector and define
$dx_1 = DT_\bF(x) dx = (dr_1, d\vf_1)$ and $d\bar x_1 =  DT(x) dx = (d\bar r_1, d\bar \vf_1)$.
Then
\[
\frac{J_W T ( x )}{J_W T_{\bF}( x )} = \sqrt{\frac{1+\bar \cV^2_{1}}{1+\cV^2_{1}}}
\frac{|d\bar r_1|}{|d r_1|}
\]
where $\cV_{1} = d\vf_1/dr_1$ and
$\bar \cV_{1} = d\bar \vf_1/ d\bar r_1$.
Then
it follows from \eqref{drphi} that
\beq
\label{eq:jacs}
\ln J_W T ( x )=\ln J_W T_\bF( x )+ \frac{1}{2}\ln (1+\bar \cV_{1}^2)-
\frac{1}{2}\ln (1+\cV_{1}^2)+\ln |1+g^1_1 +g^1_2\cV_1| .
\eeq
By the smoothness of $W$ and the curvature bounds,
there exists $C>0$ such that for any $ x , y\in W$,
$$|\ln (1+\cV_{1}^2( x_1 ))-\ln (1+\cV_{1}^2( y_1))|\leq |\cV_{1}^2( x_1 )-\cV_{1}^2( y_1)|\leq C d_{T_\bF W}( x_1 , y_1) \le C' d_W(x,y),$$
where $y_1 = T_\bF y$, and similarly for $\bar \cV_1$.
 Since $\bG$ is $C^2$, the terms involving $g^1_1$ and $g^1_2$
    satisfy a Lipschitz bound as well.
Putting this together with \eqref{distT0} and \eqref{eq:jacs} proves the lemma.
\end{proof}

In general, for $W \in \widehat \W^s$ and $n\in \mathbb{N}$, suppose $T^n$ is smooth
on $W$ and that  $T^k W \in \widehat \W^s$, $0 \le k \le n$.  Define $T^k W = W_k$ and
for $x, y \in W$, let
$x _k=T^k x $ and $ y_k=T^k y$.  Then
\begin{equation}
\begin{split}
\label{eq:dist ext}
|\ln J_W T^n( x )-\ln J_W T^n ( y)| & \leq \sum_{k=0}^{n-1}|\ln J_{W_k}T ( x _k)-\ln J_{W_k}T ( y_k)| \\
& \leq C\sum_{k=0}^{n-1}d_{W_k}( x _k, y_k)^{1/3}
\le C d_W(x,y)^{1/3} \sum_{k=0}^\infty \Lambda^{k/3},
\end{split}
\end{equation}
due to \eqref{reducenorm}.  This completes the required estimate on $J_WT$.

Finally, we prove the required bounded distortion estimate for $J_\mu T$.
By \eqref{eq:full detTf} and \eqref{DTepsg}, we have
\beq
\label{eq:det formula}
\det DT(x)  = \frac{\cos \vf}{\cos \vf_1} \big( (1+c_2)(1+c_3) - c_1(\tau + c_4) \big)
\big( (1 + g^1_1)(1+ g^2_2) - g^1_2 g^2_1 \big) =: \frac{A(x)}{\cos \bar \vf_1},
\eeq
where $c_1, \ldots, c_4$ are defined by \eqref{III} and
we have replaced $\cos \vf_1$ with $\cos \bar \vf_1$ times a smooth function on
$M \setminus \cS_1^T$
due to \eqref{eq:equiv cos}.
Note that $A(x)$ is a smooth function of
its argument wherever $T$ is smooth and has bounded $C^1$ norm on $M \setminus \cS_1^T$.
It follows that
$J_\mu T$ is a smooth function on $M \setminus \cS_1^T$
whose $C^1$-norm is bounded between $1 \pm C \ve_1$ for some uniform constant
$C$ depending on the table (recall that
$d\mu = c \cos \vf \, dm$ is the smooth invariant measure for the unperturbed
billiard $T_{\mathbf{0},\mathbf{0}}$).  The required distortion estimates
\eqref{eq:distortion stable} and \eqref{eq:D u dist} for $J_\mu T$ follow using this smoothness
and the uniform hyperbolicity of $T$ as in \eqref{eq:dist ext}.  Indeed, \eqref{eq:dist ext}
holds with exponent 1 rather than $1/3$ for $J_\mu T$.  This completes
the verification of {\bf (H4)}.

Distortion bounds for $\det DT$ with exponent $1/3$ follow from
the above considerations in addition to recalling that $1/\cos \vf$
is of order $k^2$ in $\bH_k$, while the width of such a strip along a stable or unstable
curve is $k^{-3}$.
Similarly, one may prove absolute continuity of the holonomy map between unstable
leaves as in \cite{Ch01}, but we do not do that here since we do not need this fact.


\subsection{One step expansion}

Since we have established the expansion factors given by
\eqref{reducenorm} and \eqref{eq:cos tgf},
the one-step expansion condition \eqref{eq:step1} follows from an argument similar
to the unperturbed case (see \cite[Lemma 5.56]{chernov book}) and fixes the choice of
$k_0 \in \N$, the minimum index of the homogeneity strips.
We will not reprove that lemma here. Instead, we focus on the second part of
{\bf (H3)}, given by \eqref{eq:weakened step1}.

Fix $\delta_0 > 0$ and $k_0$ satisfying \eqref{eq:one step contract} and define $\W^s$
accordingly.  For $W \in \W^s$, let $V_i$ denote the maximal homogeneous
connected components of $T^{-1}W$.

\begin{lemma}
\label{lem:sigma}
For any $\varsigma>1/2$,  there exists $C=C(\delta_0, \varsigma,\eps_0)>0$ such that for any
$W \in \W^s$, any $T\in \cF_B(Q_0, \tau_*, \ve_1)$,
\beq\label{step2}
 \sum_i \frac{|TV_i|^{\varsigma}}{|V_i|^{\varsigma}} < C	.
\eeq
\end{lemma}

\begin{proof}
According to the structure of singular curves,  a stable curve of length $\le \delta_0$ can be cut by at most $N \le \tau_{\max}/\tau_{\min}$
singularity curves in $\Si_{-1}^{T}$ (see \cite[\S 5.10]{chernov book}).  For each $s \in \Si_{-1}^{T}$
intersecting $W$, $W$ is cut further
by images of the boundaries of homogeneity strips $S^H_k$, $k \geq k_0$.	 For one such $s$, we relabel the components $V_i$ of $T^{-1}W$ on which $T$ is smooth by $V_k$, $k$ corresponding to
the homogeneity strip $\bH_k$ containing $V_k$.  By \eqref{eq:cos tgf}, there exists
$c_1=c_1(\ve_1)>0$ such that on $TV_k$, the expansion	under $T^{-1}$ is
$\geq c_1 k^2$.	 So for all $\varsigma>1/2$,
\begin{equation}
\label{eq:one sum}
\sum_{k\geq k_0}\frac{|TV_k|^{\varsigma}}{|V_k|^{\varsigma}}
\leq c_1\sum_{k\geq k_0} \frac{1}{k^{2\varsigma}}\leq \frac{c_1}{k_0^{2\varsigma-1}} .
\end{equation}
An upper bound for
\eqref{step2} in this case is given by $N$ times the bound in \eqref{eq:one sum}.
\end{proof}

This completes the verification of {\bf (H1)}-{\bf (H5)} and completes the proof of
Theorem~\ref{thm:C1}.


\subsection{Smallness of the perturbation}

In this section, we check that conditions {\bf (C1)}-{\bf (C4)} are satisfied for $\ve_1$
sufficiently small.   We will then be able to apply Theorem~\ref{thm:C2} to any map
$T \in \F_B(Q_0, \tau_*, \ve_1)$.

We fix $\eps\in (0,\eps_1)$ and choose any $T:=T_{\bF, \bG}\in \cF_B(Q_0, \tau_*, \ve_1)$,
such that
$|\bF|_{\C^1}, |\bG|_{\C^1} \le \eps$.  By the triangle inequality, it suffices to
estimate $d_\F(T_0, T)$ where $T_0 = T_{\mathbf{0}, \mathbf{0}}$ is the unperturbed billiard
map.

Denote by $\Phi^t$ the flow corresponding to $T$ and by $\Phi_0^t$ the flow corresponding to
$T_0$.
Let $x \in M \setminus (S_{-1}^T\cup S_{-1}^{T_0})$.
By the facts summarized in Section~\ref{flow review}, $\Phi^t(x)$ and $\Phi_0^t(x)$ can
be no further than a uniform constant times $\ve t$ on the billiard table.
Thus since $T$ has finite horizon bounded by $\tau_{\max}$ and the scatterers have uniformly
bounded curvature,
$T(x)$ and $T_{\bF,\mathbf{0}}(x)$ can be no more than a constant times $\sqrt{\ve}$ apart
if they lie on the same scatterer.  By the smallness of $\bG$ and \eqref{rphi}, we have
$d_M(T_{\bF, \mathbf{0}}(x), T_{\bF, \bG}) < C\ve$ and thus by the triangle inequality,
$d_M(T(x), T_0(x)) < C_f \sqrt{\ve}$ for some uniform $C_f >0$ as long as they lie
on the same scatterer.  A similar bound holds for $T^{-1}x$ and $T_0^{-1}x$.

Let $\epsilon = C_f \ve^{1/3}$.
It then follows that for any
$x\notin N_{\epsilon}(S_{-1}^T\cup S_{-1}^{T_0})$, $d(T^{-1}(x), T^{-1}_0(x))< \epsilon$.
This is {\bf (C1)}.

To establish {\bf (C2)}, we use the fact that
$J_\mu T_0 \equiv 1$ while
\[
J_\mu T(x) = \big( (1+c_2)(1+c_3) - c_1(\tau + c_4) \big)
\big( (1 + g^1_1)(1+ g^2_2) - g^1_2 g^2_1 \big)
\]
by \eqref{eq:det formula}.  Since the functions here are all bounded by uniform constants
times $\ve$ and our horizon is bounded by $\tau_{\max}$, {\bf (C2)} is satisfied.

Next, we prove {\bf (C4)}.  Inverting \eqref{DTepsg} and \eqref{DTf} and using
\eqref{eq:det formula}, we have
\[
DT^{-1}(x) = \frac{-1}{A(T^{-1}x) \cos \vf(T^{-1}x)} \left(  \begin{array}{cc}
         B + b_2 & C - a_2 \\
         D - b_1 & E + a_1  \end{array} \right)
         \left(  \begin{array}{cc}
         1 + g^2_2 & - g^1_2 \\
         - g^2_1     & 1 + g^1_1 \end{array} \right) ,
\]
where $A$ is the smooth function from \eqref{eq:det formula} and
$B = \tau(T^{-1}x) K(x) + \cos \vf(x)$, $C = - \tau(T^{-1}x)$,
$$D = -\K(T^{-1}x)(\tau(T^{-1}x) \K(x) + \cos \vf(x)) - \K(x) \cos \vf(T^{-1}x),\,\,\,\,\,
\text{ and }\,\,\,\,\,E = \tau(T^{-1}x) \K(T^{-1}x) + \cos \vf(T^{-1}x)$$
match the corresponding entries of $DT_0^{-1}x$ with $T$ replaced by $T_0$.

We split the matrix product as
\[
\left( \left(  \begin{array}{cc}
         B  & C  \\
         D  & E   \end{array} \right)
         + \left( \begin{array}{cc}
         b_2 & - a_2 \\
         - b_1 & a_1 \end{array} \right) \right)
\left(  I +
         \left(  \begin{array}{cc}
         g^2_2 & - g^1_2 \\
         - g^2_1     & g^1_1 \end{array} \right) \right)
=: F + R ,
\]
where $F = \left(  \begin{array}{cc}
         B  & C  \\
         D  & E   \end{array} \right)$
and $R$ is a matrix whose entries are smooth functions, all bounded by
a uniform constant times $\ve$.  Now defining $F_0$ to be the matrix $F$
with $T_0$ replacing $T$, we write,
\beq
\label{eq:matrix}
\begin{split}
& \| DT^{-1}(x)  - DT_0^{-1}(x) \| = \Big\| \frac{F + R}{A(T^{-1}x) \cos \vf(T^{-1} x)}
- \frac{ 1}{\cos \vf(T_0^{-1}x)} F_0 \Big\|   \\
& \le
\frac{\| F - F_0 \|}{|A(T^{-1}x) \cos \vf(T^{-1}x)|}
+ \| F_0 \| \left| \frac{1}{A(T^{-1}x) \cos \vf(T^{-1}x)} - \frac{1}{\cos \vf(T_0^{-1}x)} \right|
+ \frac{\| R \|}{|A(T^{-1}x) \cos \vf(T^{-1}x)|} .
\end{split}
\eeq
Notice that if $x \notin N_\epsilon(\cS_{-1}^T \cup \cS_{-1}^{T_0})$, then due
to the uniform expansion given by \eqref{eq:cos tgf} and the uniform transversality of
the stable cone with $\cS_0$, we have $d_M(T^{-1}x, \cS_0) \ge C\sqrt{\epsilon}$,
for some uniform constant $C$.    Thus
$\cos \vf(T^{-1}x) \ge C' \sqrt{\epsilon}$ for some uniform constant $C' > 0$.
The same fact is true for $T_0^{-1}x$.

Using this, plus the fact that the entries of $F$ and $F_0$ are smooth functions of their arguments
with uniformly bounded $\C^1$ norms, we estimate the first term of \eqref{eq:matrix} by
\[
\frac{\| F - F_0 \|}{|A(T^{-1}x) \cos \vf(T^{-1}x)|}
\le C \epsilon^{-1/2} d_M(T^{-1}x, T_0^{-1}x) \le C' \epsilon^{-1/2} \ve^{1/2} = C' C_f  \epsilon
\]
since the $C^1$ norm of $A$ is bounded above and below
by $1\pm C\ve$ by \eqref{eq:det formula}.
Similarly, the third term of \eqref{eq:matrix} is bounded by $C \epsilon$.

Since $\| F_0 \|$ is uniformly bounded,
we split the middle term of \eqref{eq:matrix} into the sum of two terms,
\[
\left| \frac{1}{A(T^{-1}x) \cos \vf(T^{-1}x)} - \frac{1}{\cos \vf(T_0^{-1}x)} \right|
\le \frac{1}{\cos \vf(T^{-1}x)} \left| \frac{1}{A(T^{-1}x)} - 1 \right | +
\left| \frac{1}{\cos \vf(T^{-1}x)} - \frac{1}{\cos \vf(T_0^{-1}x)}  \right| .
\]
As noted earlier, the $\C^1$ norm of $A$ is bounded above and below
by $1 \pm C \ve$ so that the first difference above is bounded by
$C \epsilon^{-1/2} \ve \le C C_f \epsilon$.  The second difference is bounded
by $C \epsilon^{-1} d_M(T^{-1}x, T_0^{-1}x) \le C' \epsilon^{-1} \ve^{1/2} = C' C_f \epsilon^{1/2}$,
similar  to the estimate \eqref{eq:cos diff}.

Putting these estimates together in \eqref{eq:matrix}
proves {\bf (C4)} with $\epsilon = C \ve^{1/3}$.
Condition {\bf (C3)} follows similarly using the fact that $J_WT(x) = \| DT(x)v \|$ where
$v \in \mathcal{T}_xW$ is a unit vector.  The exponent of $\epsilon$ in {\bf (C3)} is better than
in {\bf (C4)} by a factor of $\epsilon^{1/2}$
since we must estimate $\left| \frac{\cos \vf(T^{-1}x)}{\cos \vf(T_0^{-1}x)} - 1 \right|$
in place of $\left| \frac{1}{\cos \vf(T^{-1}x)} - \frac{1}{\cos \vf(T_0^{-1}x)}  \right| $.


\end{document}